\documentclass[reqno]{amsart}

\usepackage{url}
\usepackage{amssymb,amsthm,amsfonts,amstext}
\usepackage{amsmath}
\usepackage{mathptmx}       
\usepackage{helvet}         
\usepackage{courier}        
\usepackage{graphicx}
\usepackage{enumerate}
\usepackage[active]{srcltx}
\usepackage{mathrsfs}
\usepackage{tikz}
\usepackage{graphicx}
\usepackage{epsfig}
\usetikzlibrary{shapes}
\usepackage{color}
\usepackage[ansinew]{inputenc} 
\newtheorem{theorem}{Theorem}[section]
\newtheorem{definition}{Definition}[section]
\newtheorem{lemma}[theorem]{Lemma}
\newtheorem{proposition}[theorem]{Proposition}
\newtheorem{corollary}[theorem]{Corollary}
\newtheorem{remark}[theorem]{Remark}
\newtheorem{conjecture}[theorem]{Conjecture}

\newcommand\bx{{\mathbf x}}

\newcommand\be{{\mathbf e}}

\numberwithin{equation}{section}

\newcommand{\mc}[1]{{\mathcal #1}}
\newcommand{\mf}[1]{{\mathfrak #1}}

\newcommand{\bb}[1]{{\mathbb #1}}

\newcommand\LL{{\mathbb L}}

\newcommand\PP{{\mathbb P}}

\newcommand\RR{{\mathbb R}}
\newcommand\TT{{\mathbb T}}
\newcommand\ZZ{{\mathbb Z}}
\newcommand\E{{\mathbb E}}
\renewcommand\P{{\mathbb P}}
\newcommand\B{{\mathbb B}}

\newcommand\ve{\varepsilon}

\newcommand{\<}{\langle}
\renewcommand{\>}{\rangle}
\DeclareMathOperator{\dg}{deg}


\date{}

\begin{document}

\title[Occupation times of long-range exclusion]{Occupation times of long-range exclusion and connections to KPZ class exponents}

\author{C\'edric Bernardin}

\address{\noindent  Universit\'e de Nice Sophia-Antipolis\\ Laboratoire J.A. Dieudonn\'e\\ UMR CNRS 7351\\ Parc Valrose\\ 06108 Nice cedex 02- France \newline e-mail: \rm \texttt{cbernard@unice.fr}}

\author{Patr\'icia Gon\c{c}alves}

\address{\noindent Departamento de Matem\' atica, PUC-RIO, Rua Marqu\^es de S\~ao Vicente, no. 225, 22453-900, Rio de Janeiro, Rj-Brazil and
CMAT, Centro de Matem\'atica da Universidade do Minho, Campus de Gualtar, 4710-057 Braga, Portugal.\newline e-mail: \rm \texttt{patg@math.uminho.pt and patricia@mat.puc-rio.br}}

\author{Sunder Sethuraman}
\address{\noindent Department of Mathematics, University of Arizona, Tucson, AZ 85721, USA.
\newline e-mail:  \rm \texttt{sethuram@math.arizona.edu}}

\begin{abstract}
With respect to a class of long-range exclusion processes on $\ZZ^d$, with single particle transition rates of order $|\cdot|^{-(d+\alpha)}$, starting under Bernoulli invariant measure $\nu_\rho$ with density $\rho$, we consider the fluctuation behavior of occupation times at a vertex and more general additive functionals.  Part of our motivation is to investigate the dependence on $\alpha$, $d$ and $\rho$ with respect to the variance of these functionals and associated scaling limits.
In the case the rates are symmetric, among other results, we find the scaling limits exhaust a range of fractional Brownian motions with Hurst parameter $H\in [1/2,3/4]$.

However, in the asymmetric case, we study the asymptotics of the variances, which when $d=1$ and $\rho=1/2$ points to a curious dichotomy between long-range strength parameters $0<\alpha\leq 3/2$ and $\alpha>3/2$.  In the former case, the order of the occupation time variance is the same as under the process with symmetrized transition rates, which are calculated exactly.  In the latter situation, we provide consistent lower and upper bounds and other motivations that this variance order is the same as under the asymmetric short-range model, which is connected to KPZ class scalings of the space-time bulk mass density fluctuations.

\end{abstract}

\maketitle
\thispagestyle{empty}


\section{Introduction}

Informally, the exclusion process is an interacting particle system consisting of a collection of continuous-time dependent random walks moving on the lattice $\ZZ^d$:  A particle at $x$ waits an exponential$(1)$ time and then chooses to displace to $x+y$ with translation-invariant probability  $p(y)$.  If, however, $x+y$ is already occupied, the jump is suppressed and the clock is reset.  The process $\eta_t = \{\eta_t(x): x\in \ZZ^d\} \in \{0,1\}^{\ZZ^d}$ for $t\geq 0$ is a Markov process which keeps track of the occupied locations on $\ZZ^d$.  These systems have been much investigated since the 1970's when they were introduced as models of queues, traffic, fluid flow etc.  In particular, the model has proved useful and fundamental in the context of statistical physics \cite{Li}, \cite{Liggett2}, \cite{Spohn}.

  The exclusion model has many invariant measures, being `mass-conservative' with no birth or death.  In fact, there is a one parameter family of Bernoulli product invariant measures $\nu_\rho$, indexed by the `mass density' $\rho\in [0,1]$ (cf. Chapter VIII in \cite{Li}). Here, under $\nu_\rho$, particles are placed at lattice points $x\in \ZZ^d$ independently with probability $\rho$.  Throughout the paper, we fix a density $\rho\in (0,1)$ and begin the process under $\nu_\rho$.

The study of the fluctuations of occupation times of a vertex, or a local region, or more generally that of additive functionals in exclusion particle systems on $\ZZ^d$, starting from an invariant measure $\nu_\rho$ has a long history going back to \cite{K} and \cite{KV}.  When the infinitesimal interactions are `finite-range', that is when $p$ is compactly supported, several interesting dependencies on the dimension $d$, the density $\rho$, and the type of underlying single particle transition probability $p=p(\cdot)$ have been found .  In particular, for the asymmetric exclusion model, when $\rho=1/2$, connections with `Kardar-Parisi-Zhang' (KPZ) class variance orders of the space-time bulk mass density of the process have been made (cf. Subsection \ref{KPZ} below).

The purpose of this article is to ask what happens if the system has `long-range' interactions, that is say when $p(\cdot)$ has a long tail, proportional to $|\cdot|^{-d+\alpha}$ for $\alpha>0$.  Such systems are of interest in models with anomalous diffusion, a subject of recent interest (cf. \cite{Havlin}, \cite{Abe} and references therein).  In the particle systems context, symmetric long-range exclusion processes have been studied with respect to tagged particles \cite{Jara_cpam}.    However, in the asymmetric context, there appears to be little work on long-range processes. We note the `long-range' systems considered in this article are {\it not} those systems, with the same name, where at rate $1$ a particle hops to the nearest empty location found by iterating a random walk kernel (cf. \cite{Andjel}).

  What are the variance orders and scaled centered limits of the occupation time at a vertex or more general additive functionals, and how do they relate to $d$, $\rho$, $\alpha$ and the structure of $p$?  In particular, one wonders under asymmetric long-range infinitesimal interactions if there are still connections with`KPZ' exponent orders, and if so how to interpret them.  Can one infer the notion of 'long-range KPZ' exponent orders, which to our knowledge have not before been considered?

To discuss these questions and to put our work in better context, we first develop connections with `second-class' particles and $H_{-1}$ norms in the setting of occupation times at the origin, and then discuss previous `finite-range' literature afterwards.

  Let $\eta_s(0)$ be the indicator of a particle at the origin at time $s$ with respect to the process, and let $\Gamma(t) = \int_0^t f(\eta_s)ds$ with $f(\eta) = \eta(0)-\rho$ be the centered occupation time up to time $t$.  Let $a^2_t = \E_\rho \big( [\Gamma (t)]^2\big)$ be the variance starting from $\nu_\rho$.

\subsection{Connection with a `second-class' particle}
\label{second-class}
 The variance may be computed from a standard argument.  By stationarity of $\nu_\rho$ and changing variables,
\begin{eqnarray*}
a^2_t & = & 2\int_0^t \int_0^s \E_\rho[f(\eta_u)f(\eta_0)]duds\\
&=&2t\int_0^t \big( 1- s/t\big) \E_\rho[f(\eta_s)f(\eta_0)]ds.
\end{eqnarray*}
Now, the covariance, or `two-point' function as it sometimes called, as $\rho = \P_\rho(\eta_s(0)=1)$ for $s\geq 0$, and by Bayes's formula,
\begin{eqnarray*}
\E_\rho [f(\eta_s)f(\eta_0)] & = & \E_\rho[\eta_s(0)\eta_0(0)] - \rho^2\\
&=& \rho\big\{\E_\rho[\eta_s(0)|\eta_0(0)=1] - \E_\rho[\eta_s(0)]\big\}\\
&=& \rho(1-\rho)\big\{\E_\rho[\eta_s(0)|\eta_0(0)=1] - \E_\rho[\eta_s(0)|\eta_0(0)=0]\big\}.
\end{eqnarray*}
From the basic coupling, which compares two exclusion systems starting from $\eta_0$ and $\eta'_0$, a configuration which `flips' the value at the origin, that is $\eta'_0(x) = \eta_0(x)$ for $x\neq 0$ and $\eta'_0(0)= 1-\eta_0(0)$,
we can track the location of the discrepancy $R_s$, initially at the origin, for times $s\geq 0$.  The dynamics of the discrepancy, or `second-class' particle, is that it moves from location $x$ to $x+y$ at time $s$ with rate $p(y)(1-\eta_s(x+y)) + p(-y)\eta_s(x+y)$.  The interpretation is that it jumps as any other particle in the system, corresponding to the part $p(y)(1-\eta_s(x+y))$; but, also it must move if one of the other particles jumps to its location, corresponding to the part $p(-y)\eta_s(x+y)$.  Hence,
$$\E_\rho[\eta_s(0)|\eta_0(0)=1] - \E_\rho[\eta_s(0)|\eta_0(0)=0] \ = \ \bar\P_\rho(R_s = 0)$$
where $\bar\P_\rho$ is the coupled measure.  See Section VIII.2 in \cite{Li} for more discussion on the basic coupling.

Putting these observations together, we have
$$a^2_t\ = \ 2t\int_0^t \big(1-s/t\big) \bar\P_\rho(R_s = 0) ds,$$
roughly $t$ times the expected occupation time of the second-class particle at the origin.

\subsection{Connection with `$H_{-1}$' norms}  Instead of dealing directly with $a^2_t$, one might consider the Laplace transform $L_\lambda = \int_0^\infty e^{-\lambda t} a^2_t dt$ and its behavior as $\lambda \downarrow 0$.  By a formal Tauberian ansatz, $t^{-1}L_{t^{-1}} \sim t^{-1}\int_0^t a^2_u du \sim a^2_t$.  Moreover, the object $L_\lambda$, after two integration by parts, may be written as
\begin{eqnarray*}
L_\lambda & = & \frac{2}{\lambda^2}\int_0^\infty e^{-\lambda t} \E_\rho[f(\eta_t)f(\eta_0)]dt.\\
&=& \frac{2}{\lambda^{2}}\E_\rho[f(\eta_0)u_\lambda(\eta_0)]
\end{eqnarray*}
where $u_\lambda(\eta) = \int_0^\infty e^{-\lambda t}T_tf(\eta) = (\lambda - {\mc L})^{-1}f(\eta)$ and $T_t$ and ${\mc L}$ are the process semigroup and generator respectively.  The term $\big\{\E_\rho[f(\eta) (\lambda - {\mc L})^{-1}f(\eta)]\big\}^{1/2}$ is well defined for $f\in \LL^2(\nu_\rho)$ and can be written in variational form in terms of $H_1$ and $H_{-1}$ (semi-)norms and the symmetric and anti-symmetric decomposition of ${\mc L} = {\mc S} + {\mc A}$, which may be leveraged in bounding $L_\lambda$.  Moreover, a useful test for when $a^2_t = O(t)$ is that the $H_{-1}$ norm $\|f\|_{-1}<\infty$.  See Subsection \ref{resolventnorms} for a more comprehensive treatment.

\subsection{Finite-range models:  Symmetric and mean-zero transtitions} When $p$ is symmetric, $p(\cdot) = p(-\cdot)$, the transition rates of the second-class particle from $x$ to $x+y$ reduce to $p(y)(1-\eta_s(x+y)) + p(-y)\eta_s(x+y) = p(y)$.  Hence, marginally, the second-class particle moves as a symmetric random walk.  In this case, $\bar\P_\rho(R_s=0)$ can be explicitly estimated.  When $p$ is finite-range, along similar lines, it was shown in \cite{K} that
$$a^2_t \ = \ \left\{\begin{array}{rl}
O(t^{3/2}) & \ {\rm in \ }d=1\\
O(t\log(t)) & \ {\rm in \ }d=2\\
O(t) & \ {\rm in \ }d\geq 3.\end{array}\right.
$$

Moreover, in the above scales, the functional CLT in the uniform topology was shown in \cite{K}, \cite{S}:
\begin{equation}
\label{BM_finiteCLT}
\frac{1}{a_N}\Gamma(Nt) \ \xrightarrow[N\rightarrow \infty] \ \left\{\begin{array}{rl}
\B_{3/4}(t) & \ {\rm in \ }d=1\\
\B(t) & \ {\rm in \ }d\geq 2.
\end{array}\right.\end{equation}
Here, $\B_{H}$ is fractional Brownian motion with Hurst parameter $H$ and $\B=\B_{1/2}$ is standard Brownian motion.

We remark similar claims on the Laplace transform $L_\lambda$ hold when $p$ is finite-range, asymmetric and mean-zero, $\sum xp(x)=0$ by different methods.  Also, corresponding CLT's and scaling limits have been shown \cite{gj_cpam}, \cite{S}, \cite{V}.

\subsection{Finite-range models:  Asymmetric transitions and KPZ exponents}
\label{KPZ}
  When $p$ is finite-range and has a drift, $m=\sum xp(x) \neq 0$, although the second-class particle $R_s$ is not a random walk, it has a mean drift of $(1-2\rho)m$ under $\bar \P_\rho$ (cf. \cite{Balazs} and references therein).  In analogy with random walks, the second-class particle should be transient exactly when $\rho\neq 1/2$.  Partly based on this intuition, it was proved for $\rho\neq 1/2$ in $d\geq 1$ that $a^2_t = O(t)$, and also the functional CLT
$N^{-1/2}\Gamma(Nt) \Rightarrow \B(t)$ (cf. \cite{B}, \cite{S.S.}, \cite{S}).

However, now fix $\rho=1/2$ for the remainder of the subsection.  This case interestingly connects with `Kardar-Parisi-Zhang' (KPZ) behavior and exponents of driven diffusive systems.       In this situation, the process macroscopic characteristic speed $(1-2\rho)\sum xp(x)$ vanishes.   By the same sort of calculation presented above in Subsection \ref{second-class}, the variance of the second-class particle can be written in terms of the `diffusivity' of the system:
$$\rho(1-\rho)\bar\E_\rho\big|R_t\big|^2 \ = \ \sum |x|^2 \E_\rho\big[(\eta_t(x) - \rho)(\eta_0(0) - \rho)\big] \ =: \ D(t)$$
 which in $d=1$ is related to the variance of the `height' function for an associated interface which is in the KPZ class (cf. Chapter 5 in \cite{Spohn} and \cite{Quastel_Valko2} for definition of the height function and more discussion).

 In \cite{vbkspohn}, it was formulated that
 $$D(t) \ = \ \left\{\begin{array}{rl}
 O(t^{4/3}) & \ {\rm in \ }d=1\\
 O(t(\log(t))^{2/3}) & \ {\rm in \ }d=2\\
 O(t) & \ {\rm in \ }d\geq 3.\end{array}\right.$$
 This has been proved, in Tauberian form, by various techniques and discussed in more detail in \cite{Balazs_Seppalainen},  \cite{CLO}, \cite{LQSY}, \cite{Quastel_Valko}, \cite{Quastel_Valko2}, and \cite{Yau}.

 Then, allowing a Gaussian ansatz, $\bar\P_\rho(R_t = 0)$ should decay as $O(t^{-2/3})$ in $d=1$, $O(t^{-1}(\log t)^{1/3})$ in $d=2$, and $O(t^{-d/2})$ in $d\geq 3$.  Although these local limit type estimates have not been shown, they would imply that the occupation time variance should satisfy the same estimates as for $D(t)$ above.  However, in $d\geq 3$, when $\rho=1/2$, the conclusion $a^2_t = O(t)$ is known \cite{S}, \cite{SVY}.

Although the conjecture in $d\leq 2$ for the order of $a^2_t$ has not been substantiated, the following $H_{-1}$ estimates have been found in \cite{B} and \cite{S3}:
As $\lambda \downarrow 0$,
\begin{eqnarray}
\label{intro_help_eq}
C\lambda^{-9/4} &\leq & L_\lambda \ \leq \ C^{-1}\lambda^{-5/2} \ \ {\rm \ in \ }d=1\nonumber\\
C\lambda^{-2}\log|\log(\lambda)| &\leq & L_\lambda \ \leq \ C^{-1}\lambda^{-2}|\log(\lambda)| \ \ {\rm \ d=2}
\end{eqnarray}
with an improvement in the second line lower bound of $C\lambda^{-2}|\log(\lambda)|^{1/2}$ when the $p$-drift, $\sum xp(x)$, lies on a coordinate axis.
  These Tauberian bounds formally imply that
\begin{eqnarray*}
Ct^{5/4} & \leq & a^2_t \ \leq \ C^{-1}t^{3/2} \  \ {\rm \ in  \ }d=1\\
Ct\log(\log(t)) & \leq & a^2_t \ \leq \ C^{-1}t\log(t) \ \ {\rm \ in \ }d=2.
\end{eqnarray*}

\subsection{Finite-range models:  General additive functionals and $H_{-1}$ norms}
Besides the occupation function, one can consider the additive functional $\Gamma_f(t) = \int_0^t f(\eta_s)ds$ for a general class of `local' mean-zero functions, $\E_\rho[f]=0$.  That is, by `local', we mean  $f$ is compactly supported:  $f(\eta)$ depends only on the variables $\eta(x)$ for $x\in \Lambda\subset \ZZ^d$ and $\Lambda$ is a finite set.  Let $\sigma^2_t(f) = \E_\rho\big(\Gamma_f(t)\big)^2$.

One may ask for which functions $f$ is $\sigma^2_t(f) = O(t)$, that is the variance is of `diffusive' order.  When $p$ is finite-range, there is a dimension dependent characterization of such $f$'s depending on the `degree' or `smoothness' of the functions (cf.  Proposition \ref{short_range_H_{-1}}).  In particular, for the symmetric process, we have seen $f(\eta) = \eta(0)-\rho$ in dimensions $d\leq 2$ is not smooth enough.

When $\sigma^2_t(f)$ is not `diffusive', divergence orders have been found for symmetric and mean-zero processes (cf. Proposition \ref{exceptional_mz_FR}) and bounds for the asymmetric model (cf. Proposition \ref{exceptional_drift_FR}).

Functional CLT's in diffusive scale, converging to Brownian motion, for $\Gamma_f(t)$ when $\sigma^2_t(f) = O(t)$ have been shown (cf. \cite{Q.J.S.}, \cite{KV}, \cite{S.X.}, and \cite{S3} and references therein for statements and more discussion).  When $p$ is mean-zero and $f$ is a degree $1$ function (such as the occupation function $f(\eta) = \eta(0)-\rho$), in $d=1$, a functional CLT in anomalous scale has been proved \cite{gj_cpam}.  Otherwise, characterizing the fluctuations of $\Gamma_f(t)$ is open.

\subsection{Long-range transitions and main results}  We will take $p$ to be `long-range' if its symmetrization $2^{-1}(p(x)+p(-x))$ is proportional to $|x|^{-d+\alpha}$ for $\alpha>0$.  This natural choice introduces the parameter $\alpha$ which controls the order of moments allowed.  We also consider several types of asymmetries, both `short' and `long', which are detailed in the next section.

When $\alpha>2$, $p$ has two moments and one suspects the asymptotics of the occupation time $\Gamma(\cdot)$ behaves as if $p$ were finite-range (cf. Theorem \ref{short_long_thm}).
Also, when $0<\alpha<1$ or $d\geq 3$, the random walk generated by $p$ is transient \cite{Spitzer_book}, and the long time behavior of $\Gamma(\cdot)$ is diffusive (cf. part of Theorems \ref{th:adm-symmetric}, \ref{th:conv-bw}, \ref{th:adm-asymmetric}).

Our main interest is when $1\leq \alpha\leq 2$ and $d\leq 2$.  When $p$ is symmetric, one of our main results is to derive a fractional Brownian motion scaling limit in $d=1$ for $\Gamma(\cdot)$ in scale $a_t = O(t^{1-(2\alpha)^{-1}})$, corresponding to Hurst parameter $H= 1-(2\alpha)^{-1}$.  This microscopic derivation of a collection of fractional BM's, in a range of Hurst parameters, generalizes the $H=3/4$ limit when $p$ is finite-range. In $d\leq 2$, other additive functional variance divergence orders and CLTs are also found (cf. Theorems \ref{th:symd1}, \ref{th:symd1bis}, and \ref{th:conv-bw}).  We also observe that most of these results also hold for a class of long-range mean-zero processes.

However, when $p$ is asymmetric with a `drift'--an example is when $p(x)$ is proportional to ${\bf1}_{(x_i>0: 1\leq i\leq d)}|x|^{-d-\alpha} $--other new phenomena appear.  In particular, in $d=1$ when $\rho=1/2$, we observe a curious transition point at $\alpha = 3/2$.    When $\alpha\leq 3/2$, we show the variance $a^2_t$ is of the same Tauberian order as if $p$ were symmetric.  In particular, when $\alpha=3/2$, we prove $a^2_t = O(t^{4/3})$ in the Tauberian sense (cf. Theorem \ref{th:sd-asymmetric-1}).

However, as $\alpha$ increases, the process is less heavy-tailed and one feels less mixing, more volatile and more susceptible to `traffic jams'.  In fact, we propose for a large class of exclusion systems that $L_\lambda$ and $a^2_t$ should increase as $\alpha$ increases.  In support, we verify this intuition for symmetric and mean-zero type processs (cf. Theorem \ref{th:monotonicity}).

Moreover, we conjecture, from (1) this intuition, (2) the statement $a^2_t = O(t^{4/3})$, in the Tauberian sense, when $\alpha = 3/2$ and $\rho=1/2$, (3) the result $a^2_t$ is of the same Tauberian order as for finite-range processes when $\alpha>2$, and (4) the belief for $d=1$ finite-range processes with drift that also $a^2_t = O(t^{4/3})$, that we have $a^2_t = O(t^{4/3})$ in the Tauberian sense for all $\alpha\geq 3/2$ in $d=1$ (cf. Conjecture \ref{conjecture}).  We note superdiffusive lower and upper bounds, consistent with this conjecture, are given in Theorem \ref{th:sd-asymmetric-1}.

We remark the apparent dichotomy in the behavior of $a^2_t$ when variously $\alpha\leq 3/2$ and $\alpha>3/2$ in $d=1$ for $\rho=1/2$ suggests a novel extension of the scope of the KPZ class behavior to long-range models.  This topic and supporting results are discussed more in Subsections \ref{conjecture_subsection}, \ref{role_subsection}.

In dimension $d=2$ when $\rho=1/2$, analogously, we show for $\alpha\leq 2$ that $a^2_t$ is of the same Tauberian order as in the symmetric case (Theorems \ref{th:adm-asymmetric}, \ref{th:sd-asymmetric-2}).  Here, it seems, the KPZ class behavior does not extend below $\alpha\leq 2$.  As in the finite-range case, what is expected for $\alpha>2$ is that $a^2_t = O(t(\log t)^{2/3})$.

In addition, when $\rho\ne 1/2$, since the process characteristic speed drifts away from the origin, one expects $a^2_t = O(t)$.  This is indeed the case and stated in Theorem \ref{th:adm-asymmetric} for almost all values of $\alpha$ and $d\leq 2$.

We also consider the variance $\Gamma_f(t)$ for general local functions $f$, and find an $\alpha$, $\rho$, $d$-dependent characterization of when $a^2_t(f) = O(t)$ (Theorems \ref{th:adm-symmetric}, \ref{th:adm-asymmetric}), and also exceptional orders (Theorems \ref{th:sd-asymmetric-1}, \ref{th:sd-asymmetric-2}).  Corresponding functional CLT's  are also given for the symmetric model (cf. Theorems \ref{th:conv-bw}) and remarked upon for the asymmetric process (cf. Remark \ref{asym_rmk}).

The methods of the article make use of a combination of martingale CLT, `duality', and $H_{-1}$ norm variational formula arguments.  In particular, part of the arguments nontrivially generalize, to long-range models, the works \cite{K}, \cite{B}, \cite{S} in the finite-range setting.  The long-range nature of the jump rate introduces new questions, such as the long-range sector condition and monotonicity work in Sections \ref{subsec:equiv} and \ref{Appendix_B}, which may be of interest themselves.

\subsubsection{Notation and plan of the article}

The canonical basis of $\RR^d$ and coordinates of a vertex $x\in \RR^d$ are denoted by $e_i$ and $x_i$ for $1\leq i\leq d$ respectively.  The usual scalar product between $x$ and $y$ in $\RR^d$ is denoted by $x \cdot y$ and the corresponding norm by $| \cdot|$.

Define the relations ` $\approx$'', $\sim$', `$\preccurlyeq$', `$\succcurlyeq$' and note usual conventions `$O(\cdot)$' and `$o(\cdot)$' between sequences $a(s) \ge 0$ and $b(s)>0$:
\begin{itemize}
\item $a(s) \approx b(s)$ when both $0< \liminf_{s \to \infty} a(s)/b(s)$
 and $\limsup_{s \to \infty} a(s)/b(s) <\infty$,
\item
$a(s) \sim b(s)$ when $\lim_{s \to \infty} a(s)/b(s)$ exists and $0 <\lim_{s \to \infty} a(s)/b(s) <\infty$,
\item
$a(s)= O(b(s))$ when $\limsup_{s \to \infty} a(s)/b(s) < \infty$,
\item $a(s)=o(b(s))$ when $\limsup_{s \to \infty} a(s)/b(s)=0$,
\item $a(s)\preccurlyeq b(s)$ when $a(s) = O(b(s))$, and
\item $a(s)\succcurlyeq b(s)$ when $b(s) = O(a(s))$.
\end{itemize}
Sometimes, the parameter $s$ will denote the time $t$ which tends to infinity.  At other times, $s = \lambda$, a parameter we will send to $0$, and the relations above are defined accordingly.

In the next section, we more carefully define the model, and discuss the results.  In Section \ref{Tools}, we give notions of $H_{-1}$ norms, `duality' with respect to the (asymmetric) exclusion process, `free particle' approximations, and other basic estimates useful in the proofs.  In Section \ref{short_long_comparison_section}, finite and long-range $H_{-1}$ norm comparison results are proved.  In Sections \ref{symmetric_section} and \ref{asymmetric_section}, we prove the main results for symmetric and asymmetric long-range exclusion processes respectively.
In Section \ref{Appendix_B}, Theorem \ref{th:monotonicity} on monotonicity of Tauberian variances with respect to $\alpha$ is proved.  Finally, in Appendix \ref{Appendix_A}, some more technical computations are collected.

\section{Definitions and main results}
Let $\alpha>0$ and let $p(\cdot)$ be a transition function on $\ZZ^d$ such that for any $y \in \ZZ^d$,
\begin{equation*}
p(y)= \frac{\gamma(y)}{|y|^{d +\alpha}}, \quad \gamma(y)= \sum_{\sigma=\pm} \; c\sum_{i=1}^d b_i^\sigma (y) \, {\bf 1}_{\sigma (y \cdot e_i) >0}
\end{equation*}
and $p(0)=0$.
Here, $c$ is a normalization constant and
$(b_i^{\pm}(y))_{1\leq i\leq d, y\in \ZZ^d}$ are nonnegative real numbers, which are bounded $b_i^{\pm}(\cdot)\leq \bar{b}$, such that $(p(\cdot) + p(-\cdot))/2$ is irreducible.

  The symmetric and antisymmetric parts of $p$ are
 denoted respectively by $s$ and $a$ where $s(y)= (p(y)+p(-y))/2$ and $a(y)= (p(y)+p(-y))/2$. The mean of $p$, equal to the mean of $a$, is defined $m=\sum_{y \in \ZZ^d} y p(y)\in \RR^d$ if it converges.

We now distinguish several types of natural asymmetric long-range probabilities:
\begin{itemize}
\item[(LA)] (Long asymmetric range) There are constants $b^\sigma_i\geq 0$ such that $b^\sigma_i(y) \equiv b^\sigma_i$, $$\min_{1\leq i \leq d}b^+_i\wedge b^-_i>0 \ \ {\rm  and \ \ }
\sum_{i=1}^d |b_i^+ - b_i^-|  > 0.$$

\item[(SA)] (Short asymmetric range) There is an $R<\infty$ and $b_i>0$ such that $b^+_i(y) = b^-_i(y) = b_i$ for $|y|> R$, $\sum_{|y|\leq R}yp(y) \neq 0$.  Here, $a$ is finite range, but jumps of all large sizes are supported by $p$.

\item[(NNA)] (Nearest-neighbor asymmetry) A particular case of the short asymmetric range probability is when $R=1$ and the asymmetry is  nearest-neighbor.
\item[(MZA)] (Mean-zero asymmetry) Another case of the short asymmetric range probability is when
$\sum_{|y|\leq R}yp(y)=0$, but $p$ is not symmetric.
\end{itemize}

We will on occasion make comparisons with respect to the more studied `finite-range' jump probability, for which symmetric, mean-zero asymmetric and asymmetric versions can be analogously defined.
\begin{itemize}
\item[(FR)]
(Finite range) There is an $R<\infty$ such that for all $1\leq i\leq d$, $b^+_i(y)=b^-_i(y) = 0$ for $|y|>R$.  As before, to avoid sublattice periodicity, we assume the symmetric part $s$ is irreducible.

\item[(FR-NN)] (Nearest-neighbor) A case of the finite-range probability is when $R=1$.  Here, necessarily $s(e_i)>0$ for $1\leq i\leq d$.
\end{itemize}

Most of our focus, to make a choice, is on long asymmetric range model (LA), and for the remainder of the article $p$ denotes such a probability.  However, some comparisons with other types of probabilities are made in Subsection \ref{short_long_comparison_section}.  In the following, quantities with respect to the different types of probabilities will be denoted with corresponding superscripts; in this respect, $(S)$ signifies the jump probability is $s$.

The corresponding $d$-dimensional exclusion process is a Markov process $\{\eta_t\, ; \, t \ge 0\}$, with state space $\Omega=\{0,1\}^{\mathbb{Z}^d}$,
whose generator acts on local functions $f:\Omega \to \RR$ as
\[
\mathcal{L} f(\eta) = \sum_{x,y \in \bb{Z}^d} p(y)\eta(x)(1-\eta(x+y)) \nabla_{x,x+y} f(\eta),
\]
where
$\nabla_{x,x+y}f(\eta) = f(\eta^{x,x+y})-f(\eta)$ and
\[
\eta^{x,x+y}(z) =
\left\{
\begin{array}{rl}
\eta(x+y), & z=x\\
\eta(x), & z=x+y\\
\eta(z), & z\neq x,x+y.\\
\end{array}
\right.
\]
We will denote by $T_t$ the associated semigroup.

As mentioned in the introduction, for every $\rho \in [0,1]$, the
Bernoulli product measure $\nu_\rho$ with density $\rho$ is invariant for $\{ \eta_t \, ; \, t\ge 0\}$. Let ${\bb P}_{\rho}$ be the law of the process  $\{\eta_t\, ; \, t \ge 0\}$ starting from $\nu_{\rho}$.  We denote by ${\bb E}_\rho$, as it will be clear by context, the expectation with respect to both $\nu_\rho$ and $\P_\rho$.  We will also use the notation $\langle f,g\rangle_\rho := \E_\rho[fg]$.

One may compute that the $\LL^2(\nu_\rho)$ adjoint ${\mc L}^*$ itself is an exclusion generator with reversed jump probability $p^*(\cdot) = p(-\cdot)$.  When $p=s$, the $\LL^2(\nu_\rho)$ process generator ${\mc L}$ and semigroup $T_t$ are reversible. The construction and basic properties of this Markov process can be found in Chapter I, VIII in \cite{Li}; its extension to $\LL^2(\nu_\rho)$, with a core including local functions, follows from the development in Section IV.4 in \cite{Li}.

Recall the additive functional for this process
$$\Gamma_f(t)=\int_0^{t} f(\eta_s)ds,$$
where $f: \Omega \to \RR$ is a local function, and its variance ${\sigma_t}^2 (f)$ with respect to the stationary measure $\nu_{\rho}$ with density $\rho$.
We now define the `limiting variance' $\sigma^2 (f)$ by
\begin{equation*}
\sigma^2 (f) \ =\ \limsup_{t \to + \infty} t^{-1} \sigma_t^2 (f).
\end{equation*}
A local function $f$ such that $\sigma^2(f)<\infty$ or equivalently $\sigma_t^2 (f) \le C t$ for a constant $C>0$ independent of $t$, is said to be {\it{admissible}}.

Define the Laplace transform of $\sigma^2_t(f)$, $L_{f} (\lambda) = \int_0^{\infty} e^{-\lambda t} \sigma_t^2 (f) dt$.  We observe that if $f$ is admissible then $\lambda^{-2}L_f(\lambda)$ is uniformly bounded as $\lambda\downarrow 0$.

The behavior of the variance $\sigma_t^2 (f)$ and $L_f(\lambda)$ are much related to the degree of $f$.  Define $\mu_f(z) =\int f d\nu_{z}$ the mean of $f$ with respect to $\nu_z$.

\begin{definition}
Let $\dg(f)$ be the
{\em degree} of the local function $f$, with respect to $\nu_\rho$, that is the integer $i\geq 0$ such that $\mu_f^{(i)}(\rho) \neq 0$ and $\mu_f^{(j)}(\rho)=0$ for any $j<i$. If $\mu_f^{(j)}(\rho)=0$ for all $0\leq j \in \bb N_0$ we say $\dg(f)=\infty$.
\end{definition}

 For a finite subset $A\subset \ZZ^d$ with cardinality $|A|$, let $\Phi_A(\eta):=\prod_{i\in A} \big(\eta(i)-\rho\big)$.  Then, $\Phi_A$ is a degree $|A|$ function and $\mu_{\Phi_A}(z) = (z-\rho)^{|A|}$.  All local, mean-zero functions $f$, $\E_\rho[f]=0$, can be decomposed in terms of $\{\Phi_A: A\subset \ZZ^d\}$:  Since the occupation variables are at most $1$,
$$f(\eta) \ = \ \sum_{n\geq 1}\sum_{|A|=n} c(A)\Phi_A(\eta),$$
in terms of coefficients $c(A)$ where all sums are finite.  In particular, if $f$ is a degree $i$ local function then $\mu_f(z)$ is a degree $i$ polynomial.

Moreover, we may conclude,
\begin{equation}
\label{degree_decomp}
\begin{array}{rl}
 {\rm if \ } \dg(f)=1, &  {\rm then \ }\sum_{|A|=1}c(A) \neq 0\\
{\rm if \ } \dg(f) =2, & {\rm then \ }\sum_{|A|=2}c(A) = 0 \ {\rm and \ }\sum_{|A|=1}c(A)\neq 0\\
{\rm if \ }\dg(f)\geq 3, & {\rm then \ }\sum_{|A|=1}c(A) = \sum_{|A|=2}c(A)=0.
\end{array}
\end{equation}

It will be helpful, before stating our main long-range results in Subsections \ref{short_long_comparison_section} -- \ref{role_subsection}, to state precisely some of the work on finite-range systems.

\subsection{Previous work on (FR) models.}

Admissibility has been previously characterized for exclusion with finite range probabilities $p^{(FR)}$ in \cite{B}, \cite{S.X.}, \cite{S}.
\begin{proposition}
\label{short_range_H_{-1}}
Suppose $p^{(FR)}$ is mean-zero.  Then, a local function $f$ is admissible exactly when
$$\dg(f) \ \geq \ \left\{\begin{array}{rl}
3& \ {\rm in \ } d=1\\
2& \ {\rm in \ }d=2\\
1& \ {\rm in \ }d\geq 3.\end{array}\right.$$

But when $p^{(FR)}$ has a drift, $\sum xp^{(FR)}(x)\neq 0$, then $f$ is admissible exactly when
$$\dg(f) \ \geq \ \left\{\begin{array}{rl}
& 1\ {\rm if \ }\rho\neq 1/2 {\rm \ or \ } d\geq 3\\
& 2\ {\rm if \ }\rho = 1/2 \ {\rm and \ } d\leq 2. \end{array}\right.
$$
\end{proposition}

  In the exceptional cases, the following is known.   We remark when $p^{(FR)}$ is symmetric, $L_f(\lambda)$ and $\approx$ below can be replaced by $\sigma^2_t(f)$ and $\sim$ respectively; see \cite{Q.J.S.}, \cite{S.X.}, \cite{S} for more details and refinements.

  \begin{proposition}
\label{exceptional_mz_FR}

  Suppose $p^{(FR)}$ is mean-zero and $f$ is local.  Then, in $d=1$,
  $$L_f(\lambda)\ \approx \ \left\{\begin{array}{rl}
  \lambda^{-5/2} & \ {\rm if \ }\dg(f) =1\\
  \lambda^{-2}|\log \lambda| & \ {\rm if \ }\dg(f)=2\\
  \lambda^{-2}& \ {\rm if \ }\dg(f)\geq 3.\end{array}\right.$$

  In $d=2$,
  $$L_f(\lambda) \ \approx \ \left\{\begin{array}{rl}
  \lambda^{-2}|\log\lambda| & \ {\rm if \ }\dg(f)=1\\
  \lambda^2& \ {\rm if \ }\dg(f)\geq 2.
  \end{array}\right.$$

  In $d\geq 3$, $$L_f(\lambda)\ \approx\  \lambda^{-2}.$$
  \end{proposition}

  When $p^{(FR)}$ has a drift, $\rho=1/2$ and $\dg(f)=1$, the behavior of $\sigma^2_t(f)$ will be of the same conjectured orders $t^{4/3}$ in $d=1$ and $t(\log t)^{2/3}$ in $d=2$ with respect to the occupation time function $f_0(\eta) = \eta(0)-1/2$ discussed in the introduction.

On the other hand, the bounds on $L_\lambda = L_{f_0}(\lambda)$ given in \eqref{intro_help_eq} extend to degree $1$ functions \cite{B}, \cite{S3}.

\begin{proposition}
\label{exceptional_drift_FR}
Suppose $p^{(FR)}$ has a drift, $\sum xp^{(FR)}(x)\neq 0$, $\rho=1/2$, and $f$ is local and $\dg(f)=1$.  Then,
\begin{eqnarray*}
\lambda^{-9/4} & \leq & L_f(\lambda) \ \leq \ \lambda^{-5/2} \ \ {\rm \ in \ } d=1\\
\lambda^{-2}\log|\log(\lambda)| & \leq & L_f(\lambda) \ \leq \ \lambda^{-2}|\log(\lambda)| \ \ {\rm \ in \ }d=2.
\end{eqnarray*}
Also, in $d=2$, when in addition $\sum xp^{(FR)}(x)$ is on a coordinate axis, the lower bound can be replaced by $\lambda^{-2}|\log(\lambda)|^{1/2}$.
\end{proposition}

\subsection{Finite/Long-range and other comparisons}
\label{short_long_comparison_section}
 We now compare Tauberian variances $L_f$ and $L^{(FR)}$ when $\alpha>0$.  We remark the results of Theorem \ref{short_long_thm} holds also with respect to comparisons between $L^{(\cdot)}$, for all the types of jump probabilities mentioned before, and $L^{(FR)}$.

\begin{theorem}
\label{short_long_thm}
Let $f$ be a local function.  Then, for $\alpha>2$ and $d\geq 1$, when $\sum yp(y) = c\sum yp^{(FR)}(y)$ for a constant $c\neq 0$, we have
$$L_f(\lambda) \ \approx \ L_f^{(FR)}(\lambda).$$

\end{theorem}

We remark, in $d=1$, the `parallel' condition $\sum yp(y) = c\sum y p^{(FR)}(y)$ for a nonzero $c$ is the same as $\sum yp(y)=\sum yp^{(FR)}(y) = 0$ or both $\sum yp(y)$, $\sum yp^{(FR)}(y)\neq 0$.  The first display indicates that if $p$ has strictly more than $2$ moments then the associated long-range exclusion dynamics may be thought to have similar properties as when the jump probability is finite-range and parallel.  In particular, one may apply results for finite-range processes when $\alpha>2$.

Also, in long-range models with finite-range mean-zero asymmetries, we note Tauberian variances are comparable to their symmetric counterparts.
\begin{theorem}
\label{mza-symmetric thm}
Let $f$ be a local function.  Then, for $\alpha>0$ and $d\geq 1$, when $p = p^{(MZA)}$, we have
$$L^{(MZA)}_f(\lambda) \ \approx \ L_f^{(S)}(\lambda).$$
\end{theorem}

The proof of this statement is omitted as it follows the same proof of Lemma 4.4 in \cite{S} for finite-range mean-zero systems, as $a^{(MZA)}$ is the anti-symmetric part of a finite-range jump probability.

\subsection{Symmetric jumps} We now consider the symmetric process when $p(\cdot)=s(\cdot)$ for which we have a clear picture of the scaling limits of additive functionals.  We first characterize admissibility of local functions.

\begin{theorem}
\label{th:adm-symmetric}
Consider the symmetric long-range exclusion process in dimension $d$. We have the following characterization of admissibility.
\begin{itemize}
\item $d=1$: Every local function $f$ such that:
\subitem 1. $\dg(f) \ge 3$ is admissible,
\subitem 2. $\dg(f) =2$ is admissible  if $\alpha <2$,
\subitem  3. $\dg(f) =1$ is admissible if $\alpha<1$.
\item     $d=2$: Every local function $f$ such that:
\subitem 1. $\dg(f) \ge 2$ is admissible,
\subitem 2. $\dg(f)=1$ is admissible if $\alpha<2$.
\item $d \ge 3$: Every local function with $\dg(f)\geq 1$ is admissible.
\end{itemize}
\end{theorem}

\begin{remark}
\label{reduction_rmk}\rm
In terms of variance asymptotics, the following observation reduces the consideration of a general local degree $1$ function $f$ to that of the occupation time function $\eta(0) - \rho$.  Indeed, note that $g=f -\mu_f'(\rho)(\eta(0) - \rho)$ is at least a degree $2$ function.  When $d=1$, $\alpha<2$, $\sigma^2_t(g) = O(t)$ by Theorem \ref{th:adm-symmetric}.  Hence, if $\sigma^2_t(\eta(0)-\rho))$ is superdiffusive in growth, it is the dominant term with respect to the equation $f = g +\mu_f'(\rho)(\eta(0)-\rho)$.

Similarly, noting \eqref{degree_decomp},
a degree $k$ function $f$ can be written as $f = h + \frac{1}{k!}\mu_f{(k)}(\rho)\Phi_A$
where $|A|=k$ and $h$ is now at least a degree $k+1$ function.  Hence, one deduces $\sigma^2_t(f) \sim \sigma^2_t(\Phi_A)$ when $\sigma^2_t(\Phi_A)$ dominates $\sigma^2_t(h)$.
\end{remark}

Next, the following results give the variance behavior for exceptional functions $f$ in terms of dimension $d$.  As discussed earlier, when $\alpha>2$, the orders match those for the symmetric finite-range model (cf. Theorem \ref{short_long_thm}).
\begin{theorem}
\label{th:symd1}
Let $f$ be a local degree $1$ function.  It holds that
\begin{itemize}
\item In $d=1$
$$
\sigma_t^2 (f) \ \sim \ \left\{\begin{array}{rl} t, & \ {\rm if \ }\alpha <1 \\
t\log (t), & \ {\rm if \ } \alpha=1\\
t^{2-1/\alpha}, & \ {\rm if \ } 1<\alpha<2\\
t^{3/2}(\log (t))^{-1/2},& \ {\rm if \ } \alpha = 2
\\t^{3/2},& \ {\rm if \ }\alpha>2.\end{array}\right. $$

\item In d=2
$$
\sigma_t^2 (f) \ \sim \ \left\{\begin{array}{rl} t, & \ {\rm if \ }\alpha <2 \\
t\log (\log(t)), & \ {\rm if \ } \alpha=2
\\t\log(t),& \ {\rm if \ }\alpha>2.\end{array}\right. $$
\item In $d\geq3$,
$$
\sigma_t^2 (f) \ \sim t, \quad\textrm{for all} \quad \alpha. $$
\end{itemize}
\end{theorem}

\begin{theorem}
\label{th:symd1bis}
Let $d=1$, $\alpha \ge 2$ and let
$f$ be a local degree $2$ function.
Then, as $\lambda\downarrow 0$, we have
\begin{equation*}
L_f(\lambda)
\ \approx \    \lambda^{-2} |\log (\lambda)|.
\end{equation*}
\end{theorem}

\begin{remark}
\label{rem_after_symd1bis}\rm
When $\dg(f)=2$, we expect variance asymptotics $\sigma_t^2 (f) \sim t \log (t)$ if $\alpha>2$ and $\sigma_t^2 (f) \sim t\log(\log(t))$ if $\alpha=2$, which is not seen at the level of the Laplace transform $L_f(\lambda)$.  However, in the nearest-neighbor case, by computing the Green's function of a system of two interacting exclusion particles, which seems more difficult when jumps are not nearest-neighbor, these asymptotics are shown in \cite{S.X.}.
\end{remark}

The following convergence results hold. Recall $\B_H$ denotes fractional Brownian motion with Hurst exponent $H$, and $\B=\B_{1/2}$ is standard Brownian motion.

\begin{theorem}
\label{th:conv-bw}
\begin{enumerate}[i)]
\item If $f$ is an admissible function then we have weak convergence in the uniform topology:
$$\frac{1}{\sigma_{N} (f) } \Gamma_f (tN) \xrightarrow[N\rightarrow{\infty}]\,  \B(t).$$
\item If $f$ is a (non-admissible) function of degree $1$, we have the following weak convergences in the uniform topology

\begin{itemize}
\item In $d=1$
$$
\frac{1}{\sigma_{N} (f) } \Gamma_f (tN) \xrightarrow[N\rightarrow{\infty}]\, \ \left\{\begin{array}{rl} \B(t), & \ {\rm if \ } \alpha = 1\\
\B _{1-1/{2\alpha}}(t), &  \ {\rm if \ } 1<\alpha < 2\\
\B_{3/4} (t),&  \ {\rm if \ } \alpha \geq 2.\end{array}\right.
$$
\item In $d= 2$, for all $\alpha \ge 2$,
$$
\frac{1}{\sigma_{N} (f) } \Gamma_f (tN)  \xrightarrow[N\rightarrow{\infty}]\, \B(t).
$$
\end{itemize}
\item If $f$ is a (non-admissible) function of degree $2$, i.e. $\alpha \ge 2$ and $d=1$, then for any $t>0$, we have the one-time CLT, convergence in law
$$\frac{1}{\sigma_{N} (f) } \Gamma_f (tN) \xrightarrow[N\rightarrow{\infty}]\, {\rm \mathcal{N}(t)}$$
where ${\rm \mathcal{N}(t)}$ is a centered normal variable with variance $t$ as $N \to + \infty$.

\end{enumerate}
\end{theorem}

 The last part is weaker than the previous lines in Theorem \ref{th:conv-bw} as the exact asymptotics of $\sigma_{tN}(f)$ have not been found (cf. Remark \ref{rem_after_symd1bis}).

\subsubsection{Mean-zero (MZA) processes}
\rm We make a few remarks on (MZA) systems and note all statements in Theorems \ref{th:adm-symmetric} and \ref{th:symd1bis} hold for these processes.  In addition, statements in Theorem \ref{th:symd1}, interpreted in the Tauberian sense, that is with respect to the asymptotics of $L_f(\lambda) = \int_0^\infty e^{-\lambda t}\sigma^2_t(f)dt$, also hold for (MZA) processes.

Indeed, by the bound $\sigma^2_t(f) \leq 10t^{-1}L^{(S)}_f(t^{-1})$ in the $H_{-1}$ norm Lemmas \ref{seth bound} and \ref{sunder equivalence}, and admissibility for the symmetric process in Theorem \ref{th:adm-symmetric}, the same admissibility statements follow for (MZA) systems.  Also, the Tauberian variance statements for the symmetric process transfer to (MZA) processes by Theorem \ref{mza-symmetric thm}.

Finally, we remark, the statement in Part (i) Theorem \ref{th:conv-bw} also holds for (MZA)-systems, by the method in \cite{V} for finite-range mean-zero systems, since $a^{(MZA)}$ is the anti-symmetric part of a finite-range mean-zero jump probability.
Otherwise, the fluctuations have not been considered.

\subsection{Asymmetric jumps}

We now consider (LA) asymmetric processes with long-range probability $p$, which require more delicate considerations than in the symmetric situation.

However, we remark all results of this subsection also hold for long-range (SA) models with short-range asymmetries, with similar proofs.

\begin{theorem}
\label{th:adm-asymmetric}
Consider the asymmetric long-range exclusion process in dimension $d$. We have the following characterization of admissibility.
\begin{itemize}
\item $d=1$: Every local function $f$ such that:\\
 1. $\dg(f) \ge 3$ is admissible,\\
2. $\dg(f) = 2$ is admissible if $\alpha \neq 2$,\\
 3. $\dg(f)=1$ is admissible if $\rho \neq 1/2$ and $\alpha \neq 1, 2$ or if $\rho=1/2$ and $\alpha<1$.
\item     $d=2$: Every local function $f$ such that:\\
 1. $\dg(f) \ge 2$ is admissible,\\
2. $\dg(f)=1$ is admissible \\
\quad \quad \quad\quad if and only if $\rho \ne 1/2$ for all $\alpha$ or if $\rho=1/2$ and  $\alpha<2$.
\item $d \ge 3$: Every local function such that $\dg(f)\geq 1$ is admissible.\end{itemize}
\end{theorem}

\begin{remark}
\label{asym_rmk}
\rm

Cases left open, by our methods, are the boundary cases when $d=1$, $\alpha=1,2$, $\rho\neq{1/2}$ and $\dg (f)=1$
or when $d=1$, $\alpha=2$, $\dg(f)=2$ for which we conjecture such functions are admissible.  Moreover, we show later in Theorems \ref{th:sd-asymmetric-1} and \ref{th:sd-asymmetric-2} that functions not satisfying either the assumptions of Theorem \ref{th:adm-asymmetric} or the two cases above are not admissible.

 When all mean-zero local functions are admissible, that is when $\alpha<1$ or $d\geq 3$, the functional CLT display in Part (i) Theorem \ref{th:conv-bw} holds by the same argument as for Corollary 2.1 in \cite{S.S.}.  Otherwise, the fluctuation limits for $\Gamma_f$ have not been characterized. \end{remark}

The next results give upper and lower bounds on $L_f(\lambda)$ in exceptional non-admissible situations.  Formal estimates on $\sigma^2_t(f)$ can be recovered by the formal Tauberian relation $\sigma^2_t(f) \sim t^{-1}L_f(t^{-1})$.

\begin{theorem}
\label{th:sd-asymmetric-1}
Consider the asymmetric long-range exclusion process in dimension $d=1$ with $\alpha \ge 1$  and $\rho=1/2$. Let $f$ be a local function of degree one.
\begin{itemize}
\item When $\alpha=1$, as $\lambda \downarrow 0$,
$$ L_f(\lambda)
\ \sim\   \lambda^{-2}|\log (\lambda)|.$$
\item When $1 < \alpha \le 3/2$, as $\lambda \downarrow 0$,
$$   L_f(\lambda)
\ \sim\  \lambda^{1/\alpha-3}.$$
\item When $3/2 \le \alpha <2$, there exists a constant $C$ such that for all small $\lambda$,
$$ C^{-1} \lambda^{-1/2\alpha-2}
  \ \le\  L_f(\lambda)
\ \le\  C \lambda^{1/\alpha -3 }  $$
\item When $\alpha=2$,  there exists a constant $C$ such that for all small $\lambda$
$$C^{-1} \lambda^{-9/4} {|\log (\lambda) |^{1/4}} \ \le\  L_f(\lambda)
\ \le\  C \frac{\lambda^{-5/2} }{\sqrt{|\log (\lambda) }|}.$$
\item When $\alpha>2$, let $L^{(FR)}_f(\lambda)$ correspond to $p^{(FR)}$ with a drift, $\sum xp^{(FR)}(x) \neq 0$.  Then, by Theorem \ref{short_long_thm}, $L_f(\lambda) \approx L^{(FR)}_f(\lambda)$, and the bounds in Proposition \ref{exceptional_drift_FR} hold.
\end{itemize}
\end{theorem}

\begin{theorem}
\label{th:sd-asymmetric-2}
Consider the asymmetric long-range exclusion process in dimension $d=2$ with $\alpha \ge 2$  and $\rho=1/2$. Let $f$ be a local function of degree one.
\begin{itemize}
\item When $\alpha=2$, as $\lambda\downarrow 0$,
$$  L_f(\lambda)
\  \approx \ \lambda^{-2}\log (|\log (\lambda)|) .$$

\item When $\alpha>2$, let $L^{(FR)}_f(\lambda)$ correspond to $p^{(FR)}$ with a drift, $\sum xp^{(FR)}(x) \neq 0$.  Then, by Theorem \ref{short_long_thm}, $L_f(\lambda) \approx L^{(FR)}_f(\lambda)$, and the bounds in Proposition \ref{exceptional_drift_FR} hold.
\end{itemize}
\end{theorem}

\begin{remark}\rm
We note all upper bounds in Theorems \ref{th:sd-asymmetric-1} and \ref{th:sd-asymmetric-2} hold in the Abelian sense:  That is, $\sigma^2_t(f) \leq 10 t^{-1}L^{(S)}_f(t^{-1})$ by the $H_{-1}$ norms  Lemmas \ref{seth bound} and \ref{sunder equivalence}, and the variance bounds for the symmetric long-range process in Theorem \ref{th:symd1}.
\end{remark}

\subsection{A conjecture and partial monotonicity argument}
\label{conjecture_subsection}

As remarked in the Introduction, with respect to finite-range asymmetric exclusion processes, when $\rho = 1/2$ and $\sum yp^{(FR)}(y)\neq 0$, it is believed that the occupation time variance $\sigma^2_t(\eta(0)-1/2)) \approx t^{4/3}$ in $d=1$ and $\approx t(\log t)^{2/3}$ in $d=2$.
 Given Theorem \ref{short_long_thm}, these are the same orders conjectured for the variance, in the Tauberian sense, for the long-range asymmetric exclusion process when $\alpha>2$ in $d=1,2$.

Now, as $\alpha$ increases, the jump probability $p$ becomes less heavy-tailed.  Correspondingly, because of the exclusion dynamics, particles which are bunched together disperse slower and traffic jams are more likely to persist.  In particular, it is known that the occupation time at the origin has positively associated increments in time \cite{S}.  One feels consequently that the origin occupation time is more volatile as $\alpha$ grows, that is $\alpha\mapsto E_{\rho}\big[\int_0^t f(\eta_s) ds\big]^2= \sigma^2_t(f)$, and $\alpha\mapsto \int_0^\infty e^{-\lambda t}\E_{\rho}[fP_t f]dt = L_f(\lambda)$, in terms of their orders, are increasing functions of $\alpha$, where $f(\eta) = \eta(0)-\rho$.

Recall, also, at $\alpha = 3/2$, the order of the variance $\sigma^2_t(f)$, in both the symmetric and asymmetric cases, in the Tauberian sense, is $O(t^{4/3})$, the same order believed under asymmetric finite-range dynamics.  These comments form the basis of the following conjecture.

\begin{conjecture}
\label{conjecture}
  For $\alpha\geq 3/2$ and $\rho=1/2$, with respect to the exclusion dynamics with asymmetric long-range jump probability $p$, the Tauberian variance satisfies
  $$L_f(\lambda) \ = \ \int_0^\infty e^{-\lambda t}\sigma^2_t(f)dt \ \approx\  \left\{\begin{array}{rl}\lambda^{-7/3} & \ {\rm in \ }d=1\\
  \lambda^{-2}|\log(\lambda)|^{2/3} & \ {\rm in \ }d=2.\end{array}\right.$$
\end{conjecture}
Correspondingly, this type of approximation would formally imply $\sigma^2_t(f) \approx t^{4/3}$ in $d=1$ and $\sigma^2_t(f)\approx t(\log(t))^{2/3}$ in $d=2$.

In partial support of the conjecture, consider long-range models with finite range mean-zero asymmetries or symmetric transitions, where the jump probability $p^{(MZA)}$ or $s$ is replaced by jump rates
$\bar p^{(MZA)} = (b_1c)^{-1}p^{(MZA)}$ or $\bar s = (b_1c)^{-1} s$ respectively, which are sub-probabilities, effectively having removed the normalization. Let also $\bar L^{(MZA)}_f(\lambda)$ and $\bar L^{(S)}_f(\lambda)$ be the respective Tauberian variances, and $\bar a^{(MZA)}$ be the anti-symmetric part of $\bar p^{(MZA)}$.

As the process is speeded up by $(b_1c)^{-1}$, a calculation with the generator gives $\bar L^{(\cdot)}_f(\lambda) =    (b_1c)^3L^{(\cdot)}_f(b_1c\lambda)$.  In particular, the order, as $\lambda\downarrow 0$, of $\bar L^{(\cdot)}_f(\lambda)$ and $L^{(\cdot)}_f(\lambda)$ are the same.

\begin{theorem}
\label{th:monotonicity}
Consider long-range exclusion processes with jump rates $\bar p^{(MZA)}$ or $\bar s$.  Let $\varepsilon:=\max_{|y|\leq R} |{\bar a}^{(MZA)}(y)|$.  Then, for $1<\alpha_0<\beta_0$, there is $\varepsilon_0=\varepsilon_0(\alpha_0,\beta_0)>0$ such that for $0<\varepsilon<\varepsilon_0$ and $\lambda>0$, the map
$\alpha\in [\alpha_0,\beta_0]\mapsto \bar L_f^{(MZA)}(\lambda)$ or $\alpha\in [\alpha_0,\beta_0]\mapsto \bar L_f^{(S)}(\lambda)$
is non-decreasing.
\end{theorem}

\begin{remark}
\rm
We conjecture the same monotonicity statement holds for the Tauberian variance $\bar L^{(SA)}_f(\lambda)$ with respect to short asymmetric long range processes, with jump rate $\bar p^{(SA)} = c^{-1}p^{(SA)}$, when the drift $\sum x\bar p^{(SA)}(x)\neq 0$.  If such monotonicity could be verified, by Theorem \ref{th:sd-asymmetric-1}, there would be a constant $C$ such that $\bar L^{(SA)}_f(\lambda) \geq  C\lambda^{-7/3}$ when $\alpha\geq 3/2$ and $\rho=1/2$.  To show the corresponding upperbound on $\bar L^{(SA)}_f(\lambda)$, by Theorem \ref{short_long_thm}, it is sufficient to consider finite-range models.  In particular, to complete the Tauberian part of Conjecture \ref{conjecture}, it would be enough to show for  an $\alpha>2$ and $\rho=1/2$ that $\bar L^{(FR)}_f(\lambda) \leq C \lambda ^{-7/3}$, an estimate which is expected.

We show, however, for general short-range asymmetric (SA) processes with jump rate $\bar p^{(SA)}$, a formula for the derivative of the map $\alpha\mapsto \bar L_f^{(SA)}(\lambda)$ in the first part of the proof of Theorem \ref{th:monotonicity}.  Namely, the derivative equals $2\lambda^{-2}\langle f, v^\alpha_\lambda\rangle_\rho$, where $v^\alpha_\lambda$ solves a certain resolvent equation.
\end{remark}

\subsection{Role of $\alpha = 3/2$}
\label{role_subsection}
 Given Conjecture \ref{conjecture}, it seems the long-range parameter value $\alpha = 3/2$ is a change-point for the occupation time dynamics with respect to $d=1$ asymmetric exclusion with jump probability $p$ when $\rho = 1/2$.
On the one hand, for $\alpha\leq 3/2$, the occupation time variance behaves as that under the symmetric dynamics (cf. Theorems \ref{th:symd1}, \ref{th:sd-asymmetric-1}).  But, otherwise, it would seem, for $\alpha\geq 3/2$, the variance acts as that under an asymmetric finite-range (FA) model.

Technically, the symmetric part of the generator $L$ ``dominates'' the anti-symmetric part exactly when $0<\alpha<3/2$.  At $\alpha =3/2$, they are of the same order, and exact computations can be made.  A more physical intutition of the phenomena is yet to be understood.

That the occupation time variance orders are computed exactly, namely $1$ for $0<\alpha\leq 1$ and $2-1/\alpha$ for $1\leq \alpha \leq 3/2$ in $d=1$ (cf. Theorem \ref{th:sd-asymmetric-1}), in particular a power of $4/3$ for $\alpha = 3/2$, is one of the few exact calculations with respect to the fluctuations of asymmetric particle systems across process characteristics.

When $\rho\neq 1/2$, we conjecture the same scaling behavior, as in Theorem \ref{th:sd-asymmetric-1} and Subsection \ref{conjecture_subsection}, for the occupation time of the vertex in the moving frame with process characteristic velocity $\lfloor (1-2\rho)m\rfloor$.

Finally, it would be of interest to explore more the proposed `extension' of the KPZ class to other long-range models when $3/2\leq \alpha \leq 2$.  One feels that it
is perhaps a generic feature of a large class of mass conservative particle systems.

\section{Tools}
\label{Tools}

The goal of this section is to develop for long-range processes, $H_{-1}$ norm estimates, generalized `duality' decompositions, `free particle' approximations and other technical bounds useful in the sequel.  We refer the reader to \cite{KLO}, \cite{B}, \cite{S} for more discussion of the material in the finite-range context.

\subsection{Resolvent norms}
\label{resolventnorms}

Denote the symmetric and antisymmetric parts of $\mathcal{L}$ by $\mathcal{S}$ and $\mathcal{A}$, respectively:
\begin{equation*}
\mathcal{S}:=\frac{\mathcal{L}+\mathcal{L}^*}{2}
\end{equation*}
\begin{equation*}
\mathcal{A}:=\frac{\mathcal{L}-\mathcal{L}^*}{2}.
\end{equation*}
A straightforward calculation shows that ${\mc S}$ itself generates the symmetric exclusion process with jump probability $s$: On local functions,
$${\mc S}f(\eta) \ = \ \sum_{x,y\in \ZZ^d}  p(y) \big[f(\eta^{x,x+y}) - f(\eta)\big].$$
The corresponding Dirichlet form $\langle f, -\mathcal{L}f\rangle_\rho$, acting on local functions, after a calculation, is given by
\begin{equation}\label{dirichletform}
\langle f, -\mathcal{L}f \rangle_{\rho}\ =\ \langle f,-\mathcal{S}f\rangle_{\rho} \ = \ \frac{1}{2}\sum_{x,y\in \ZZ^d} s(y)\E_\rho\Big[\big(f(\eta^{x,x+y}) - f(\eta)\big)^2\Big] \ \geq \ 0.\end{equation}
In particular, $-{\mc S}$ is a nonnegative operator.

We now define the following resolvent norms. Fix $\lambda>0$ and consider
$(\lambda-\mathcal{S})^{-1}:\LL^2(\nu_\rho)\rightarrow{\LL^2(\nu_{\rho})}$
where, in terms of the semigroup $T^{(S)}_t$ for the symmetric process generated by ${\mc S}$,
\begin{equation*}
(\lambda-\mathcal{S})^{-1}f(\zeta):=\int_{0}^{\infty}e^{-\lambda t}T^{(S)}_tf(\zeta)dt.
\end{equation*}
Denote by $\mathcal{H}_{1,\lambda}$ the closure of local functions $f$ such that $\|f\|_{1,\lambda}^2:=\<f,(\lambda-\mathcal{S})f\>_{\rho}<\infty$. Let ${\mc H}_{-1,\lambda}$ be its topological dual with respect to $\LL^2(\nu_{\rho})$ and let $\| \cdot \|_{-1,\lambda}$ be its norm.  One has
\begin{eqnarray*}
\|f\|_{-1,\lambda} & = & \sup\Big\{\langle f, \phi\rangle_\rho/\|\phi\|_{1,\lambda}: \ \phi \ {\rm local}\Big\}\\
&=&
\langle f, (\lambda - {\mc S})^{-1}f\rangle_\rho\\
&=& \int_0^\infty e^{-\lambda t}\langle f, T^{(S)}_t f\rangle_\rho.\end{eqnarray*}

Analogously, let
$\mathcal{H}_1$ be the closure over local $f$ such that $\|f\|_1^2:=\<f,-\mathcal{S}f\>_{\rho}<\infty$. Denote ${\mc H}_{-1}$ as its topological dual with respect to $\LL^2({\nu}_{\rho})$ and $\| \cdot\|_{-1}$ its norm, namely $\|f\|_{-1} = \sup\big\{\langle f, \phi\rangle_\rho/\|\phi\|_1: \ \phi \ {\rm local}\big\}$.

By the formulas, we have $\|f\|_{1,\lambda} \geq \|f\|_1$ and so $\|f\|_{-1,\lambda}\leq \|f\|_{-1}$.  Moreover, as $T^{(S)}_t$ is reversible with respect to $\nu_\rho$, $\langle f, T^{(S)}_t f\rangle_\rho = \langle T^{(S)}_{t/2} f, T^{(S)}_{t/2}f\rangle_\rho\geq 0$.  Hence, the limit $\lim_{\lambda\downarrow 0} \|f\|_{-1,\lambda} = \|f\|_{-1}$ exists, which may be infinite.

The resolvent $(\lambda - {\mc L})^{-1}:\LL^2(\nu_\rho) \rightarrow \LL^2(\nu_\rho)$ is given by
$$(\lambda - {\mc L})^{-1} f(\zeta) \ = \ \int_0^\infty e^{-\lambda t}T_t f(\zeta) dt,$$
with respect to the (asymmetric) generator ${\mc L}$ and semigroup $T_t$, will be important in many arguments.
Observe that by a simple integration by parts and stationarity of the process, we may relate the Tauberian variance $L_f(\lambda)$ to the quadratic form with respect to $(\lambda - {\mc L})^{-1}$:
\begin{eqnarray}
\label{eq:lapltrnasfotvar}
L_f (\lambda) & = & \int_0^\infty e^{-\lambda t}\sigma^2_f(t)dt \nonumber\\
&=&2\int_0^\infty e^{-\lambda t}\int_0^t\int_0^s \langle f, T_{s-u}f\rangle_\rho du ds dt\nonumber\\
&=&  \frac{2}{\lambda^2} \; \langle f, (\lambda -{\mc L})^{-1} f \rangle_{\rho}.
\end{eqnarray}

As discussed in \cite{S},
$$
\left[\frac{1}{2}\left (\lambda - {\mc L})^{-1} + (\lambda - {\mc L}^*)^{-1}\right)\right]^{-1} \ = \ (\lambda - {\mc L}^*)(\lambda - {\mc S})^{-1}(\lambda - {\mc L}) \ = : Q,$$
the point being that one can symmetrize in the inner product $\langle f, (\lambda - {\mc L})^{-1}f\rangle_\rho$ and interpret it as the dual form with respect to the operator $Q$.  Since $\langle f, Qf\rangle_\rho = \langle (\lambda - {\mc L})f, (\lambda - {\mc S})^{-1}(\lambda - {\mc L})f\rangle_\rho \geq 0$ for all local $f$, we see that $Q$ and $Q^{-1}$ are nonnegative symmetric operators which admit square roots.  Hence, we may apply Schwarz inequality to obtain
\begin{equation}
\label{L_schwarz}
L_{f+g}(\lambda) \ \leq \ 2L_f(\lambda) + 2L_g(\lambda).
\end{equation}

We now recall a basic estimate, proved in \cite{S}, which applies for general Markov processes.
\begin{lemma}\label{seth bound}
For $t>0$ and $f\in{\LL^2(\nu_{\rho})}$ such that $\E_\rho[f]=0$, we have
\begin{eqnarray*}
{\mathbb E}_{\rho}\Big[\Big(\Gamma_f(t)\Big)^2\Big] &  \leq & {10\, t \, \langle f , (1/t  -{\mc L})^{-1} f \rangle_{\rho}} \ = \ 10t^{-1}L_f(t^{-1}).
\end{eqnarray*}
\end{lemma}

In \cite{S}, the following sup variational form for the quadratic form is proved.  The inf variational form is an equivalent relation.

\begin{lemma}\label{sunder equivalence}
Let $f:\Omega \rightarrow{\mathbb{R}}$ be a local function and let $\lambda>0$. Then,
\begin{eqnarray*}
\<f,(\lambda-\mathcal{L})^{-1}f\>_{\rho} & = & \sup_{g}\Big\{ 2\<f,g\>_{\rho}-\|g\|_{1,\lambda}^2-\|\mathcal{A}g\|_{-1,\lambda}^2 \Big\}\\
& = & \inf_{g}\Big\{ \|f+\mathcal{A}g\|_{-1,\lambda}^2+\|g\|_{1,\lambda}^2 \Big\},
\end{eqnarray*}
where the supremum
and the infimum
are taken over local functions $g$. In particular, by taking $g\equiv0$, we have
\begin{equation*}
\<f,(\lambda-\mathcal{L})^{-1}f\>_{\rho} \ \le\  \<f,(\lambda-\mathcal{S})^{-1}f\>_{\rho}.
\end{equation*}
\end{lemma}

 We remark, although these variational formulas are quite difficult to compute, by restricting
the supremum or the infimum over the class of degree one functions, that is linear combinations of the functions $\{\eta (x) -\rho: x \in \ZZ^d\}$, we can sometimes
 extract interesting lower and upper bounds.

Putting things together, we obtain the following estimate which bounds the variance, with respect to the process generated by ${\mc L}$, in terms of the symmetric part ${\mc S}$.

\begin{corollary}\label{cor seth bound}
For $t>0$ and $f\in{\LL^2(\nu_{\rho})}$ such that $\E_\rho[f]=0$, we have
\begin{equation*}
{\mathbb E}_{\rho}\Big[\Big(\Gamma_f(t)\Big)^2\Big]\  \leq\ {10\, t \, \| f \|^2_{-1,t^{-1}}} \ = \ 10t^{-1}L_f^{(S)}(t^{-1}).
\end{equation*}
\end{corollary}

\subsection{Duality}
\label{subsec:duality}
We now detail certain `duality' decompositions which often help simplify calculations.
 For finite subsets $A \subset \ZZ^d$, let $\Psi_A$ be the function
$$\Psi_A \ = \ \prod_{x \in A} \cfrac{\eta (x) -\rho}{\sqrt{\chi(\rho)}},$$
where $\chi(\rho)=\rho(1-\rho)$. The collection $\{ \Psi_A \; : \; A \subset \ZZ^d\}$ forms an orthonormal basis of ${\mathbb L}^2 (\nu_{\rho})$.

Let ${\mc E}_n = \{A\subset\ZZ^d: |A|=n\}$ be the class of subsets of $\ZZ^d$ with $n\geq 1$ points.  Let also ${\mc H}_n$ be the set of functions ${F}: {\mc E}_n \to \RR$ such that  $\sum_{|A|=n} F^2 (A) <\infty$; when $n=0$, ${\mc H}_0$ denotes the space of constants.
Denote also, for $n\geq 1$, ${H}_n$ as the space of `$n$-point' functions $f$ in form $f=\sum_{|A|=n} {\mf f} (A) \Psi_A$ with ${\mf f} \in {\mc H}_n$; for $n=0$, as before, $H_0$ denotes the space of constants.  We have thus the orthogonal decomposition
$${\mathbb L}^2 (\nu_{\rho}) \ = \ \oplus_{n\ge 0} {H}_n.$$

Functions ${\mf f}$ in ${\mc H}_n$ can be identified with a symmetric function ${\mf f}: \chi_n \backslash D_n \to \RR$ where $\chi_n =(\ZZ^d)^n$ and $D_n=\{ (x_1, \ldots,x_n) \in ({\ZZ^d})^n \, ; \, \exists i \ne j \; \textrm{such that} \;  x_i =x_j\}$ via ${\mf f} (x_1, \ldots,x_n) := {\mf f} (\{ x_1, \ldots,x_n\})$. In the sequel, we will use this identification implicitly.

We now decompose the generator ${\mc L}$ on the basis $\{ \Psi_A \; : \; A \subset \ZZ^d\}$. Given a subset $A$ of $\ZZ^d$ and $x,y \in \ZZ^d$ denote by
$A_{x,y}$ the set $A_{x,y} =A \backslash \{ x \} \cup\{ y \}$ if $x \in A$ and $y \notin A$, by $A_{x,y} =A \backslash \{ y \} \cup \{x \}$ if $ x \notin A$ and by $A_{x,y} =A$ otherwise.  Let also $\mathcal{E}:=\bigcup_{n\geq 0}\mathcal{E}_n$.  Then,
\begin{eqnarray*}
{\mc L} f &=& \sum_{A \in {\mc E}} ({\mf L} {\mf f})(A) \Psi_A, \\
{\mc S} f &=& \sum_{A \in {\mc E}} ({\mf S} {\mf f})(A) \Psi_A,\\
{\mc A} f &=& \sum_{A \in {\mc E}} ({\mf A} {\mf f})(A) \Psi_A,
\end{eqnarray*}
where
$${\mf L}= {\mf S} +{\mf A}\ \ {\rm and \ \ }{\mf S}={\mf L}^1, \ \ {\mf A} = (1-2\rho) {\mf L}^2 + 2 \sqrt{\chi (\rho)} ({\mf L}^+ -{\mf L}^-),$$
and
\begin{eqnarray*}
({\mf L}^1 {\mf f}) (A) &=& (1/2) \sum_{x,y \in \ZZ^d} s(y-x) \left[ {\mf f} (A_{x,y} ) -{\mf f} (A)\right],\\
({\mf L}^2 {\mf f}) (A) &=& \sum_{x \in A, y \notin A} a(y-x) \left[ {\mf f} (A_{x,y} ) -{\mf f} (A)\right],\\
({\mf L}^{-} {\mf f}) (A) &=& \sum_{x \notin A, y \notin A} a(y-x) {\mf f} (A \cup \{ x \} ),\\
 ({\mf L}^{+} {\mf f}) (A) &=& \sum_{x \in A, y \in A} a(y-x) {\mf f} (A \backslash \{ y \} ).
\end{eqnarray*}

The operator ${\mf S}$, which generates the dual symmetric exclusion process, takes ${\mc H}_n$ to ${\mc H}_n$ for $n\geq 0$. Its restriction to ${\mc H}_n$ is the generator of the set of $n$ particles interacting by the exclusion rule with the jump probability $s$. This property represents the classical self-duality of the symmetric exclusion process \cite{Li}.

Since the spaces $\{H_n: n\geq 0\}$ are orthogonal and ${\mc S}$ is invariant on each $H_n$, for $f\in H_n$ and $g\in H_m$ with $n\neq m$, we have $\|f + g\|^2_{1,\lambda} = \|f\|^2_{1,\lambda} + \|g\|^2_{1,\lambda}$.  Similarly, from the sup-variational formula in Lemma \ref{sunder equivalence}, we have
\begin{equation}
\label{orthogonality}
\|f+g\|^2_{-1,\lambda} \ =\  \|f\|^2_{-1,\lambda} + \|g\|^2_{-1,\lambda}.\end{equation}

Although self-duality is not valid in the asymmetric setting, the decomposition of the generator gives an extension of the duality relations.
Note that the operators ${\mf L}^1$ and ${\mf L}^2$ preserve the degree of functions, but that
${\mf L}^+$ and ${\mf L}^-$ respectively increase and decrease the degree by $1$. The operator ${\mf A}$ has a decomposition of the form
$${\mf A} = \sum_{n \ge 1}\Big( {\mf A}_{n-1\, n} +{\mf A}_{n \, n} +{\mf A}_{n \, n+1}\Big),$$ where ${\mf A}_{n \, m}$ is the projection onto ${\mc H}_m$ of the restriction of ${\mf A}$ to ${\mc H}_n$.

Later on, we will primarily consider functions of degree $1$ and degree $2$.  We note the following action of the operators ${\mf A}_{11}=(1-2\rho) {\mf B}_{11}$ and ${\mf A}_{12} =2 {\sqrt{\chi (\rho)}} {\mf B}_{12}$:
\begin{eqnarray*}
({\mf B}_{11} {\mf f}) (x) &=& \sum_{y\in \ZZ^d} a (y-x) \left[ {\mf f} (y) -{\mf f} (x)\right],\\
({\mf B}_{12} {\mf f} ) (\{ x,y\}) & = &a (y-x) \left[ {\mf f} (x) - {\mf f} (y) \right].
\end{eqnarray*}

\subsection{Approximation by free particles}
\label{subsec:approx}
We now discuss `free particle' approximations though which $n$-particle exclusion interactions can be estimated in terms of $n$-`free' or independent particles.
For a local function $f= \sum_{|A|=n} {\mf f}(A) \Psi_A \in {H}_n$,
the $H_{1,\lambda}$ norm can be written in terms of the dual function ${\mf f}\in {\mc H}_n$:
\begin{equation}
\label{H_1_formula}
\| f\|_{1,\lambda}^2 \ =\  \lambda \sum_{|A|=n} {\mf f}^2 (A) + \sum_{u,v \in \ZZ^d} \sum_{|A|=n} s(v-u) \left[ {\mf f}(A_{u,v}) -{\mf f} (A) \right]^2.
\end{equation}
Similarly, the $H_{-1,\lambda}$ norm of $f$ can be written in terms ${\mf f}$.

Because of the exclusion interaction, it is not easy, even for simple functions, to compute these norms. The idea then is to compare them to corresponding norms without the exclusion, that is for a system composed of free particles.  Observe there exists a positive constant $K_0$ such that
\begin{equation}
\label{s_s_0_bound}
K_0^{-1} s_0 (\cdot) \ \le \ s (\cdot) \ \le \ K_0 s_0 (\cdot)
\end{equation}
where $s_0$ is the symmetric probability, defined for $y\in \ZZ^d$ by
\begin{equation*}
s_0 (y) \ = \ \cfrac{c_0} {|y|^{d+\alpha}},
\end{equation*}
where $c_0$ is a normalization constant.

The ${\bb H}_{1, \rm{free}, \lambda}$-norm of the symmetric function $F:\chi_n\to \RR$  is defined by
\begin{equation*}
\|F\|_{1,  {\rm{free}},\lambda}^2 \ =\  \lambda \cfrac{1}{n!} \sum_{\bx } F^2 (\bx ) + \cfrac{1}{n!} \sum_{j=1}^n \sum_{z\in \ZZ^d } \sum_{\bx} s_0 (z) \left[ F({\bx} +z{\be }_j) -F({\bx}) \right]^2
\end{equation*}
where $\bx + z \be_j=(x_1, \ldots, x_{j-1}, x_j +z, x_{j+1}, \ldots, x_n)$. If $n=1$, the formula reduces to
\begin{equation*}
\|F\|_{1,{\rm{free}},\lambda}^2 \ =\  {\lambda} \sum_{x \in \ZZ^d} F^2 (x) + \sum_{z,x \in \ZZ^d } s_0 (z-x) \left[ F(z) -F(x) \right]^2.
\end{equation*}
When $n=2$, it is given by
 \begin{equation*}
\|F\|_{1,{\rm{free}},\lambda}^2 \ =\  \cfrac{\lambda}{2} \sum_{x,y  \in \ZZ^d} F^2 (x,y) + \sum_{z,x,y \in \ZZ^d } s_0 (z-x) \left[ F(z,y) -F(x,y) \right]^2.
\end{equation*}
The implicit `free' dynamics is that each particle moves independently according to jump probability $s_0$.

 The  ${\bb H}_{-1, \rm{free},\lambda}$-norm of the symmetric function $G: \chi_n \to \RR$  is defined by
\begin{equation*}
\|G\|_{-1,  {\rm{free}},\lambda}^2 \ =\ \sup_{F: \chi_n \to \RR} \left\{ \cfrac{1}{n!} \sum_{\bx}  F(\bx) G(\bx) - \| F \|_{1, {\rm{free},\lambda} }^2 \right\}.
\end{equation*}

To ${\mf f}\in {\mc H}_n$, we associate a symmetric function ${\tilde {\mf f}}: \chi_n \to \RR$ which coincides with ${\mf f}$ outside $D_n$ and for $(x_1, \ldots,x_n) \in D_n$ by
\begin{equation*}
{\tilde {\mf f}} (x_1, \ldots,x_n)\ =\ {\mathbf E} \left[ {\mf f} (X_1(T), \ldots, X_n (T))\right]
\end{equation*}
where ${\bf E}$ is the expectation with respect to the law of $n$-independent simple symmetric random walks $(X_1 (t), \ldots,X_n(t))_{t \ge 0}$ on $\ZZ^d$ starting from $(x_1, \ldots, x_n)$ and $T$ is the hitting time of $\chi_n \backslash D_n$. For example, if ${\mf f} \in {\mc H}_2$ then
\begin{equation}
\label{tilde_formula}
{\tilde {\mf f}} (x,y)\ =\
\left\{\begin{array}{lr}
{{\mf f}} (\{x,y\}) & \ \  {\text{ if } }  x \ne y,\\
(2d)^{-1} \sum_{i=1}^d \left( {\mf f}(\{x + e_i, x\}) +{\mf f}(\{x-e_i,x\}) \right) &\ \ {\text{ if }} x = y.
\end{array}\right.
\end{equation}

With respect to the symmetric function $F:\chi_n \to \RR$, we also associate the function ${\mf W}_n {F} : \chi_n\to \RR$ which coincides with $F$ outside $D_n$ and is equal to $0$ on $D_n$.

\begin{lemma}
\label{lem:afp}
Let $n \ge 1$. There exists a constant $C_{n,d}$ independent of $\lambda$ such that for $f \in H_n$ and its dual function ${\mf f}\in {\mc H}_n$ we have
\begin{equation*}
C_{n,d}^{-1} \|{\tilde {\mf f}} \|_{1, {\rm{free}},\lambda}^2 \ \le\   \| f\|_{1,\lambda}^2 \ \le\  C_{n,d} \|{\tilde {\mf f}} \|_{1,{\rm{free}},\lambda}^2.
\end{equation*}
It follows that
\begin{equation*}
\| f\|_{-1,\lambda}^2 \ \le \ C_{n,d} \|{ {\mf W}_n { \tilde {\mf f}} } \|_{-1,{\rm{free}, \lambda}}^2.
\end{equation*}
\end{lemma}

\begin{proof}
We only give the proof of the first claim for $n=2$ to reduce notation; the argument for general $n\geq 1$ is similar.  The second claim is a consequence of the first one since $\|\tilde {\mf f}\|^2_{1, {\rm free}, \lambda} \geq \|f\|^2_{1,\lambda} = \|{ {\mf W}_n { \tilde {\mf f}} } \|_{1,\lambda, {\rm{free}}}^2$ noting the $H_{1,\lambda}$ formula in \eqref{H_1_formula}, and the dual form of $\|\cdot\|_{-1,{\rm free},\lambda}$.
Let now $C$ be a positive constant independent of $\lambda$ whose value can change from line to line.

The first term in \eqref{H_1_formula}, noting \eqref{tilde_formula}, can be bounded by Schwarz's inequality:
\begin{equation*}
C^{-1} \sum_{x,y\in \ZZ^d} {\tilde {\mf f}}^2 (x,y) \ \le \ \sum_{x\neq y} {\mf f}^2 (\{x,y\}) \ \le \ \sum_{x,y\in \ZZ^d} {\tilde {\mf f}}^2 (x,y).
\end{equation*}
With respect to the second term in \eqref{H_1_formula}, noting \eqref{s_s_0_bound}, by replacing ${\mf f}$ with $\tilde {\mf f}$,
we have trivially
$$\sum_{z,x,y \in \ZZ^d } s(z-x) \left[ {{\mf f}} (\{z,y\}) -{{\mf f}}(\{x,y\}) \right]^2 {\bf 1}_{z \ne y, z \ne x, x \ne y}
 \ \le \ C \sum_{z,x,y \in \ZZ^d } s_0 (z-x) \left[ {\tilde {\mf f}} (z,y) -{\tilde {\mf f}}(x,y) \right]^2.$$
 On the other hand,
 to show
 $$\sum_{z,x,y\in \ZZ^d} s_0(z-x)\left[ {\tilde {\mf f}} (z,y) -{\tilde {\mf f}} (x,y)\right]^2 \ \leq \
 C \sum_{z,x,y\in \ZZ^d}s(z-x)\left[ {\tilde {\mf f}} (z,y) -{\tilde {\mf f}} (x,y)\right]^2 {\bf 1}_{z\ne y, z\ne x, x\ne y}$$
 it is enough to verify
 $$\sum_{x\ne y} s_{0} (y-x) \left[ {\tilde {\mf f}} (y,y) -{\tilde {\mf f}} (x,y)\right]^2 \ \leq \ C\sum_{z,x,y \in \ZZ^d } s(z-x) \left[ {{\mf f}} (\{z,y\}) -{{\mf f}}(\{x,y\}) \right]^2{\bf 1}_{z\neq y, z\neq x, x\neq y}.$$

To this end, by Schwarz's inequality, we have
\begin{eqnarray*}
&&\sum_{x\ne y} s_{0} (y-x) \left[ {\tilde {\mf f}} (y,y) -{\tilde {\mf f}} (x,y)\right]^2\\
&&\ \ \leq \ C \sum_{i=1}^d \sum_{x \ne y} s_0 (y-x) \left\{\left[ {\mf f}(\{y + e_i, y\}) -{\mf f}(\{x,y\}) \right]^2 + \left[ {\mf f}(\{y - e_i, y\}) -{\mf f}(\{x,y\}) \right]^2\right\}.
\end{eqnarray*}
Since $\sup_{i =1, \ldots ,d} \sup_{z \ne 0, \pm e_i} s_0 (z)/s_0 (z \pm e_i) \le C$ and $\sum_{i=1}^d 1 =d$,
the right-side above is bounded by
\begin{eqnarray*}
C  \sum_{x,y,z\in\ZZ^d} s_0 (z-x) \left[ {\mf f}(\{z, y\}) -{\mf f} (\{x,y\}) \right]^2 {\bf 1}_{x \ne y, x \ne z, z\ne y},\end{eqnarray*}
as desired.
\end{proof}

\subsection{Fourier estimates}
\label{theta_section}
Let $\TT^d =[0,1)^d$ be the $d$-dimensional torus.  Denote the Fourier transform of the function $\psi\in {\mathbb L}^2 (\chi_n)$ by ${\widehat \psi}$:  For $(s_1, \ldots,s_n) \in (\TT^d)^n$,
\begin{equation*}
{\widehat \psi} (s_1, \ldots,s_n) \ =\  \cfrac{1}{\sqrt{n!}} \sum_{(x_1, \ldots, x_n) \in \chi_n} e^{2 \pi i(x_1 \cdot s_1 + \ldots+ x_n \cdot  s_n)} \psi (x_1, \ldots,x_n).
\end{equation*}

As the `free' dynamics consists of independent random walks moving with jump probability $s_0$, the ${\bb H}_{1,{\rm free}, \lambda}$-norm of $\psi$ is
$$\|\psi\|^2_{1, {\rm free}, \lambda} \ = \ \frac{1}{(2\pi)^{nd}}\int_{(\TT^d)^n} \left(\lambda + \sum_{i=1}^d\theta_d(s_i;s_0(\cdot)) \right)|\hat\psi(s_1,\ldots, s_n)|^2 ds_1\ldots ds_n.$$
Also, the ${\bb H}_{-1, {\rm{free}},\lambda}$-norm of $\psi$ is written as

\begin{equation}
\label{eq:h-1lfp}
\| \psi \|_{-1, \lambda, {\rm{free}},}^2 \ = \ \frac{1}{(2\pi)^{nd}}\int_{({\TT^d})^n} \cfrac{|{\hat \psi} (s_1, \ldots,s_n)|^2}{\lambda + \sum_{i=1}^d\theta_d (s_i;s_0(\cdot))  }\, {d {s}_1 \ldots d{s}_n}.
\end{equation}
Here,  for $u \in \TT^d$, and symmetric transition function $r:\ZZ^d \rightarrow [0,1]$,
\begin{equation}\label{theta function}
\theta_d (u;r(\cdot)) \ = \ 2 \sum_{z \in \ZZ^d} r(z) \sin^2 (\pi u\cdot z).
\end{equation}

When `free' particle $H_{\pm 1}$ norms are used in the sequel, $r = s_0(\cdot)$.  However, in the proof of the functional CLT in Theorem \ref{th:conv-bw}, $r = s(\cdot)$,
the symmetric part of $p$ given by
$$ s(z) = \frac{c \gamma (z)}{{|z|}^{d+\alpha}}, \quad {\rm and }\quad \gamma(z)=\sum_{j=1}^d \cfrac{b_j^+ + b_j^-}{2}{\bf 1}_{z \cdot e_j \ne 0}.$$
Note that $s_0$ is a special case of the more general formulation of $s$.

We now state an estimate used throughout the proofs.
Let ${\mc C}_d$ be the set of extremal points of $[0,1]^d$,
\begin{equation}
\label{eq:extrpoints}
{\mc C}_d=\{ \sigma_1 e_1 + \ldots + \sigma_d e_d \, ; \, \sigma_i \in \{0,1\} \}.
\end{equation}
We note the function $\theta_d(u;s(\cdot))$ is smooth, even, positive on $\TT^{d} \backslash {\mc C}_d$  and vanishes exactly on ${\mc C}_d$.

\begin{lemma}
\label{lem:thetad}
Let $\gamma_0= \frac{1}{2} \sum_{j=1}^d (b_j^+ + b_j^-)$.
The function $\theta_d = \theta_d(\cdot;s(\cdot))$ is bounded above by a positive constant. For $u\in \TT^d$ and $w \in {\mc C}^d$, $\theta_d(u-w) = \theta_d(u)$ and, as $u-w\rightarrow 0$,
\begin{equation*}
\theta_d(u-w)\  = \ J(d,\alpha) F_{\alpha} (u-w) + o(F_{\alpha} (u-w) )
\end{equation*}
where
\begin{equation*}
F_{\alpha} (x)\ = \
\begin{cases}
|x|^{\alpha} \quad &\text{ if } \alpha <2\\
 |x|^2\log (|x|) \quad &\text{ if } \alpha=2\\
  |x|^{2} \quad &\text{ if } \alpha >2
\end{cases}
\end{equation*}
and
\begin{equation*}
J(d,\alpha)\ =\
\begin{cases}
c_0\gamma_0\int_{ q \in \RR^d } \cfrac{\sin^2 \left( \pi q_1 \right)}{|q|^{d+\alpha}}dq   \quad &\text{ if } \alpha <2\\
-\cfrac{ c_0\gamma_0\pi^2}{d} \quad &\text{ if } \alpha=2\\
\cfrac{c_0\gamma_0\pi^2}{d (\alpha -2)} \quad &\text{ if } \alpha >2.
\end{cases}
\end{equation*}
\end{lemma}

\begin{proof}
By periodicity of $\theta_d$, we can restrict the proof to the case $w=0$. Since $s_0$ is a radial function, we can write $\theta_d (u)$ as
\begin{equation*}
\theta_d (u) \ =\  c_0| u|^{\alpha} \left[ |u|^{d} \sum_{z \ne 0}  \cfrac{\gamma(|u|z) }{| \,|u| z\, |^{d+ \alpha}}\,  \sin^2 \left(\pi \cfrac{u}{|u|} \cdot |u| z \right) \right].
\end{equation*}
This is equivalent in order, as $u$ vanishes, to
\begin{equation*}
c_0\gamma_0|u|^{\alpha} \int_{|q| \ge |u| } \cfrac{1}{|q|^{d+\alpha}} \, \sin^2 \left( \pi \cfrac{u}{|u|} \cdot q \right) \, dq\ =\   c_0 \gamma_0|u|^{\alpha} \int_{|q| \ge |u| } \cfrac{1}{|q|^{d+\alpha}} \, \sin^2 \left( \pi q_1 \right) \, dq.
\end{equation*}
Here, the second equality follows from the invariance of the Lebesque measure by the orthogonal group.

If $\alpha<2$, the last integral is convergent. If $\alpha>2$, the integral diverges as $u$ vanishes as
\begin{equation*}
 \int_{|q| \ge |u| } \cfrac{1}{|q|^{d+\alpha}} \, \sin^2 \left( \pi \cfrac{u}{|u|} \cdot q \right) \, dq \ \sim\  \cfrac{\pi^2}{d (\alpha -2)} |u|^{2 -\alpha}.
\end{equation*}
If $\alpha=2$, the integral diverges as
\begin{equation*}
 \int_{|q| \ge |u| } \cfrac{1}{|q|^{d+2}} \, \sin^2 \left( \pi \cfrac{u}{|u|} \cdot q \right) \, dq \ \sim\  -\cfrac{\pi^2}{d} \log (|u|).
\end{equation*}
\end{proof}

\subsection{One point function lower bounds} \label{one_pt_section}

The following lower bound will be useful in the proof of Theorems \ref{th:sd-asymmetric-1} and \ref{th:sd-asymmetric-2}, and may be skipped on first reading.  We estimate the variational formulas of the resolvent norms given in Lemma \ref{sunder equivalence} with respect to the occupation function $\Psi_{\{0\}}$.

Recall
the decomposition of the probability $p = s + a$ and
the notation in Subsection \ref{theta_section}.  Let $\theta_d = \theta_d(\cdot;s_0(\cdot))$ and
\begin{eqnarray*}
F^d_{\lambda, \rho} (u) &:=& [ \lambda+ \theta_d (u) ] + (1-2\rho)^2 \cfrac{|{\hat a} (u)|^2  }{\lambda+ \theta_d (u)} \\
&&\ \ \ \ \ \ \ + \chi (\rho)  \sum_{V \in {\mc C}_d} \int_{ s\in D_V (u)}  \frac{|{\hat a} (s) +{\hat a} (u-s)|^2 }{\lambda + \theta_d (s) + \theta_d (u-s) } \, ds,
\end{eqnarray*}
where
\begin{equation}
\label{D_Veq}
D_V(u) \ := \ \left\{s \in [0,1)^d, \, (u-s+V) \in [0,1)^d \right\},
\end{equation}
and
\begin{equation}
\label{eq:Ramses2}
I_d (\lambda,\rho)\ :=\ \int_{\TT^d} \cfrac{1}{F^d_{\lambda, \rho} (u)}\; du.
\end{equation}

\begin{proposition}
\label{I_prop}
There exists a constant $C$, not depending on $\lambda$, such that
$$\langle (\lambda -{\mc L} )^{-1} \Psi_{\{0\}}, \Psi_{\{0\}} \rangle_\rho \ \geq \ C I_d(\lambda,\rho).$$
\end{proposition}

\begin{proof}
The first step is to use the sup-variational formula in Lemma \ref{sunder equivalence} to express
\begin{equation*}
\langle (\lambda -{\mc L} )^{-1} \Psi_{\{0\}}, \Psi_{\{0\}} \rangle_\rho \ =\  \sup_{g} \left\{ 2 \langle \Psi_{\{0\}}, g \rangle - \| g\|_{1, \lambda}^2 - \| {\mc A} g \|_{-1, \lambda}^2\right\}.
\end{equation*}

The second step is to restrict the supremum over functions $g= \sum_{x \in \ZZ^d} {\mf g} (x) \Psi_{\{x\}}$ in  ${H}_{1}$ to get a lower bound.  By orthogonality relation \eqref{orthogonality} and Lemma \ref{lem:afp}, we have
\begin{eqnarray}\label{useful estimate}
\|g \|_{1,\lambda}^2 &\le & C \| {\mf g} \|_{1, {\rm{free}}, \lambda}^2 \ = \  C \left[ \lambda \sum_{x} {\mf g}^2 (x) + \sum_{x,y} s_{0} (y-x) \left[ {\mf g} (y) -{\mf g} (x) \right]^2  \right] \nonumber\\
\|{\mc A} g \|_{-1,\lambda}^2 & = & \big\|\sum_{|A|=1}({\mf A}_{1,1}{\mf g})(A)\Psi_A\big \|_{-1,\lambda}^2 + \big\|\sum_{|A|=2}({\mf A}_{1,2}{\mf g})(A)\Psi(A)\big \|_{-1,\lambda}^2\nonumber\\
&\le & C \left[  \| {\mf W}_1 {\mf A}_{1,1} {\mf g} \|_{-1, {\rm{free}}, \lambda}^2 +   \|  {\mf W}_2 {\mf A}_{1,2} {\mf g} \|_{-1, {\rm{free}}, \lambda}^2 \right].
\end{eqnarray}
Recall the operators ${\mf T}_{1,1}:={\mf W}_1 {\mf A}_{1,1}$ and ${\mf T}_{1,2}:={\mf W}_1 {\mf A}_{1,2}$ act on functions defined on $\ZZ^d$ and $(\ZZ^d)^2$ respectively, and are given by
\begin{eqnarray*}
({\mf T}_{1,1} {\mf g}) (x)&=& (1-2\rho) \sum_{y\in \ZZ^d} a(y -x) \left[ {\mf g} (y) -{\mf g} (x) \right], \\
({\mf T}_{1,2} {\mf g} )(x,y) &=&  \sqrt{\chi(\rho)}  a(y -x) \left[ {\mf g} (x) -{\mf g} (y) \right].
\end{eqnarray*}
It follows that
\begin{eqnarray}
\label{lb_help_eq}
&&\langle (\lambda -{\mc L} )^{-1} \Psi_{\{0\}}, \Psi_{\{0\}} \rangle_\rho\\
 & &\quad \quad \quad  C \sup_{ {\mf g}} \Big\{ 2 {\mf g} (0) - \lambda \sum_{x\in\ZZ^d} {\mf g}^2 (x) + \sum_{x,y\in\ZZ^d} s_{0} (y-x) \big[ {\mf g} (y) -{\mf g} (x) \big]^2\nonumber \\
&&  \quad \quad \quad \quad \quad  - \| {\mf T}_{1,1} {\mf g}\| _{-1, {\rm{free}}, \lambda}^2 - \| {\mf T}_{1,2} {\mf g}\| _{-1, {\rm{free}}, \lambda}^2 \Big\}\nonumber
\end{eqnarray}

The last step is to express the terms in this formula via the Fourier transform of ${\mf g}$. We have, as $a$ is anti-symmetric,
\begin{eqnarray*}
\widehat{ {\mf T}_{1,1} {\mf g}  }\, (s)&=& -  (1-2\rho) \, \widehat {a}(s) \,  \widehat {\mf g} (s), \\
\widehat { {\mf T}_{1,2} {\mf g} } \, (s,t)&=&  -\sqrt{\chi(\rho)} \left[ \widehat {a}(s) + \widehat {a}(t) \right]  \, \widehat {\mf g} (s+t).
\end{eqnarray*}
Recall ${\mc C}_d = \left\{ \sigma_1 e_1 + \ldots +\sigma_d e_d \, ; \, \sigma_i \in \{0,1\} \right\}\subset \ZZ^d$.
Observe that the set $[0,2)^d$ is equal to the disjoint union of the sets $[0,1)^d + V$ over $V \in {\mc C}^d$. Then, by periodicity of $\hat g$, $\theta_d$ and ${\hat a}$,  we have
\begin{equation*}
\| {\mf T}_{1,2} {\mf g}\| _{-1, {\rm{free}}, \lambda}^2\ =\ \chi (\rho)  \int_{[0,1)^d} |{\hat g} (u)|^2 \left[ \sum_{V \in {\mc C}_d} \int_{ s\in D_V (u)}  \frac{|{\hat a} (s) +{\hat a} (u-s)|^2 }{\lambda + \theta_d (s) + \theta_d (u-s) } \, ds \right] \, du.
\end{equation*}

Because $\mf g$ is a real function, $\hat {\mf g}$ has even real and odd imaginary parts.  To obtain a lower bound of \eqref{lb_help_eq}, we maximize, over such square integrable complex functions $\varphi: \TT^d \to \bb C$,
 the following expression
\begin{equation}
\label{one_point_help}
\int_{\TT^d} du  \left\{ 2 {\varphi} (u)  - F^d_{\lambda, \rho} (u) \, | \varphi (u)|^2\right\}  \, du.
\end{equation}
Noting, for $A>0$, that $\sup_{x \in \RR} [2x-A x^2] = 1/ A$ is realized at $x= 1/ A$, the supremum in \eqref{one_point_help} is attained at
 $\varphi = 1/  F^d_{\lambda , \rho}$ and  the value of the supremum is $I_d(\lambda,\rho)$.
\end{proof}

\section{Comparison results:  Proof of Theorem \ref{short_long_thm}}
\label{subsec:equiv}

 The proof of Theorem \ref{short_long_thm}, given at the end of the section, makes use of two preliminary results which we first argue.
 In particular, Lemma \ref{sector_condition} states a type of sector condition, perhaps useful in other problems, for long-range models when $\alpha>2$ and $d\geq 1$.

Denote by $\| \cdot \|_{\pm 1, (FA)}$, $\|\cdot\|_{\pm 1, (FA-NN)}$ and $\| \cdot \|_{\pm 1, (NNA)}$ the ${\mc H}_{\pm 1}$-norms defined in terms of ${\mc S}^{(FA)}$, ${\mc S}^{(FA-NN)}$ and ${\mc S}^{(NNA)}$ respectively,
and similar expressions with respect to $(FA-NN)$ and $(NNA)$ generators.

\begin{lemma} \label{norms_comparison}
For $\alpha>2$, and $d\geq 1$, there exist constants $C=C(p,d),D=D(p,d)>0$ such that on local functions $\varphi$,
\begin{eqnarray}   \label{usefulestimate}
C^{-1} \; \| \varphi \|_{1,(FA)}^2 &\le& \| \varphi\|_1^2 \ \; \le \ \; C \| \varphi\|_{1,(FA)}^2\\
D^{-1} \; \| \varphi \|_{-1,(FA)}^2 &\le &\| \varphi\|_{-1}^2 \ \; \le\  \; D\| \varphi\|_{-1,(FA)}^2.\nonumber
\end{eqnarray}

\end{lemma}

\begin{proof}  The second display follows from first in \eqref{usefulestimate} and the definition of $H_{-1}$ norms.

To prove \eqref{usefulestimate}, we now give a reduction:  By irreducibility of $s^{(FA)}$, Lemma 3.7 in \cite{S2} states that $\|\cdot \|_{\pm, (FA)}$ and $\|\cdot\|_{\pm, (FA-NN)}$ are equivalent.  Hence, we need only to show \eqref{usefulestimate} with respect to $p^{(FA-NN)}$.

Recall, the Dirichlet form
$\| \varphi \|_{1}^2=\sum_{x,y\in\mathbb{Z}^d}s(y)D_{x,x+y}(\varphi)$.  Similarly,
$ \| \varphi \|_{1,(FA-NN)}^2= \sum_{x\in\mathbb{Z}^d}\sum_{i=1}^d s^{(FA-NN)}(e_i)D_{x,x+e_i}(\varphi)$.
Here, for $u,v\in{\mathbb{Z}^d}$,
$D_{u,v}(\varphi)=\E_\rho (\varphi(\eta^{u,v})-\varphi(\eta))^2$.

We now argue in $d=1$, and remark later on modifications to $d\geq 2$.
The left inequality in \eqref{usefulestimate} is trivial since $s^{(FA-NN)}(1)=2^{-1}$, $s(1)=c2^{-1}(b_1^++b_1^-)>0$ and so $\| \varphi \|_{1}^2\geq \frac{s(1)}{s^{(FR-NN)}(1)} \| \varphi \|_{1,(FR-NN)}^2$.

For the right inequality in \eqref{usefulestimate}, consider the bond $(x,x+y)$ for $y>0$.  Rewrite $\eta^{x,x+y}$ as a series of nearest-neighbor exchanges.  One exchanges in succession the values on bonds $(x, x+1)$, $(x+1, x+2)$ and so on to bond $(x+y-1,x+y)$.  In this way, the value at $x$ is now at $x+y$.  Exchange now on bonds $(x+y-1, x+y-2)$, and so on to $(x,x+1)$.  This puts the value initially at $x+y$ at $x$, also shifts back the values at intermediate points to their initial states.  The Dirichlet bond $D_{x,x+y}(\varphi)$, by invariance of $\nu_\rho$, by adding and subtracting $2y -1$ terms and Schwarz inequality is bounded
$$D_{x,x+y}(\varphi) \ \leq \ 2y\sum_{z=x}^{x+y-1} D_{z,z+1}(\varphi).$$

Since $\alpha>2$, we have $\sum y^2s(y)<\infty$ and
\begin{eqnarray*}
\|\varphi\|_1^1 & \leq & \sum_y 2ys(y)\sum_x\sum_{z=x}^{x+y-1}D_{z,z+1}(\varphi) \\
&\leq & \big(\sum_y 2y^2s(y)\big)\sum_x D_{x,x+1}(\varphi) \ \leq \ s^{(FR-NN)}(1)^{-1} \big(\sum_y 2y^2s(y)\big) \|\varphi\|^2_{1, FR-NN}.
\end{eqnarray*}

In $d\geq 2$, the proof of the left inequality in \eqref{usefulestimate} is similar, as $s^{(FA-NN)}(e_i), s(e_i)>0$ for $1\leq i\leq d$.  For the right inequality, an exchange over the bond $(x,x+y)$ is decomposed by nearest-neighbor exchanges first on bonds $(x,x+e_1)$ to $((x_1+y_1-1,x_2), (x_1+y_1,x_2))$, and then from $((x_1+y_1,x_2), (x_1+y_1,x_2+1))$ to $x+y$.  Then, as before in the $d=1$ argument, exchanges are made on the vertical and horizontal lines to bring the value at $x+y$ to $x$, and shift back other values.  The analysis is now analogous with more notation (cf. Appendix 3.3 in \cite{K.L.}).
\end{proof}

Now, suppose $m =\sum yp(y)$ is such that $m_i\geq 0$ for $1\leq i\leq d$. Let
${\mc L}^1$ be the generator of a nearest-neighbor finite range (FA-NN) exclusion process where
\begin{eqnarray*}
({\mc L}^1 f)(\eta) &=& \sum_{z\in\ZZ^d} \Big[(2m_1 -1)\eta(z+e_1)(1-\eta(z))\nabla_{z,z+e_1}f(\eta) \\
&& \ \ \ \ \ + \sum_{i=2}^d (2m_i)\eta(z+e_i)(1-\eta(z))\nabla_{z,z+e_i}f(\eta)\Big].\end{eqnarray*}

\begin{lemma}\label{sector_condition}
Suppose $\alpha>2$, $d\geq 1$ and consider the exclusion process generated by $ \tilde{{\mc  L}}={\mc L} +{\mc L}^1$.  Then, $\tilde{\mc L}$ satisfies a sector condition: There exists
 a constant $C=C(p,d)$ such that on local functions $\varphi, \psi:\Omega\rightarrow{\mathbb{R}}$ we have
\begin{equation*}
\langle (-\tilde{{\mc L}})\varphi,  \psi \rangle_{\rho} \ \le\  C\,  \| \varphi\|_{1, (FA-NN)}\,  \| \psi \|_{1, (FA-NN)}.
\end{equation*}
\end{lemma}

We remark this lemma is a type of generalization, to the long-range setting, of the finite-range sector inequality in Lemma 5.2 of \cite{V}:
\begin{equation}
\label{finiterangesector}
\langle (-\widehat{{\mc L}})\varphi,  \psi \rangle_{\rho} \ \le\  C\,  \| \varphi\|_{1, (FA)}\,  \| \psi \|_{1, (FA)}
\end{equation}
where $\widehat{{\mc L}}$ is the generator of a finite-range mean-zero exclusion process.

\begin{proof}
We work in $d=1$, but a similar but more notationally involved argument, decomposing a jump from $x$ to $x+y$ into nearest-neighbor jumps parallel to axes, as in the proof of Lemma \ref{norms_comparison}, yields the proof for  $d\geq 2$.

First, we notice that $\tilde{{\mc  L}}$ generates a process which can be decomposed into certain `looping' operators $\tilde{\mathcal{L}}_k$:
\begin{equation}
\tilde{{\mc  L}}\varphi(\eta)=\sum_{k\in{\mathbb{Z}}}\sum_{x\in\ZZ}\tilde{\mathcal{L}}^x_k\varphi(\eta),
\end{equation}
where, for $k>0$,
\begin{equation*}
\tilde{{\mc L}}^x_k\varphi(\eta) =  p(k)\Big\{\eta(x)(1-\eta(x+k))\nabla_{x,x+k}\varphi(\eta)+2\sum_{y=x}^{x+k-1}
\eta(y+ 1)(1-\eta(y))\nabla_{y,y+ 1}\varphi(\eta)\Big\}
\end{equation*}
and for $k<0$ the last term in the curly braces above takes the form $2\sum_{y=x+k+1}^{x}\eta(y-1)(1-\eta(y))\nabla_{y,y+1}\varphi(\eta)$.

We now fix $k>0$--the arguments in the case when $k<0$ will be analogous.
Define the following $\nu_\rho$-measure-preserving transformations:  Let $T_j:\Omega \rightarrow \Omega$ be given by
\begin{eqnarray*}
T_0\eta &=& \eta^{x, x+k}\\
 T_j\eta &=& \eta^{x+k-j+1, x+k-j} \quad \quad \quad {\rm  \ \ for\ } 1\leq j\leq k\\
 T_{j} &=& T_{2k-j} \quad \quad \quad \quad \quad \quad \quad {\rm \ \ for\ } k+1\leq j\leq 2k-1.
 \end{eqnarray*}
   Define also the sets
   \begin{eqnarray*}
   E_0 &=& \{\eta(x)=1\}\\
   E_j &=& \{\eta(x+k-j+1)=1\} \quad {\rm   for \ \ } 1\leq j\leq k\\
   E_{j} &=& E_{2k-j} \quad \quad \quad \quad \quad \quad\quad\quad{\rm for \ } k+1\leq j\leq 2k-1.
   \end{eqnarray*}
    Then, since $\eta(z)(1-\eta(w))\nabla_{z,w}\varphi(\eta) = \eta(z)\nabla_{z,w}\varphi(\eta)$, a moment's thought gives
\begin{eqnarray*}\tilde{{\mc L}}^x_k\varphi(\eta) & = &{p}(k)\Big\{\eta(x)\nabla_{x,x+k}\varphi(\eta)+\eta(x+1)\nabla_{x,x+1}\varphi(\eta)\\
&&\ \ \ \ \ \ \ \ \ \ \ \ \ \ \ \ \ \ \  + 2\sum_{y=x+1}^{x+k-1}
\eta(y+1)\nabla_{y,y+1}\varphi(\eta)\Big\}\\
&=& p(k)\sum_{j=0}^{2k-1} 1_{E_j}(\eta)\Big\{\varphi(T_j\eta)- \varphi(\eta)\Big\}.\end{eqnarray*}

Now, $E_j = T_{j-1}\cdots T_0E_0$ for $1\leq j\leq 2k-1$.  Also, noting the order of the exchanges $\{T_k\}$, we take a particle from $x$ to $x+k$, bring it back to $x$ in $k$ nearest-neighbor steps, and then put back the offset values in $k-1$ nearest-neighbor exchanges. Then, $T_{2k-1}\cdots T_0 = I$, the identity map.
For local functions $\varphi$ and $\psi$, we have
\begin{eqnarray*}
&&\langle(-\tilde{\mathcal{L}}_k^x)\varphi,  \psi \rangle_{\rho} \ = \ p(k)\E_\rho \Big[\sum_{j=0}^{2k-1}1_{E_j}(\eta)\big[\varphi(T_j\eta) - \varphi(\eta)\big]\psi(\eta)\Big]\\
&&\ \ \ =\  p(k)\E_\rho \Big[\sum_{j=0}^{2k-1}1_{E_0}(\eta)\big\{\varphi(T_j\cdots T_0\eta) - \varphi(T_{j-1}\cdots T_0\eta)\big\}\psi(T_{j-1}\cdots T_0\eta)\Big],
\end{eqnarray*}
after changing variables, with convention $T_{-1}\cdots T_0 = I$.  As the sum over the curly brackets telescopes in the last display and vanishes for each $\eta$, we may subtract $\psi(\eta)$ to obtain
\begin{eqnarray*}
\langle(-\tilde{\mathcal{L}}_k^x)\varphi,  \psi \rangle_{\rho} & = & p(k)\E_\rho \Big[\sum_{j=0}^{2k-1}1_{E_0}(\eta)\big\{\varphi(T_j\cdots T_0\eta) - \varphi(T_{j-1}\cdots T_0\eta)\big\}\\
&&\ \ \ \ \ \ \ \ \ \ \ \ \ \ \ \ \ \ \ \cdot \big\{\psi(T_{j-1}\cdots T_0\eta)-\psi(\eta)\big\}\Big].
\end{eqnarray*}

The last display is rewritten, noting the $j=0$ term vanishes, as
\begin{eqnarray*}
&& p(k)\E_\rho \Big[\sum_{j=2}^{2k-1}1_{E_0}(\eta)\big\{\varphi(T_j\cdots T_0\eta) - \varphi(T_{j-1}\cdots T_0\eta)\big\}\big\{\psi(T_{j-1}\cdots T_0\eta)-\psi(T_0\eta)\big\}\Big]\\
&&\ \ + \ p(k)\E_\rho \Big[\sum_{j=1}^{2k-1}1_{E_0}(\eta)\big\{\varphi(T_j\cdots T_0\eta) - \varphi(T_{j-1}\cdots T_0\eta)\big\}\big\{\psi(T_0\eta)-\psi(\eta)\big\}\Big] \\
&&\ \  \ \ \ \ \ =:\  A + B.
\end{eqnarray*}

In the following, $C$ is a constant which may change line to line.  From Schwarz's inequality, changing variables,
$$\sum_{j=2}^{2k-1} \E_\rho\big\{\psi(T_{j-1}\cdots T_0\eta)-\psi(T_0\eta)\big\}^2 \ \leq \ Ck \sum_{j=2}^{2k-1} \sum_{l = 1}^{j-1} \E_\rho\big\{\psi(T_l\eta) - \psi(\eta)\big\}^2.$$
Note that $T_l$ for $l\geq 1$ is a nearest-neighbor operation.
Hence, by another Schwarz's inequality, $\inf_{\varepsilon>0} \{a^2\varepsilon+b^2\varepsilon^{-1}\}=2ab$, $1_{E_0}\leq 1$, and counting nearest-neighbor bonds, we have after summing on $x$, that $\sum_x |A|$ is bounded
\begin{eqnarray*}
\sum_x |A|
& \leq &
 Ck^3p(k)\varepsilon_k  \|\psi\|^2_{1, (FA-NN)} + Ckp(k)\varepsilon^{-1}_k \|\varphi\|_{1, (FA-NN)}^2.
 \end{eqnarray*}
 Taking $\varepsilon_k = \varepsilon k^{-1}$, summing on $k$, and optimizing on $\varepsilon$, we have
 $$\sum_k \sum_x |A|
 \ \leq \ C\big[\sum_k k^{2}p(k)\big] \|\varphi\|_{1, (FA-NN)}\|\psi\|_{1, (FA-NN)}.$$

 Similarly, since $\E_\rho\big\{\psi(T_0\eta) - \psi(\eta)\big\}^2 = \E_\rho\big\{\psi(\eta^{x,x+k}) - \psi(\eta)\}^2$,
 \begin{eqnarray*}
 \sum_k\sum_x |B| & \leq & C\varepsilon\sum_{k,x} k^2 p(k) \|\varphi\|^2_{1, (FA-NN)} + 2\varepsilon^{-1} \sum_{k,x} p(k) \E_\rho\big\{\psi(T_0\eta) - \psi(\eta)\big\}^2\\
 &=& C\varepsilon\sum_k k^2 p(k) \|\varphi\|^2_{1, (FA-NN)} + 2\varepsilon^{-1} \|\psi\|^2_{1}\\
 &  \leq &C\big[\sum_k k^2p(k)\big]^{1/2}\|\varphi\|_{1, (FA-NN)}\|\psi\|_1.
 \end{eqnarray*}
  By Lemma \ref{norms_comparison}, $\|\psi\|^2_{1} \leq C\|\psi\|^2_{1, (FA-NN)}$.

Finally, combining bounds on $A$ and $B$, we get the desired estimate for $\langle (-\tilde{\mc L})\varphi, \psi\rangle_{\rho} = \sum_k\sum_x (A + B)$.
 \end{proof}

\medskip

\noindent{\it Proof of Theorem \ref{short_long_thm}.}    For local functions $f$, we first compare $L_f(\lambda)$ with  $L^{(FA-NN)}_f(\lambda)$ for $\lambda>0$.  Recall from Lemma \ref{sunder equivalence} that
$$L_f(\lambda) \ = \ 2\lambda^{-2}\sup_\varphi \Big\{2\langle f,\varphi\rangle_\rho - \langle \varphi, (\lambda - {\mc S})\varphi\rangle_\rho - \langle {\mc A}\varphi, (\lambda - {\mc S})^{-1}{\mc A}\varphi\rangle_\rho\Big\}.$$

Consider $\tilde{{\mc  L}}={\mc L} +{\mc L}^1 = {\mc A} + {\mc A}^1 + {\mc S} + {\mc S}^1$, decomposed into anti-symmetric and symmetric parts.  Then, by the triangle inequality, with respect to the $\|\cdot\|_{-1}$ norm,
\begin{eqnarray*}
&&\langle {\mc A}\varphi, (\lambda - {\mc S})^{-1}{\mc A}\varphi\rangle_\rho\\
&&\ \ \leq \ 3\langle {\mc A_1}\varphi, (\lambda - {\mc S})^{-1}{\mc A_1}\varphi\rangle_\rho
+ 3\langle \tilde{\mc L}\varphi, (\lambda - {\mc S})^{-1}\tilde{\mc L}\varphi\rangle_\rho\\
&&\ \ \ \ \ \ \ +  3\langle[ {\mc S} + {\mc S}^1]\varphi, (\lambda - {\mc S})^{-1}[{\mc S} + {\mc S}^1]\varphi\rangle_\rho.
\end{eqnarray*}
The second inner product, by Lemmas \ref{norms_comparison}, \ref{sector_condition} and that $\|\cdot\|_1 \leq \|\cdot\|_{1,\lambda}$, is bounded
\begin{eqnarray*}
&&\langle \tilde{\mc L}\varphi, (\lambda - {\mc S})^{-1}\tilde{\mc L}\varphi\rangle_\rho \\
& &\ \ \ \ \leq \
C\|\varphi\|_{1, (FA-NN)} \<(\lambda - {\mc S})^{-1}{\mc L}\varphi, (-{\mc S}^1)(\lambda - {\mc S})^{-1}{\mc L}\varphi\>_\rho^{1/2}\\
&& \ \ \ \ \leq \  C\|\varphi\|_{1, (FA-NN)} \<(\lambda - {\mc S})^{-1}{\mc L}\varphi, (\lambda - {\mc S})(\lambda - {\mc S})^{-1}{\mc L}\varphi\rangle_\rho^{1/2}\\
&&\ \ \ \ = \  C\|\varphi\|_{1, (FA-NN)}\|\tilde{\mc L}\varphi\|_{-1,\lambda}
\end{eqnarray*}
where the constant $C$ may change every line.  Dividing through by $\|\tilde{\mc L}\varphi\|_{-1,\lambda}$, we obtain
$\|\tilde{\mc L}\varphi\|_{-1,\lambda} \leq C \|\varphi\|_{1,(FA-NN)}$.

Similarly, the third inner product, by Lemma \ref{norms_comparison} and definition of $\|\cdot\|_{-1,\lambda}$, is bounded
$$\langle[ {\mc S} + {\mc S}^1]\varphi, (\lambda - {\mc S})^{-1}[{\mc S} + {\mc S}^1]\varphi\rangle_\rho \ \leq \
C\langle \varphi, (\lambda - {\mc S}^1)\varphi\rangle_\rho.$$

Hence, substituting into the variational formula for $L_f(\lambda)$ and optimizing over $\varphi$, we have $L_f(\lambda) \geq C L_f^{(FA-NN)}(\lambda)$.  Similarly, we bound $L_f^{(FA-NN)}(\lambda) \geq C L_f(\lambda)$ with possibly a different constant $C$,  starting from the variational formula for $L_f^{(FA-NN)}(\lambda)$.

Finally, by similar arguments as above with the known `finite range' sector inequality \eqref{finiterangesector}, we conclude $L_f^{(FA-NN)} \approx L_f^{(FA)}$.
\qed

\section{Proof of results: Symmetric jumps}
\label{symmetric_section}

The proofs of Theorem \ref{th:adm-symmetric}, Theorem \ref{th:symd1} and Theorem \ref{th:symd1bis} are based on the self-duality property of the exclusion process, and follow from several computations.
On the other hand, the proof of Theorem \ref{th:conv-bw} follows the martingale approximation scheme in \cite{K}, \cite{Q.J.S.} and \cite{S} for the finite-range case.
Nevertheless, several estimates are different and require care because of the presence of the heavy tails of the probability $p(\cdot)=s(\cdot)$.  In the remainder of the section, we abbreviate $\theta_d = \theta_d(\cdot; s(\cdot))$ (cf. Subsection \ref{theta_section}).

\subsection{Proof of Theorem \ref{th:adm-symmetric}}

  By the basis decomposition in Subsection \ref{subsec:duality}, a local, mean-zero function can be written as
$$f \ = \ \sum_{n\geq 1}\sum_{|A|=n} {\mf f}(A)\Psi_A$$
where $A\subset {\mc E}$ and all sums are finite.
Let $n\ge 1$ be such that $\alpha \wedge 2 < nd$ and suppose $\dg(f)=n$.  By the remark \eqref{degree_decomp}, (1) if $n=1$, $\sum_{|A|=1}{\mf f}(A) \neq 0$; (2)  If $n=2$, $\sum_{|A|=1}{\mf f}(A) = 0$; (3) and if $n\geq 3$, $\sum_{|A|=1}{\mf f}(A) = \sum_{|A|=2}{\mf f}(A) =0$.

Note that  $\sum_{|A|=k}{\mf f}(A){\bf 1}_A$ is the dual form of $\sum_{|A|=k}{\mf f}(A)\Psi_A$ for $k\geq 1$.
To show $f$ is admissible, it is enough to show in case (1) that ${\bf 1}_A$ is admissible for all $A\in {\mc E}_1$; in case (2), it is enough to prove $\sum_{|A|=1}{\mf f}(A) {\bf 1}_A$  and ${\bf 1}_A$ for $|A|\geq 2$ are admissible; in case (3), we need to show $\sum_{|A|=1}{\mf f}(A){\bf 1}_A$, $\sum_{|A|=2}{\mf f}(A){\bf 1}_A$ and ${\bf 1}_A$ for $|A|\geq 3$ are admissible.

To show ${\bf 1}_A$ for $|A|\geq n$ is admissible, by Lemma \ref{seth bound}, we need only to bound $\|{\bf 1}_A\|_{-1,\lambda}$ uniformly as $\lambda\downarrow 0$.  By Lemma \ref{lem:afp}, it is sufficient to prove
\begin{equation}
\label{help_adm}
{\limsup}_{\lambda \to 0} \| {\mf W}_n {\widetilde {\bf 1}_{A}} \|_{-1, \lambda, {\rm free}} \ < \ \infty .
\end{equation}

Since
 the function $g={\mf W}_n {\widetilde {\bf 1}_A}=1$ when $\{x_1, \ldots,x_n\} =A$ and vanishes otherwise, its Fourier transform
 is bounded. Thus, expressing the ${\bb H}_{-1, \lambda, {\rm free}}$-norm in Fourier space (cf. (\ref{eq:h-1lfp})), the display \eqref{help_adm} follows if we show that
\begin{equation*}
{\limsup}_{\lambda \to 0} \int_{({\TT^d})^n} \cfrac{d {k}_1 \ldots d{k}_n}{\lambda + \theta_d (k_1) +\ldots \theta_d (k_n) } \ < \  \infty.
\end{equation*}

The divergence of this integral can only happen for $(k_1, \ldots,k_n)$ close to a point in ${\mc C}_d \times \ldots \times {\mc C}_d$.  It is straightforward to check that all divergences are the same as for $(k_1, \ldots, k_n)$ close to $(0,\ldots,0)$. Standard analysis, using Lemma \ref{lem:thetad}, shows the bound \eqref{help_adm}.

But, when $\sum_{|A|=\ell}{\mf f}(A) =0$, the square of the Fourier transform of ${\mf W}_\ell {\widetilde{\sum_{|A|=\ell}{\mf f}(A){\bf 1}_A}}$ diverges quadratically near points in $({\mc C}_d)^\ell$, for instance as $k_1^2 + \cdots + k_\ell^2$ near the origin.  Since at these points, by Lemma \ref{lem:thetad}, $\theta_d(k)$ diverges with smaller or equal order, the norm $\| {\mf W}_\ell {\widetilde{\sum_{|A|=\ell}{\mf f}(A){\bf 1}_A}}\|_{-1,\lambda, {\rm free}}$ converges as $\lambda \downarrow 0$.

Combining these estimates, we conclude $f$ is admissible in all cases. \qed

\subsection{Proof of Theorem \ref{th:symd1bis}}
Let $f(\eta)= (\eta(0) -\rho) (\eta (1) -\rho)= \chi (\rho) \Psi_{\{0,1\}}$ whose dual function ${\mf f} = \chi (\rho) {\bf 1}_{\{0,1\}}$.
By our assumption ${\mc L}={\mc S}$, Remark \ref{reduction_rmk} and that functions of degree strictly larger than $2$ are admissible by Theorem \ref{th:adm-symmetric}, and (\ref{eq:lapltrnasfotvar}), we need only show
\begin{equation*}
\langle f, (\lambda- {\mc L})^{-1} f \rangle_{\rho} \ =\ \langle f, (\lambda - {\mc S})^{-1}f\rangle_\rho \ = \ \|f \|_{-1,\lambda}^2 \ \approx \ |\log \lambda|.
\end{equation*}

Further, by Lemma \ref{lem:afp}, we need only to show this estimate with $\| f \|_{-1, \lambda}$ replaced by $\| {\mf W}_2 {\tilde{\mf f}} \|_{-1, \lambda, {\rm{free}}}$.
Observe, by \eqref{tilde_formula}, that $({\mf W}_2 {\tilde{\mf f}}) (x,y) = \chi(\rho) \left[ {\bf 1}_{x=0, y=1} +{\bf 1}_{x=1, y=0}\right]$ and its Fourier transform is $\chi(\rho)\left[e^{2\pi is_1} + e^{2\pi i s_2}\right]$. Then, by (\ref{eq:h-1lfp}), it is enough to show
\begin{equation*}
\int_{\TT^2} \frac{1}{\lambda +\theta_1 (s_1) +\theta_1 (s_2) } \, {ds_1\,  ds_2} \ \approx \ |\log \lambda|
\end{equation*}
as $\lambda \downarrow 0$.
This is straightforwardly accomplished using Lemma \ref{lem:thetad} and standard analysis.
 \qed

\subsection{Proof of Theorem \ref{th:symd1} }

By Remark \ref{reduction_rmk}, the lower order of variance for degree $2$ functions in Theorem \ref{th:symd1bis},  and admissibility of functions of at least degree $3$ in Theorem \ref{th:adm-symmetric}, we need only to consider $f(\eta) =\eta (0) -\rho$.  Recall from \eqref{eq:lapltrnasfotvar} that the Laplace transform $L_{f} (\cdot) $ of  $\sigma^2_f(t)$ is given by
$L_{f} (\lambda) = 2 \lambda^{-2} \langle f, (\lambda -{\mc L})^{-1} f \rangle_{\rho}$ which equals $2\lambda^{-2}\langle f, (\lambda - {\mc S})^{-1}f\rangle_\rho$ as ${\mc L} = {\mc S}$ by assumption.

Write $f= {\sqrt{\chi (\rho)}} \Psi_{\{0\}}\in H_1$ and consider its dual function ${\mf f}=\sqrt{\chi(\rho)}{\bf 1}_{\{0\}}\in {\mc H}_1$.
Identifying cardinality $1$ subsets of $\ZZ^d$ with points in $\ZZ^d$, we see that the generator ${\mf S}$ restricted to ${\mc H}_1$ is nothing but the generator of a random walk on ${\ZZ}^d$ with kernel $s$.
Then,
\begin{eqnarray*}
L_{f} (\lambda) &=&2 \chi (\rho) \lambda^{-2} (\lambda -{ \mf S})^{-1} (\{0\}, \{ 0 \})\\
& = & 2 \chi (\rho) \lambda^{-2}  \int_{\TT^d} \frac{du}{\lambda + \theta_d (u)}\\
&=& 2\chi(\rho)\lambda^{-2}\int_{0}^{\infty} e^{-\lambda t} \left[ \int_{\TT^d} e^{-\theta_d (u) t} du \right] \, dt
\end{eqnarray*}
using Fubini's Theorem for the last line.

After two integration by parts, we recover the variance
\begin{equation}
\label{variance_theta}
\sigma_t^2 (f) \ =\ 2 \chi (\rho) \int_{\TT^d}  \, \cfrac{ \theta_d (u) \, t -1 + e^{-\theta_d (u)t}}{\theta^2_d (u)} du.
\end{equation}
Now, by Lemma \ref{lemma-a1}, which analyzes \eqref{variance_theta} through standard analysis and Lemma \ref{lem:thetad}, we obtain Theorem \ref{th:symd1}.
\qed

\subsection{Proof of Theorem \ref{th:conv-bw} }

The functional CLT follows from a combination of arguments.  In particular, since the symmetric exclusion process starting from $\nu_\rho$ is reversible, part (i) follows from the Kipnis-Varadhan theorem \cite{KV}.  Also, the proof of part (iii) is the same as in Section 3.2 in Kipnis \cite{K} given the scalings in Theorem {\ref{th:symd1bis}}.

However, part (ii) is more involved as the long-range character of the process needs to be addressed.

\subsubsection{Proof of Theorem \ref{th:conv-bw}, ii)}
Let $f$ be a local function of degree $1$.  Again, by Remark \ref{reduction_rmk} and the lower order variance growth of degree $2$ or more functions in Theorem \ref{th:symd1}, it is enough to prove the result for the function $f(\eta)=\eta(0) -\rho$.  In the following, we denote $\bar\eta(x):= \eta(x)-\rho$.

Recall, the notation from the introduction, $a_N=\sigma_{N} (f)$. In order to show $A^{(N)}_t :=a^{-1}_N \Gamma_{f} (tN)$ converges in the uniform topology as $N\uparrow\infty$, it is sufficient to show tightness in the sup-norm, and that the finite-dimensional distributions converge. Tightness is established with the same argument as for Theorem 1.2 in \cite{S} with respect to the finite-range limit \eqref{BM_finiteCLT}.   Also, by the
 Markov property and scalings in Theorem \ref{th:sd-asymmetric-1}, convergence of finite-dimensional distributions to $\B (t)$ when $d=1, \alpha=1$ or $d=2, \alpha \ge 2$, $\B_{1-1/2\alpha}(t)$ when
  $d=1, 1 < \alpha <2$, and $\B_{3/4} (t)$ when $d=1, \alpha \ge 2$ follow from the convergence of the marginal
sequence $A^{(N)}_t$ to a Gaussian limit.  We now give a sketch how to obtain this marginal convergence.

\medskip

Let $T>0$ be fixed. Suppose there is a function $v^T_s$ such that for $s \in [0,T]$,
$$(\partial_s + \mc L) v^T_s(\eta) = -\bar \eta_s (0)$$
and $v_T^T =0$. Then, by Dynkin's formula
$$\mc M^T_t = v^T_t(\eta_t) - v^T_0(\eta_0) - \int_0^t (\partial_s + \mc L) v^T_s(\eta_s)ds$$
is a centered martingale and
\begin{equation}
\label{eq:decompo}
\int_0^T \bar \eta _s(0) ds \ = \ v^T_0(\eta_0) + \mc M^T_T.
\end{equation}
Moreover, by the martingale property, $v^T_0(\eta_0)$ and $\mc M^T_T$ are uncorrelated since ${\mc M}_0^T=0$.  Then, $a^2_T=\E_\rho[\Gamma^2_{f}(T)]$ is the sum of the variances of these terms.
Define the limiting variances, assuming they converge,
\begin{eqnarray*}
\sigma^2_{1,T}\ :=\ \lim_{N\rightarrow \infty} \E_\rho\Big( \frac{1}{a_{N}}  \, v^{TN}_0(\eta_0)\Big)^2\ \ {\rm and \ \ }
\sigma^2_{2,T} \ :=\ \lim_{N\rightarrow\infty} \E_\rho\Big(  \frac{1}{a_{N}}  \mc M_{TN}^{TN}\Big)^2.
\end{eqnarray*}

Write
\begin{eqnarray*}
&&\Big| \E_\rho\Big[e^{itA^{(N)}_T} - e^{-\frac{t^2}{2}(\sigma^2_{1,T} + \sigma^2_{2,T})}\Big]\Big|\\
&&\ \ \ \leq \ \E_\rho\Big|\E_{\eta(0)}\Big[e^{\frac{it}{a_{N}}  \, M_{TN}^{TN}} - e^{-\frac{t^2}{2}\sigma^2_{2,T}}\Big]\Big|
 +  \Big|\E_\rho\Big[ e^{\frac{it}{a_{N}}  \, v^{TN}_0(\eta_0)} - e^{-\frac{t^2}{2}\sigma^2_{1,T}}\Big]\Big|.
\end{eqnarray*}
Later, in Lemmas \ref{lem:time-zeroterm} and \ref{lem:martingaleterm}, we show $\sigma^2_{1,T}$ and $\sigma^2_{2,T}$ indeed converge, and that the first and second terms above vanish, finishing the marginal convergence argument.

\medskip
To make rigorous this sketch, we first establish the martingale decomposition (\ref{eq:decompo}). Let $p_t (y)$ be the continuous-time
transition probability
 of the random walk on $\bb Z ^d$, starting at the origin, with translation-invariant symmetric rates $p(x,x+y):= p(y)=s(y)$.  Define
$$u_t(x) \ = \ \int_0^t p_s(x) ds,$$
the Green's function, which satisfies
$$\partial_t u_t \ = \ \Delta u_t + \delta_0$$
where $\Delta$ is the generator of the random walk, $\Delta f(x) = \sum_{y\in \mathbb{Z}^d} p(y) (f(x+y)-f(x))$.

We now verify that
$U_t^T(\eta):= v^T_t(\eta)$ where
$$
U_t^T(\eta) \: = \ \sum_{x\in \mathbb{Z}^d} u_{T-t}(x) \bar \eta (x).$$
Indeed, write
\begin{eqnarray*}
\partial_s U^T_s &=& - \sum_{x\neq 0} \Delta u_{T-s}(x)\bar \eta (x)- (\Delta u_{T-s}(0) + 1)\bar \eta (0)\\
&=& -\sum_{x\in \mathbb{Z}^d} \Delta u_{T-s}(x) \bar \eta (x)- \bar \eta (0) \ \; = \ \; - {\mc L} U_s^T - \bar \eta (0)
\end{eqnarray*}
noting $U_t(\eta^{x,x+y}) - U_t(\eta) = \big(u_{T-t}(x+y) - u_{T-t}(x)\big)\big(\eta(x) - \eta(x+y)\big)$, $p(\cdot)=s(\cdot)$ and
\begin{eqnarray}
\label{L_U_action}
\mc L U^T_t(\eta) &=& \sum_{x,y\in \mathbb{Z}^d}p(y) \eta(x)(1-\eta(x+y))\big (u_{T-t}(x+y)-u_{T-t}(x)\big)\\
&=& \sum_{x,y\in \ZZ^d}p(y)  \big(u_{T-t}(x+y) - u_{T-t}(x)\big)\eta(x).\nonumber
\end{eqnarray}

Observe that $U_T^T(\eta) \equiv 0$,
since $u_0(x)=0$ for all $x\in\mathbb{Z}^d$.  Hence,  (\ref{eq:decompo}) follows and
\begin{eqnarray*}
\int_0^T \bar\eta_s(0) ds &=& U_0^T(\eta_0) + \mc M _T^T.
\end{eqnarray*}

\begin{lemma}
\label{lem:time-zeroterm}
We have
\begin{equation}
\label{eq:zeroterm}
\frac{1}{a_N} \;U_0^{NT}(\eta_0)\ :=\ \frac{1}{a_N} \sum_{x \in \ZZ^d} u_{NT}(x) \bar\eta_0(x)
\end{equation}
converges weakly as $N\uparrow \infty$ to a centered Normal variable with limiting variance $\sigma^2_{1,T}$.
When $0<\alpha\leq 1$ in $d=1$ or $\alpha\geq 2$ in $d=2$, $\sigma_{1,T}^2= 0$.  But, for $\alpha>1$ in $d=1$, $0<\sigma^2_{1,T}<\infty$.
\end{lemma}

\begin{proof}
The Fourier transform of $u_{t} (\cdot)$ is given by
\begin{equation*}
{\hat u} _t (k) \ =\  \int_0^t e^{-(1-{\hat p} (k)) s } ds
\end{equation*}
for $k \in \TT^d$ where ${\hat p}(k) = \sum_{y \in \ZZ^d} p(y) e^{2 \pi ik \cdot y} $ is the Fourier transform of $p(\cdot) = s(\cdot)$.
By symmetry of $s(\cdot)$, the fact that $1- \cos (2\pi k \cdot y) =2 \sin^2 (\pi k \cdot y)$, and definition of $\theta_d$ in \eqref{theta function}, we have
\begin{equation*}
1- {\hat p}(k) \ =\  2 \sum_{y \in \ZZ^d } s(y) \sin^2 (\pi k \cdot y) \ =\  \theta_d (k).
\end{equation*}
Thus, we obtain
\begin{equation}
\label{hat_u}
\hat{u}_t(k)\ = \ \cfrac{1- e^{-\theta_d (k) t} }{\theta_d (k)}\end{equation}
and as a consequence
\begin{equation}
\label{eq:utx}
u_t (x) \ =\
\int_{\TT^d} e^{-2i \pi k \cdot x} \left[ \cfrac{1- e^{-\theta_d (k) t} }{\theta_d (k)} \right] dk.
\end{equation}

By Parseval's relation, $\E_\rho[(\eta(x) - \rho)^2] = \rho(1-\rho) = \chi(\rho)$, and the equation for $a_N^2=\sigma^2_N(f_0)$ in \eqref{variance_theta}, the variance of $a^{-1}_N U_0^{NT}(\eta_0)$ under $\nu_\rho$ is equal to
\begin{eqnarray}
\label{zero_variance}
 \frac{\chi(\rho)}{a^2_N}  \sum_{x \in \ZZ^d}  |{u}_{TN}|^2 (x) & = & \chi(\rho) \int_{\TT^d} {\left[ \cfrac{1- e^{-\theta_d (k) TN}}{\theta_d (k) }\right]^2 dk} \nonumber \\
 &&\ \ \ \cdot \Big[2\chi(\rho)\int_{\TT^d}\frac{\theta_d(u)N - 1 + e^{-\theta_d(u)N}}{\theta^2_d(u)}du\Big]^{-1}.
\end{eqnarray}

\begin{enumerate}[i)]
\item If $d=1$ and $\alpha=1$, by the scaling relation $a^2_N\sim N\log(N)$, $\theta_1(k)\sim |k|$ (cf. Lemma \ref{lem:thetad}), and simple computation, the variance \eqref{zero_variance} vanishes as $N\uparrow\infty$.
Therefore, (\ref{eq:zeroterm}) converges in distribution to the Dirac mass centered at $0$.

\item If $d=1$ and $1<\alpha<2$, recall $a_N \sim N^{1-1/2\alpha}$. By (\ref{hat_u}), we have
\begin{eqnarray*}
\sum_{x \in \ZZ} |u_t (x)|^2 &= & \int_{0}^1 \left[ \cfrac{1- e^{-\theta_1 (k) t} }{\theta_1 (k)} \right]^2 dk\\
&=&2 \int_{0}^{1/2} \left[ \cfrac{1- e^{-\theta_1 (k) t} }{\theta_1 (k)} \right]^2 dk \ =\  2t^{2-1/\alpha} \int_{0}^{t^{1/\alpha} /2} \left[ \cfrac{1- e^{-t \theta_1 (\ell t^{-1/\alpha}) } }{t\theta_1 (\ell t^{-1/\alpha})} \right]^2 d\ell.
\end{eqnarray*}
By Lemma \ref{lem:thetad} and dominated convergence, we have, as $t\uparrow\infty$,
\begin{equation}
\label{eq:utx2}
\sum_{x \in \ZZ} |u_t (x)|^2 \ \sim \ 2t^{2- 1/\alpha}  \int_0^{\infty} \left[ \frac{1-e^{-a_1 (\alpha) \ell^{\alpha}}}{a_1 (\alpha) \ell^{\alpha} }\right]^2 d\ell,
\end{equation}
where the constant $a_1 (\alpha)$ is such that $\theta_1 (k) \sim a_1 (\alpha) |k|^{\alpha}$ as $k\downarrow 0$. A similar argument shows, for $x \in \ZZ$ and $t>0$, that
\begin{equation}
\label{utx3}
|u_{t} (x) | \ \le \  \int_{0}^{1} \left| \cfrac{1- e^{-\theta_1 (k) t} }{\theta_1 (k)} \right| dk \ = \ O \left( t^{1-1/\alpha}\right).
\end{equation}
Also, by the same type of analysis, one concludes that $\sigma^2_{1, T}$, the limit of \eqref{zero_variance} as $N\uparrow\infty$, converges.

Now, for $\beta \in \RR$, we have
\begin{eqnarray*}
&&\log \left[ \int d\nu_{\rho} (\eta) \, \exp \left( \frac{i \beta}{a_N } \sum_{x \in \ZZ} u_{NT} (x) {\bar \eta} (x) \right)\right]\\
&&\ \ \  = \ \log  \left[ \prod_{x \in \ZZ} \int d\nu_{\rho} (\eta) \, \exp \left( \frac{i \beta}{a_N } u_{NT} (x) {\bar \eta} (x) \right)\right] \\
&& \ \ \ = \ \log \left[ \prod_{x \in \ZZ} \left[ 1 -\frac{\beta^2}{2 a_{N}^2 } u_{NT} (x)^2 + O(a_N^{-3}  |u_{NT} (x)|^3 )\right]\right].
\end{eqnarray*}
Since $\sum_{x} |u_{NT} (x)|^3 \le (\sum_{x} |u_{NT} (x)|^2 ) \sup_{x} |u_{NT} (x)| = O(a_N^2  N^{1-1/\alpha})$ and $e^{-z} = 1-z + O(z^2)$ as $|z|\downarrow 0$, by (\ref{eq:utx2}) and \eqref{utx3}, we get
\begin{equation*}
\lim_{N\rightarrow\infty} \int d\nu_{\rho} (\eta) \, \exp \left( \frac{i \beta}{a_N } \sum_{x \in \ZZ} u_{NT} (x) {\bar \eta} (x) \right) \ = \ \exp \left(- \sigma_{1,T}^2  \beta^2 /2\right).
\end{equation*}

\item If $d=1$ and $\alpha > 2$, the argument is similar to the case when $1<\alpha<2$.  If $\alpha =2$, using the substitution $k = \beta_t u$ with $t\beta^2_t|\log \beta_t| = 1$ and $\beta_t =O((t\log(t))^{-1/2})$, the proof is also analogous.

\item If $d=2$ and $\alpha \geq 2$, as when $d=1$ and $\alpha =1$, noting the scaling relation for $a^2_N$ in Theorem \ref{th:symd1} and that $\theta_d(k)\sim |k|^2$ for $\alpha>2$ and $\theta_d(k)\sim |k|^2\log(|k|)$ for $\alpha=2$ by Lemma \ref{lem:thetad}, the limit of the variance in \eqref{zero_variance} vanishes and \eqref{eq:zeroterm} converges to the Dirac mass at $0$.

\end{enumerate}
\end{proof}

\begin{lemma}
\label{lem:martingaleterm}
For any fixed $T>0$, the limiting variance satisfies $0<\sigma^2_{2,T}<\infty$ and
\begin{equation}
\label{eq:martingaleterm}
\lim_{N\rightarrow \infty}\E_\rho\Big|\E_{\eta_0}\Big[\frac{1}{a_N (T)} {\mc M}_{TN}^{TN} - e^{-\frac{t^2}{2}\sigma^2_{2,T}}\Big]\Big| \ =\  0.
\end{equation}
\end{lemma}

\begin{proof}
Although $U^T_s$ is not a local function, by standard approximations, the quadratic variation of the martingale ${\mc M}_{T}^{s}$ is $\int_0^t {\mc L}(U^T_s)^2 - 2U^T_s{\mc L}U^T_s ds$.  Recalling $p(\cdot) = s(\cdot)$, the integrand may be computed as
$${\mc L}(U^T_s)^2 - 2U^T_s{\mc L}U^T_s\ = \ \sum_{x,y\in\ZZ^d}p(y-x)\big(u_{T-s}(y) - u_{T-s}(x)\big)^2\eta_s(x)\big(1-\eta_s(y)\big).$$
Hence, the variance $\sigma^2_{2,T}$ is given by
\begin{eqnarray}
\label{form_of_sigma}
\lim_{N\uparrow\infty}\frac{1}{a^2_N}\E_\rho\big( {\mc M}^{TN}_{TN}\big)^2 &=& \lim_{N\uparrow\infty}\frac{\rho(1-\rho)}{a^2_N}\int_0^{TN}\sum_{x,y\in\ZZ^d}p(y-x)\big(u_{TN-s}(y) - u_{TN-s}(x)\big)^2  ds\nonumber\\
&=& \lim_{N\uparrow\infty}\frac{2\rho(1-\rho)}{a^2_N} \int_0^{NT} \int_{\TT^d}\theta_d(k)|\hat{u}_{TN -s}(k)|^2dk ds
\end{eqnarray}
using a form of Parseval's relation:  The random walk Dirichlet form
$$\frac{1}{2}\sum_{x,y\in\ZZ^d}p(y-x)\big(u_{TN-s}(y) - u_{TN-s}(x)\big)^2  =  -\langle u_{TN-s},\Delta u_{TN-s}\rangle  = \int_{\TT^d}\theta_d(k)|\hat u_{TN-s}(k)|^2dk.$$
  Then, the limit converges to a positive quantity, noting the explicit form of $\hat{u}_t$ in \eqref{hat_u}, Lemma \ref{lem:thetad}, and the asymptotics of $a_N$ (cf. Theorem \ref{th:symd1}), and from standard analysis as used in the proof of Lemma \ref{lem:time-zeroterm}.

Now, by Feynman-Kac's formula, for $\beta \in \RR$, the process
\begin{equation*}
{\mc N}_t^{T,\beta} \ =\  \exp \left\{ i \beta U_t^T (\eta_t) - i \beta U_0^T (\eta_0) - \int_0^t e^{- i \beta U_s^T (\eta_s)} (\partial_s +{\mc L}) e^{i \beta U_s^T (\eta_s)} ds  \right\},
\end{equation*}
for $0\leq t\leq T$, is a martingale with expectation $1$. By the form of $U^T$ and \eqref{L_U_action},  we have
\begin{equation*}
e^{- i \beta U_s^T (\eta_s)} (\partial_s +{\mc L}) e^{i \beta U_s^T (\eta_s)}
\ = \ i \beta (\partial_s + {\mc L}) U_s^T (\eta_s) + A(\beta, s,T)
\end{equation*}
with $A(\beta, s, T)$ equal to
\begin{equation*}
\sum_{x,y} p(y-x)  \left[ e^{i \beta (u_{T-s} (y) -u_{T-s} (x))} - i\beta (u_{T-s} (y) -u_{T-s} (x)) -1   \right]  \eta_s (x) (1-\eta_s (y)).
\end{equation*}
We have to show that
\begin{eqnarray*}
&&{\bb E}_{\rho}\left|\E_{\eta(0)} \left[\exp\left( i \beta \cfrac{{\mc M}_{TN}^{TN}}{a_N}\right) -\exp\left( -\sigma_{2,T}^2 \beta^2  /2)\right) \right]\right|\\
&&={\bb E}_{\rho} \left|\E_{\eta(0)} \left[ {\mc N}_{TN}^{TN,\beta/a_N } \; \left\{  \exp\left[ -\int_0^{NT} A\left( \frac{\beta}{a_N } , s,NT \right) ds \right] - \exp\left( -\sigma_{2,T}^2 \beta^2  /2)\right) \right\}  \right]\right|
\end{eqnarray*}
vanishes as $N\uparrow\infty$.

Note, for $x,t\in \RR$,
\begin{equation}
\label{complex_estimate}
|e^{i tx} - 1- it x +{x^2 t^2}/2| \ \le\  Ct^2 x^2 \min( 1, |t x|)
\end{equation}
and that $a^{-1}_N\sup_x|u_{NT-s}|(x)\rightarrow 0$ by \eqref{utx3}, $a_N$-asymptotics in Theorem \ref{th:symd1} and straightforward computations.

With this estimate, there exists a constant $C>0$ such that
\begin{eqnarray}
\label{eq:NTN}
\big|{\mc N}_{TN}^{TN, \beta/ a_N }\big| & \le &\exp\left[ \int_0^{NT} \left|\, A\left( \frac{\beta}{a_N } , s,NT \right) \right| ds\,  \right] \\
&\le& \exp \left\{ \frac{C\beta^2}{a^2_N } \int_0^{NT} \sum_{x,y} p(y-x) [ u_{NT-s}(y) -u_{NT -s} (x) ]^2 ds \right\}\nonumber\\
& =& \exp \left\{ \frac{C\beta^2}{2a^2_N } \int_0^{NT} \left( \int_{\TT^d} \theta_d (k) | \, {\hat u}_{NT -s} (k) \, |^2 \, dk  \right) ds \right\},\nonumber
\end{eqnarray}
where the second inequality comes from a Taylor expansion and the equality from the Parseval relation for the $\Delta$-Dirichlet form.

As the variance in \eqref{form_of_sigma} converges, the quantity $ \int_0^{NT} |A(a_N^{-1} \beta, s,NT) | ds$ and \eqref{eq:NTN} are uniformly bounded in $N$.
Therefore, things are reduced to show that
\begin{equation}
\label{eq:convprob}
\lim_{N \to \infty} \, \int_0^{NT} A\left( \frac{\beta}{a_N } , s,NT \right) ds \ =\  \cfrac{\sigma_{2,T}^2 \beta^2}{2}
\end{equation}
in probability under ${\PP}_{\rho}$.

Then, to prove (\ref{eq:convprob}), noting \eqref{complex_estimate}, it is sufficient to show, in probability, that
\begin{equation*}
\lim_{N \to \infty} \, \frac{1}{a_N^2 }  \int_0^{NT} \left[  \sum_{x,y\in\ZZ^d} b_N (s,x,y) \, \eta_s (x) (1-\eta_s (y) ) \right] ds  \ =\  \sigma_{2,T}^2
\end{equation*}
where
$$b_N (s,x,y) \ =\  p(y -x) (u_{NT -s} (y) -u_{NT-s} (x))^2.$$
This statement, by the form of $\sigma^{2,T}_2$ \eqref{form_of_sigma}, would follow if we can replace $\eta_s(s)(1-\eta_s(y))$ by $\rho(1-\rho)$ in $\LL^2(\P_\rho)$:
\begin{equation}
\label{eq:NTN24}
\lim_{N \to \infty} \, \frac{1}{a_N^2 }  \int_0^{NT} \left[  \sum_{x,y\in\ZZ^d} b_N(s,x,y)  \, \left\{ \eta_s (x) (1-\eta_s (y) ) - \rho (1-\rho) \right\} \right] ds  \ =\  0.
\end{equation}

To prove (\ref{eq:NTN24}), after squaring terms, since $\big(a_N^{-2}\int_0^{NT} \sum_{x,y}b_N(s,x,y)ds\big)^2$ converges in \eqref{form_of_sigma}, we need only show the covariance
$$\E_\rho\Big[\big\{ \eta_s (x) (1-\eta_s (y) ) - \rho (1-\rho) \big\}\big\{ \eta_u (z) (1-\eta_u (w) ) - \rho (1-\rho) \big\}\Big]$$
vanishes uniformly in $x,y,z,w$ as $|u-s|\uparrow\infty$.
As
$$\eta(\ell)(1-\eta(k)) - \rho^2 \ =  \ (1-\rho)(\eta(\ell)-\rho) -\rho(\eta(k)-\rho) - (\eta(\ell)-\rho)(\eta(k)-\rho),$$
by a calculation using the duality process decompositions in Subsection \ref{subsec:duality}, namely the symmetric semigroup action
$$T_t \prod_{i=1}^n (\eta(x_i)-\rho) \ =\  \sum_{|A|=n} p^{(n)}(\{x_1,\ldots, x_n\}, A)\prod_{y\in A}(\eta(y)-\rho),$$
 the covariance is
bounded by
$$C(\rho)\Big\{ p^{(1)}_{|u-s|}(x,z) + p^{(1)}_{|u-s|}(x,w) + p^{(1)}_{|u-s|}(y,z) + p^{(1)}_{|u-s|}(y,w) + p^{(2)}_{|u-s|}\big((x,y), (z,w)\big)\Big\}$$
where $p^{(n)}$ is the continuous-time transition probability of $n$ particles in symmetric simple exclusion for $n\geq 1$.  By Corollary VIII.1.9 in \cite{Li}, we have the bound
$$p^{(2)}_v\big((k_1,k_2), (\ell_1,\ell_2)\big) \ \leq\  C\sum_{i,j = 1}^2 p^{(1)}_v(k_i,\ell_j).$$
  As $p^{(1)}_v(k,\ell) = p^{(1)}_v(0,k-\ell)$, to show the covariance vanishes, we show $\lim_{v\uparrow\infty}p^{(1)}_v(0,k) = 0$ uniformly in $k$.

To this end, we bound $p^{(1)}_v(0,k)^2 = p^{(1)}_v(0,k)p^{(1)}_v(k,0) \leq p^{(1)}_{2v}(0,0)$ uniformly in $k$.  But,
$$p^{(1)}_v(0,0) \ = \ \int_0^1 e^{-v(1-\hat{p}(k))}dk\  =\  \int_0^1 e^{-v\theta_d(k)}dk.$$
Since for $\alpha\geq 1$, by Lemma \ref{lem:thetad}, $\theta_d(k) \geq C|k|^2$ near the zeroes of $\theta_d$, we have $p^{(1)}_v(0,0) \leq C'v^{-1/2}$, which shows the covariance vanishes uniformly.
\end{proof}

\section{Proof of results: Asymmetric jumps}
\label{asymmetric_section}

The proofs of the results for the asymmetric model rely on several ingredients, among them careful estimation of variational formulas for $L_f(\lambda)$, which we have partially prepared for in Subsection \ref{one_pt_section}, and several technical results collected in Appendix \ref{Appendix_A}.

\subsection{Proof of Theorem \ref{th:adm-asymmetric}}
\label{proof theo 2.7}

We first make a few reductions.  By Corollary \ref{cor seth bound}, the variance $\sigma^2_f(t) \leq 10 t^{-1} L^{(S)}_f(t^{-1})$.   Then, by Theorem \ref{th:adm-symmetric}, which bounds $L^{(S)}_f(\lambda)$, all statements in Theorem \ref{th:adm-asymmetric} follow modulo a few exceptions in $d\leq 2$.  In $d=1$, we still need to show  (a) admissibility when $\dg (f)=1$, $\alpha\in (1,2)\cup (2,\infty)$ and $\rho \ne 1/2$,
and (b) admissibility when $\dg (f)= 2$, $\alpha> 2$, $\rho \in [0,1]$.
In $d=2$, the case not obtained is (c) admissibility when $\dg(f)=1$, $\alpha\geq 2$ and $\rho \neq 1/2$.

When $\alpha>2$, by Lemma \ref{seth bound}, $\sigma^2_f(t) \leq 10t^{-1}L_f(t^{-1})$ and by Theorem \ref{short_long_thm}, $L_f(\lambda)\approx L^{(FR)}_f(\lambda)$ with respect to a jump probability $p^{(FR)}$ with a drift.  Also, by Proposition \ref{short_range_H_{-1}}, when $\rho \neq 1/2$, $\lambda^2L_f^{(FR)}(\lambda)$ is bounded as $\lambda\downarrow 0$ for all local $f$.  Hence, $\lambda^2L_f(\lambda)$ is also bounded and $\sigma^2_f(t)= O(t)$ when $\rho\neq 1/2$ in $d=1,2$, and so parts (a) and (c) in these cases also hold.  Also, by Proposition \ref{short_range_H_{-1}}, for local functions $f$ with degree $\dg(f)=2$, and any $0\leq \rho\leq 1$, we know $\lambda^2L^{(FR)}_f(\lambda)$ is bounded as $\lambda\downarrow 0$.  Therefore, $\lambda^2L_f(\lambda)$ is also bounded and $\sigma^2_f(t)=O(t)$, establishing part (b).

What remains then to conclude the proof of Theorem \ref{th:adm-asymmetric} is to show admissibility of degree one functions when
\begin{itemize}
\item[(A)] $\alpha \in (1,2)$, $d=1$ and $\rho \ne 1/2$, and
\item[(B)] $\alpha=2$, $d=2$ and $\rho \ne 1/2$.
\end{itemize}
By Remark \ref{reduction_rmk} and the already proven admissibility of functions of at least degree $2$ in these cases (A) and (B), it is sufficient to focus on the degree $1$ function $f(\eta)= \eta(0) -\rho$.

For the rest of the section, we remind that $\theta_d = \theta_d(\cdot; s_0(\cdot))$ (cf. Subsection \ref{theta_section}) in all the formulas.

\subsubsection{Proof of (A)}
\label{subsec:blablabla}
To prove $f(\eta)=\eta(0)-\rho$ is admissible, by Lemma \ref{seth bound}, we need to bound $\langle f, (t^{-1}-{\mc L})^{-1}f\rangle_\rho$.  Then, by  Lemma \ref{sunder equivalence}, using the inf form, to get an upper bound, we restrict the infimum to the set of functions $g$ of degree one.  By the estimate \eqref{useful estimate}, in the `free particle' formulation, we have
\begin{eqnarray*}
&&\inf_{{\text{g of degree one}}}\{\|\eta(0)-\rho+\mathcal{A}g\|_{-1,\lambda}^2+\|g\|_{1,\lambda}^2\}\\
&&\ \ \leq \
\inf_\varphi\{\|\delta_0+\mathfrak{T}_{1,1}\varphi\|_{-1,\lambda,{\rm{free}}}^2+\|\mathfrak{T}_{1,2}\varphi\|_{-1,\lambda,{\rm{free}}}^2+\|\varphi\|_{1,\lambda, {\rm{free}}}^2\},
\end{eqnarray*}
which is further expressed, in terms of the Fourier transform $\hat \varphi$, as
\begin{eqnarray}
\label{real_infimum}
&&\inf_{\hat\varphi}\Big\{\int_{0}^1\frac{|1+(1-2\rho)\hat a(u)\hat\varphi(u)|^2}{\lambda+\theta_1(u)}du+\int_{0}^1(\lambda+\theta_1(u))|\hat\varphi(u)|^2 du\nonumber\\
&&\ \ \ \ \ \ \ \ \ + \ \chi(\rho)^2\int_{0}^1|\hat\varphi(u)|^2\int_{0}^1\frac{|\hat a(s)+\hat a(u-s)|^2}{\lambda+\theta_1(s)+\theta_1(u-s)}ds \; du\Big\}.
\end{eqnarray}
We note, as $\varphi$ is real, ${\hat \varphi}$ is a complex function with even real part and odd imaginary part. The previous infimum is taken over this set of complex functions.

Now, for real numbers $b,c>0$ and $a\neq 0$, we observe
\begin{equation*}
\inf_{z\in{\mathbb{C}}} \Big\{\frac{|1+iaz|^2}{b}+c|z|^2\Big\}\ =\ \frac{1}{b+\frac{a^2}{c}}
\end{equation*}
and the infimum is realized at $z=i a/(bc+a^2)$. In our case, we have
\begin{eqnarray*}
ia&=&(1-2\rho)\hat{a}(u),\\
b&=&\lambda+\theta_1(u)\\
c&=&\lambda+\theta_1(u)+\chi(\rho)^2\int_0^1\cfrac{|\hat a (s)+\hat a(u-s)|^2}{\lambda+\theta_1(s)+\theta_1(u-s)}ds.
\end{eqnarray*}

 Then, the infimum \eqref{real_infimum} is realized for the function
\begin{equation*}
{\hat \varphi} (u)\ =\ -\cfrac{G^{(1)}_{\lambda,\rho} (u)}{(1-2\rho){\hat a} (u) \left[ \lambda+ \theta_1 (u) +G_{\lambda, \rho} (u)\right]},
\end{equation*}
where $G^{(1)}_{\lambda,\rho} $ is given by
\begin{equation*}
G^{(1)}_{\lambda,\rho}(u)\ =\ \frac{(1-2\rho)^2|\hat a(u)|^2}{\lambda+\theta_1(u)+\chi(\rho)^2\int_0^1\cfrac{|\hat a (s)+\hat a(u-s)|^2}{\lambda+\theta_1(s)+\theta_1(u-s)}ds}.
\end{equation*}
Noting that $G^{(1)}_{\lambda, \rho}$ is even, we see ${\hat \varphi} (u)$ has odd imaginary part and zero real part.

 Therefore, we obtain
the infimum \eqref{real_infimum} is equal to
\begin{equation*}
\int_{0}^1 \frac{1}{\lambda+\theta_1(u)+G^{(1)}_{\lambda,\rho}(u) }du.
\end{equation*}
 We split the above integral over $u$-regions $[0,\delta]$, $[\delta, 1-\delta]$ and $[1-\delta, 1]$ for $\delta>0$ small.  The contributions to the first and last regions are the same, while the integral over the middle region is $O(1)$ independent of $\lambda$ since $\theta_1$ vanishes only on ${\mc C}_1$.

By Lemma \ref{lemma on hat a}, $\sup_{s\in \TT}|\hat{a}(s) + \hat{a}(u-s)|^2 \leq Cu^2$.  Also, by Lemma \ref{lemma-a2}, for $1<\alpha<2$,
$$\int_{(0,\delta)\cup(1-\delta,1)} \frac{ds}{\lambda +\theta_1(s) + \theta_1(s-u)}ds \ \leq \ C_0(\lambda + u^\alpha/C_1)^{1/\alpha -1}.$$
On the other hand, $\int_\delta^{1-\delta}(\lambda +\theta_1(s) + \theta_1(s-u))^{-1}ds = O(1)$ not depending on $\lambda$.

Hence, there exist $\kappa_0,\kappa_1>0$ such that
for any
$0<u\leq \delta$
\begin{equation*}
G^{(1)}_{\lambda,\rho}(u)\ \geq\ {\frac{\kappa_0u^2}{\lambda+u^\alpha+u^2[(\lambda+\kappa_1 u^\alpha)^{1/\alpha-1}+1]}}.
\end{equation*}
Therefore
\begin{equation*}
\int_{0}^\delta\frac{du}{\lambda+\theta_1(u)+G^{(1)}_{\lambda,\rho}(u)}\ \leq\ {\int_{0}^{\delta}\frac{du}{\lambda+u^\alpha+\cfrac{\kappa_0u^2}
{\lambda+u^\alpha+u^2[1+(\lambda+\kappa_1 u^\alpha)^{1/\alpha-1}]}}}\ =:\ J(\lambda),
\end{equation*}
where
\begin{equation*}
\limsup_{\lambda\to 0}J(\lambda)\ =\ \int_{0}^\delta\frac{du}{u^\alpha+\cfrac{\kappa_0u^2}
{u^\alpha+u^2+\kappa_1^{1/\alpha-1} u^{3-\alpha}}}.
\end{equation*}

\begin{itemize}
\item  If $1<\alpha<3/2$, as $u\to 0$,
\begin{equation*}
u^\alpha+\cfrac{\kappa_0u^2}
{u^\alpha+u^2+\kappa_1^{1/\alpha-1} u^{3-\alpha}}\ \sim\ \kappa_0 u^{2-\alpha}
\end{equation*}
because $3-\alpha>\alpha$ and $\alpha>2-\alpha$.

\quad

\item  If $3/2<\alpha<2$, as $u\to 0$,
\begin{equation*}
u^\alpha+\cfrac{\kappa_0u^2}
{u^\alpha+u^2+\kappa_1^{1/\alpha-1} u^{3-\alpha}}\ \sim\ \frac{\kappa_0}{\kappa_1^{1/\alpha-1}} u^{\alpha-1}
\end{equation*}
because $3-\alpha<\alpha<2$ and $\alpha-1<\alpha$.

\quad

\item
 If $\alpha=3/2$, as $u\to 0$,
\begin{equation*}
u^\alpha+\cfrac{\kappa_0u^2}
{u^\alpha+u^2+\kappa_1^{1/\alpha-1} u^{3-\alpha}}\ \sim \frac{\kappa_0}{1+\kappa_1^{-1/3}} u^{1/2}.
\end{equation*}

\end{itemize}

In all these cases, $\limsup_{\lambda \to 0} J(\lambda)$ is finite, finishing the proof of part (A). \qed

\subsubsection{Proof of (B)}
We proceed as in Section \ref{subsec:blablabla}, and note that it suffices to show
\begin{equation}
\label{eq:nice}
\limsup_{\lambda \to 0} \int_{\TT^2} \cfrac{1}{\lambda + \theta_2 (u) + G^{(2)}_{\lambda,\rho} (u)} \, du \ < \  \infty
\end{equation}
where
\begin{equation*}
G^{(2)}_{\lambda,\rho} (u) =\frac{(1-2\rho)^2|\hat a(u)|^2}{\lambda+\theta_2 (u)+\chi(\rho)^2\int_{\TT^2} \cfrac{|\hat a (s)+\hat a(u-s)|^2}{\lambda+\theta_2(s)+\theta_2(u-s)}ds}.
\end{equation*}

We split the integral appearing in (\ref{eq:nice}) in five parts according to when $u$ is close to one of the four points in ${\mc C}_2$ or not. The integral corresponding to the exceptional region is bounded $O(1)$ independent of $\lambda$ as in part (A).
The four remaining integrals can all be treated similarly, and we restrict ourselves to the integral corresponding to the small ball $\{ u \in \TT^2 \, ; \, |u| \le \delta\}$ where $\delta>0$ is small.

In the sequel $C$ is a positive constant, which can depend on $\delta$ but not on $\lambda$, changing line to line. By Lemma \ref{lemma on hat a} and Lemma \ref{lem:thetad}, we have
\begin{equation*}
\int_{|u| \le \delta} \cfrac{1}{\lambda + \theta_2 (u) + G_{\lambda,\rho} (u)} \, du \ \le\  \int_{|u| \le \delta} \cfrac{1}{\lambda + C |u|^2 |\log|u| | + C H_{\lambda,\rho} (u)} \, du
\end{equation*}
where, recalling $m$ is the mean of $p$,
\begin{equation*}
H_{\lambda,\rho} (u) \ =\  \cfrac{|u \cdot m|^2}{\lambda + C |u|^2 | \log|u| | + C | u|^2 \int_{\TT^2} \frac{ds}{\lambda + \theta_2 (s) + \theta_2 (s-u)}}.
\end{equation*}

We now write
\begin{eqnarray*}
\int_{\TT^2} \frac{ds}{\lambda + \theta_2 (s) +\theta_2(s-u)} \ = \ \sum_{w \in {\mc C}_2 } \int_{|s-w| \le \delta/2} \frac{ds}{\lambda + \theta_2 (s) + \theta_2 (s-u)} + R_{\delta} (\lambda),
\end{eqnarray*}
where $\sup_{\lambda>0} R_\delta (\lambda) \le C$
since $\theta_2$ is positive and vanishes only on ${\mc C}_2$.  Similarly, as $|u|\leq \delta$, all integrals in the sum over $w \in {\mc C}_2$ are equivalent in order to the integral on the domain $\{|s| \le \delta/2\}$. By Lemma \ref{lem:thetad} and the fact, for $|x|$ small, that $|x^2| |\log|x|| \ge |x|^2$, it follows
\begin{eqnarray*}
\int_{\TT^2} \frac{ds}{\lambda + \theta_2 (s) + \theta_2 (s-u)} &\le& C\int_{|s| \le \delta/2} \frac{ds}{\lambda + |s|^2 + |s-u|^2}  +C\\
&\le& C\int_{|s| \le \delta/2} \frac{ds}{\lambda + |s|^2 + |u|^2}  +C\\
&\le& C \left| \log( \lambda + |u|^2) \right| + C,
\end{eqnarray*}
where the second  inequality is obtained from $|x|^2/4 \leq (|y|^2 + |x-y|^2)/2$ and the third from direct computations.

Substituting into $H_{\lambda,\rho}$ and noting again $|x|^2 | \log |x| | \ge |x|^2$ for small $|x|$, we get
\begin{eqnarray*}
H_{\lambda,\rho} (u) & \ge &\cfrac{|u \cdot m|^2}{\lambda + C |u|^2 | \log |u| | + C |u|^2 \left| \log( \lambda + |u|^2) \right|  } \\
&\ge& \cfrac{|u \cdot m|^2}{\lambda + C |u|^2 \left| \log( \lambda + |u|^2) \right|  }.
\end{eqnarray*}

Fix $\varepsilon \in (0,1)$ and observe, for $\delta$ sufficiently small, that $\sup_{|t| \le \delta }\{ |t|^{\ve}  |\log |t|| \} \le 1$. Then,
\begin{equation*}
H_{\lambda,\rho} (u) \ \ge\  \cfrac{|u \cdot m|^2}{\lambda +  C \left( \lambda + |u|^2 \right)^{1-\ve}  },
\end{equation*}
and we arrive at an upper bound for the integral in (\ref{eq:nice}) given by
\begin{equation}
\label{ub_nice}
C \int_{|u| \le \delta} \left[ \lambda +|u|^{2} + \cfrac{|u \cdot m|^2}{\lambda + C(\lambda +|u|^2)^{1-\ve} }\right]^{-1} du.
\end{equation}

We can assume $m=(m_1,m_2) \in \RR^2$ is such that $m_1 \ne 0, m_2 \ne 0$, and so $|u \cdot m|^2 \ge C |u|^2$. Otherwise,
 we choose a rotation $R_{-\theta}$ of angle $-\theta \in (0,2\pi)$, such that $R_{-\theta} m$ satisfies the previous condition, and change variables $v=R_{\theta} u$ in the above integral \eqref{ub_nice}.  Thus, an upper bound of \eqref{ub_nice} is
 \begin{eqnarray*}
&&C \int_{|u| \le \delta} \left[ \lambda +|u|^{2} + \cfrac{|u|^2}{\lambda + C (\lambda +|u|^2)^{1-\ve} }\right]^{-1} du \\
&&\ \ \ \ \ \ \ \ \ \ \le \ C \int_{|u| \le \delta} \left[ \lambda +|u|^{2} + \cfrac{|u|^2}{C (\lambda +|u|^2)^{1-\ve} }\right]^{-1} du,
\end{eqnarray*}
where we note $\lambda \le \lambda^{1- \ve} \le (\lambda +|u|^2)^{1-\ve}$ for all small $\lambda$.

Through polar coordinates, we are left to show
\begin{equation*}
\limsup_{\lambda \to 0} \int_{0}^{\delta} \cfrac{r}{\lambda + r^2 + C \frac{r^2}{(\lambda +r^2)^{1-\ve}} } \, dr \  < \  \infty.
\end{equation*}
Changing variables $v= \lambda^{1/2} r$, the integral
\begin{eqnarray*}
&&\int_0^{\delta \lambda^{-1/2}} \frac{v}{1+v^2 + C\lambda^{\ve -1} \frac{v^2}{(1+ v^2)^{1-\ve}}} dv\\
&&\ \ \ \ \leq \  \int_0^1 v dv +  C \lambda^{1- \ve} \int_1^{\delta \lambda^{-1/2}}  \frac{ (1+v^2)^{1-\ve}}{v} dv \ = \ O(1).
\end{eqnarray*}
This finishes the proof of (B). \qed

\subsection{Proof of Theorem \ref{th:sd-asymmetric-1}}

Only the results for $\alpha\leq 2$ need proof.  The upper bounds are obtained using Corollary \ref{cor seth bound} and Theorem \ref{th:symd1}.  Indeed, for completeness, we discuss the case $1<\alpha<2$, the rest being similar.
From Theorem \ref{th:symd1} we have that $\sigma^2_t(f)\sim t^{2-1/\alpha}$. Then, by the change of variables $\lambda t=s$, we obtain
\begin{equation*}
{\mc L}_f(\lambda) \ = \ \int_0^{\infty}e^{-\lambda t}\sigma_t^2(f)dt\ \leq\ \lambda^{1/\alpha-3}\int_0^{\infty}e^{s}s^{2-1/\alpha}ds \ = \ O(\lambda^{1/\alpha - 3}).
\end{equation*}

To address the lower bounds, we first note a bound for degree $2$ functions $g$ in $d=1$.  When $\alpha<2$, by the admissibility Theorem \ref{th:adm-asymmetric}, such a $g$ is admissible.  When $\alpha=2$, by Proposition \ref{sunder equivalence} and Theorem \ref{th:symd1bis}, the Tauberian variance $L_g(\lambda)\leq L^{(S)}_g(\lambda )\leq C \lambda^{-2}|\log \lambda|$, which is of smaller order than the desired lower bound for degree $1$ functions in this situation; in fact, we believe $g$ is admissible in this case (cf. Remark \ref{asym_rmk}), although this is not needed here.

Hence, decompose a local degree $1$ function $f$ as $f=\Psi_{\{0\}} + g$.   By the inequality $L_{\Psi_{\{0\}}}(\lambda)  \leq 2L_f(\lambda)+ 2L_g(\lambda)$ in \eqref{L_schwarz},  we need only to prove the lower bound for the specific one-point function $f(\eta) = \Psi_{\{0\}}$.   Recall the notation in Subsection \ref{one_pt_section} which is used throughout this subsection.
\medskip

Noting \eqref{eq:lapltrnasfotvar}, we apply Proposition \ref{I_prop} and estimate the integral $I_{1} (\lambda, 1/2)$ there which serves as a lower bound for $\langle \Psi_{\{0\}}, (\lambda - {\mc L})^{-1}\Psi_{\{0\}}\rangle_\rho$.
For this purpose,
we restrict the integration domain of the integral $I_1(\lambda, 1/2)$ in \eqref{eq:Ramses2}, around a small neighborhood of $0$, say $(0,\delta)$, for $\delta>0$ small.  Note, since $u$ is very small, the domains $D_V$ for $V\in {\mathcal C}_1$ (cf. \eqref{D_Veq}) take form
\begin{equation*}
D_0 (u) =[0,u], \quad D_1 (u)= [u,1].
\end{equation*}
Since $\rho=1/2$, $d=1$, it follows, from Lemma \ref{lemma on hat a}, that the sums of the two integrals, over domains $D_0$ and $D_1$, appearing in the definition of $F^1_{\lambda,1/2}$ in \eqref{eq:Ramses2} with respect to the integral $I_1(\lambda, 1/2)$ are of order
\begin{equation}
\label{integral_asym}
b_\alpha (u)  \int_0^1 \frac{ds}{\lambda+ \theta_1 (s) + \theta_1 (s-u)}.
\end{equation}
where
\begin{equation*}
b_{\alpha} (u)\  =\
\begin{cases}
\sin^2 (\pi u) \log^2 (u), \quad &{\text{ if }} \alpha=1,\\
 \sin^{2} (\pi u), \quad &{\text{ if }} \alpha>1.
 \end{cases}
 \end{equation*}

We rewrite the integral in \eqref{integral_asym} as the sum of the integrals over $[0,\delta]$, $[\delta, 1-\delta]$ and $[1-\delta, 1]$. By periodicity of $\theta_1$, the integral on $[1-\delta,1]$ is the same as that over $[0,\delta]$. Also, the integral on $[\delta,1-\delta]$ is $O(1)$ independent of $\lambda$ as $\theta_1$ vanishes only at $0$ and $1$.
However, in Lemma \ref{lemma-a2}, in Appendix \ref{Appendix_A}, $\alpha$-dependent bounds are given for the integral
$\int_0^{\delta}  (\lambda+ \theta_1 (s) + \theta_1 (s-u))^{-1}ds$.

We now substitute these estimates for the integral into the formula for $I_1(\lambda, 1/2)$.
\begin{enumerate}[i)]
\item For $\alpha=1$, since $b_1 (u) = \sin^2 (\pi u) \log^2 (u) \sim \pi^2 u^2 \log^2 (u)$ for $u \sim 0$,  for some positive constants $C_0,C_1$,
\begin{equation*}
I_1 (\lambda, 1/2) \ \succcurlyeq\  \int_0^{\delta} \frac{du}{\lambda + u + u^{2} \log^2 (u) \left[1+ C_0 \log \left (1+ \frac{C_1}{\lambda+ u/C_1}\right) \right]  }.
\end{equation*}
To show the last integral is equivalent in order to $\int_{0}^{\delta} (\lambda +u)^{-1} du =\log(1+ \delta/\lambda)$, it is sufficient to verify that the difference
\begin{equation*}
R_{\lambda}\ :=\ \int_{0}^{\delta} \cfrac{u^{2} \log^2 (u) \left[1+ C_0 \log \left (1+ \frac{C_1}{\lambda+ u/C_1}\right) \right] }{(\lambda +u)\left\{\lambda + u + u^{2} \log^2 (u) \left[1+ C_0 \log \left (1+ \frac{C_1}{\lambda+ u/C_1}\right) \right] \right\} }du \ = \ o(|\log \lambda|).
\end{equation*}
To this end, note that the denominator of the integrand is bounded below by $(\lambda +u)^2$. For small $\ve \in (0,1)$, as $u^{2} \log^2 (u) =O(u^{2-\ve})$ for $u$ small, the numerator is bounded by above by a constant times $u^{2-\ve} |\log (\lambda)|$. Then,  by the change of variables $u = \lambda v$, we have
\begin{equation*}
R_{\lambda} \ \le\  C |\log (\lambda)|  \int_{0}^{\delta} \cfrac{u^{2-\ve}}{(\lambda +u)^2} du \ =\ O(\lambda^{1-\ve}|\log (\lambda)|).
\end{equation*}

\item For $\alpha \in (1,2)$, since $b_\alpha (u)=\sin^2 (\pi u) \sim \pi^2 u^2$ for $u \sim 0$, it follows, for positive constants $C_0,C_1$, that
\begin{equation*}
I_1 (\lambda, 1/2) \ \succcurlyeq \ \int_0^{\delta} \frac{du}{\lambda + |u|^{\alpha} + C_0 u^{2}(1+ (\lambda + u^{\alpha} /C_1)^{1/\alpha -1}) }.
\end{equation*}

\begin{itemize}
\item Assume that $1<\alpha \le 3/2$. Changing variables $z=\lambda^{-1/\alpha} u$, and noting when $\alpha\leq 3/2$ and $\lambda\leq 1$ that $\lambda^{3/\alpha -2} \leq 1$, we have
\begin{eqnarray*}
I_1 (\lambda, 1/2)  &\succcurlyeq& \lambda^{1/\alpha-1} \int_0^{\delta \lambda^{-1/\alpha}} \frac{dz}{(1+z^{\alpha}) + \lambda^{3/\alpha -2}\,  z^2 (1 +\kappa_1 z^{\alpha} )^{1/\alpha -1} }\\
 &\succcurlyeq & \lambda^{1/\alpha -1} \int_0^{\delta \lambda^{-1/\alpha} } \frac{dz}{(1+z^{\alpha}) +  z^2 (1 +\kappa_1 z^{\alpha} )^{1/\alpha -1} }\  \succcurlyeq \ \lambda^{1/\alpha -1}.
\end{eqnarray*}

\item Assume that $3/2 \le \alpha< 2$. Changing variables $u=\lambda^{1-1/(2\alpha)} z$, similarly,
\begin{eqnarray*}
I_1 (\lambda, 1/2) &\succcurlyeq &\lambda^{-1/(2\alpha)}  \int_0^{\delta \lambda^{1/(2\alpha) -1}  } \frac{dz}{1 + \lambda^{\alpha -3/2} z^{\alpha} + z^{2}(1 +\kappa_1 \lambda^{\alpha -3/2} z^{\alpha})^{1/\alpha -1} }\\
& \succcurlyeq& \lambda^{-1/(2\alpha)}  \int_0^{\delta \lambda^{1/(2\alpha) -1}  } \frac{dz}{1 + z^{\alpha}
+  z^{2}} \ \succcurlyeq \ \lambda^{-1/(2\alpha)}.
\end{eqnarray*}
\end{itemize}

\item For $\alpha=2$, since $b_2 (u) =\sin^2 (\pi u) \sim \pi^2 u^2$ for $u\sim 0$, changing variables $u=\lambda^{3/4}z$, similarly we have $I_1(\lambda, 1/2)$ is greater in order than
\begin{eqnarray*}
&&\lambda^{-1/4}  \int_0^{\delta \lambda^{-3/4}  } \frac{dz}{1 + \lambda^{1/2} z^2\log(\lambda^{3/4}z)+ z^{2}R(\lambda,z)}\\
&&\ \ \succcurlyeq \lambda^{-1/4}\int_0^{M}
\frac{dz}{1 + \lambda^{1/2} z^2\log(\lambda^{3/4}z)+ z^{2}R(\lambda,z)},
\end{eqnarray*}
where $M>0$ and
$$R(\lambda, z) = \Big\{(1+\lambda^{1/2}z^2|\log(\lambda^{3/4}z)|) \cdot\big|\log \lambda + \log(1 + \lambda^{1/2}z^2|\log(\lambda^{3/4}z)|)\big|\Big\}^{-1/2}.$$
Since $R(\lambda, z)$ is of order $|\log \lambda |^{-1/2}$, we have further
\begin{eqnarray*}
I_1(\lambda, 1/2)
& \succcurlyeq &
\lambda^{-1/4}  \int_0^{M } \frac{dz}{1 + \kappa z^{2}|\log\lambda|^{-1/2}},
\end{eqnarray*}
which yields the desired lower bound. \qed

\end{enumerate}

\subsection{Proof of Theorem \ref{th:sd-asymmetric-2}}
The only statement to prove is the first one.  The
desired upper bound is a consequence of Corollary \ref{cor seth bound} and Theorem \ref{th:symd1}.
On the other hand, for the lower bound, again by Remark \ref{reduction_rmk} and admissibility of degree $2$ or more functions in $d=2$ given in Theorem \ref{th:adm-asymmetric}, we need only to focus on $f = \Psi_{\{0\}}$.

We begin as in the proof of Theorem \ref{th:sd-asymmetric-1}:
With $\alpha=2$ and $\rho=1/2$, to find a lower bound of $L_f(\lambda)$, using \eqref{eq:lapltrnasfotvar}, we estimate the integral $I_2(\lambda, 1/2)$ in \eqref{eq:Ramses2} which yields a lower bound for $\langle \Psi_{\{0\}}, (\lambda - {\mc L})^{-1}\Psi_{\{0\}}\rangle_\rho$.  We restrict the domain of integration in $I_2(\lambda, 1/2)$ over
 a small box $[0,\delta]^2$ with $\delta>0$ small.

For $u \in [0,\delta]^2$,
  by the periodicity
 in each direction of $\theta_2$
and $\hat a$, and Lemma \ref{lemma on hat a}, we bound the term in $F^2_{\lambda, 1/2}$ \eqref{eq:Ramses2} by
\begin{eqnarray*}
\sum_{V \in {\mc C}_2} \int_{s \in D_V (u)} \frac{| {\hat a} (s) + {\hat a} (u-s)|^2}{\lambda + \theta_2 (s) + \theta_2 (u-s)} ds  &\leq &  4\int_{\TT^2} \frac{| {\hat a} (s) + {\hat a} (u-s)|^2}{\lambda + \theta_2 (s) + \theta_2 (u-s)} ds\\
&\preccurlyeq&
|u|^2 \int_{\TT^2} \frac{1}{\lambda + \theta_2 (s) + \theta_2 (u-s)} ds.
 \end{eqnarray*}

We split the region of integration in five parts: The union of four sets $\{s \in \TT^2 \, ; \, |s-w| \le \delta/2\}$ for $w\in {\mc C}_2$ and its complement.
The integral on the complement is bounded $O(1)$  uniformly in $\lambda$ since $\theta_2$ vanishes exactly on ${\mc C}_2$. But, by periodicity
of $\theta_2$ in each direction, the remaining integrals over the first four regions are all equal. Thus, by Lemma \ref{lem:thetad}, $|x|^2 |\log |x|| \geq |x|^2$ for small $|x|$, and $|x|^2/4\leq (|y|^2 + |x-y|^2)/2$, we have
 \begin{eqnarray*}
  \int_{\TT^2} \frac{1}{\lambda + \theta_2 (s) + \theta_2 (u-s)} ds
&\preccurlyeq& 1 +  \int_{|s| \le \delta/2} \frac{1}{\lambda + |s|^2|\log|s|| + |u-s|^2|\log|u-s||} ds \\
&\preccurlyeq& 1+\int_{|s| \le \delta/2}  \frac{1}{\lambda + |s|^2 + |u-s|^2} ds\\
&\preccurlyeq& \left| \log (\lambda +|u|^2) \right|.
\end{eqnarray*}

Finally, by Lemma \ref{lem:thetad} again, and inequalities $|u|^2 \le |u|^2 |\log|u||$ and $ |u|^2 |\log (\lambda +|u|^2)| \le |u|^2 |\log|u|^2 |$ for small $|u|$, we obtain the lower bound,
\begin{eqnarray*}
I_2(\lambda, 1/2) & \succcurlyeq&  \int_{[0,\delta]^2} \cfrac{du}{\lambda +  |u|^2 | \log |u| | +  |u|^2 [1+ | \log(\lambda +|u|^2) |] }\\
&\succcurlyeq&  \int_{[0,\delta]^2} \cfrac{du}{\lambda + |u|^2 | \log |u| | }\\
&\succcurlyeq &  \int_0^{\delta} \frac{r dr}{\lambda +  r^2 | \log r|} \  \succcurlyeq  \   \int_{\lambda}^{\delta} \frac{dr}{r |\log (r)|} \ = \ O(|\log|\log \lambda||).
\end{eqnarray*}
\qed

\section{Monotonicity:  Proof of Theorem \ref{th:monotonicity}}
\label{Appendix_B}

It will be helpful to write the long-range short asymmetic generator, with jump rate $\bar p^{(SA)}$, in the following way, which stresses dependence on $\alpha$:  For $\alpha>1$, let ${\mc L}^{\alpha}= \bar {\mc A} + {\mc S}^\alpha$ where ${\mc S}^\alpha$ generates the symmetric long-range exclusion process with sub-probability jump rate $s_\alpha =  \cfrac{ {\bf 1}_{y \ne 0} }{|y|^{1+\alpha}}$, and
$ \bar {\mc A}$ is a finite-range anti-symmetric operator on local functions given by
\begin{equation*}
({\bar{\mc A}} f)(\eta) \ = \ \sum_{x,y\in\ZZ^d} \bar a(y) \eta(x)(1-\eta(x+y)\left[ f(\eta^{x,x+y}) -f(\eta) \right].
\end{equation*}

We remark, in the following, the arguments do not depend on the fixed value of $\rho$.
\medskip

\noindent {\it Proof of Theorem \ref{th:monotonicity}.}
We first calculate the derivative of the map $\alpha:(1,\infty)\mapsto \bar L^{(SA)}_f(\lambda)$.

Let $u_{\lambda}^{\alpha}$ be the solution in ${\bb L}^2 (\nu_{\rho})$ of the resolvent equation
\begin{equation*}
\lambda u_{\lambda}^{\alpha} - {\mc L}^{\alpha} u_{\lambda}^{\alpha}\  =\ f.
\end{equation*}
Observe, formally, the derivative of ${\mc L}^{\alpha}$ is given by $-(\alpha +1) {\mc S}^{\alpha +1}$ and the derivative $\tfrac{d}{d\alpha} u_{\lambda}^{\alpha}=: v_{\lambda}^{\alpha}$ is the solution of the resolvent equation
\begin{equation*}
(\lambda - {\mc L}^{\alpha}) v_{\lambda}^{\alpha} = - (\alpha +1) {\mc S}^{\alpha+1} u_{\lambda}^{\alpha}.
\end{equation*}

Indeed, in Lemmas \ref{L^2_bound} and \ref{lem:der-res}, we show that ${\mc S}^{\alpha +1}u_\lambda^{\alpha}$ is well-defined and in $\LL^2(\nu_\rho)$, and $v_\lambda^{\alpha}$ is the weak limit of $h^{-1}[u_\lambda^{\alpha + h} - u_\lambda^{\alpha}]$ as $h\downarrow 0$.
In particular, the function $\alpha \in (1, \infty) \mapsto \bar L_f^{(SA)} (\lambda) = 2\lambda^{-2}\langle f, u_\lambda^{\alpha,\ve}\rangle_\rho$ (cf. \eqref{eq:lapltrnasfotvar}) is differentiable, and its derivative equals
\begin{equation}
\label{derivative_monot}
\frac{2}{\lambda^2}\langle f, v_{\lambda}^{\alpha,\ve} \rangle_{\rho}.
\end{equation}

Our task now is to show that \eqref{derivative_monot} is nonnegative when the short asymmetry is a mean-zero asymmetry $(MZA)$ or the process is symmetric $(S)$.
Equivalently, we have to show
\begin{equation*}
\langle f , (\lambda- {\mc L}^{\alpha})^{-1}(- {\mc S}^{\alpha +1}) ( \lambda - {\mc L}^{\alpha})^{-1} f \rangle_{\rho}\ \ge \ 0.
\end{equation*}

Let
$g^{\alpha} =  (\lambda- {\mc L}^{\alpha})^{-1}(- {\mc S}^{\alpha +1}) ( \lambda - {\mc L}^{\alpha})^{-1} f $
and, for $z>0$,
\begin{equation*}
g^{\alpha}_z \ = \ (\lambda- {\mc L}^{\alpha})^{-1}(z- {\mc S}^{\alpha +1}) ( \lambda - {\mc L}^{\alpha})^{-1} f .
\end{equation*}
By Lemma \ref{L^2_bound}, both $g^\alpha, g^\alpha_z\in \LL^2(\nu_\rho)$ and $g^{\alpha} = \lim_{z \to 0} g^{\alpha}_z$ in $\LL^2(\nu_\rho)$.  Also, $g^\alpha_z$ belongs to the domain of ${\mc L}^\alpha$.  Hence, in the following, $g^\alpha_z$ can be approximated by local functions $g$.

It will be sufficient to prove
\begin{equation*}
\liminf_{z \to 0}\;  \langle g_{z}^{\alpha} \, , \, (\lambda- {\mc L}^{\alpha})(z- {\mc S}^{\alpha +1})^{-1} ( \lambda - {\mc L}^{\alpha})  g_{z}^{\alpha}\rangle_{\rho} \ \ge\  0
\end{equation*}
or, for local functions $g$,
\begin{equation}
\label{monot_ineq}
\langle g \, , \, (\lambda- {\mc L}^{\alpha})(z- {\mc S}^{\alpha +1})^{-1} ( \lambda - {\mc L}^{\alpha})  g\rangle_{\rho} \ \ge\  0.\end{equation}

Observe now, by the polarization formula with respect to the $H_{-1,z}$-norm defined by $z-{\mc S}^{\alpha+1}$ and the antisymmetry of ${\bar{{\mc A}}}$, that
\begin{eqnarray}
\label{eq:sdfghj}
&& \langle g \, , \, (\lambda- {\mc L}^{\alpha})(z- {\mc S}^{\alpha +1})^{-1} ( \lambda - {\mc L}^{\alpha})  g\rangle_{\rho}\\
&& = \ \langle \big[(\lambda- {\mc L}^{\alpha})\big]^{*} g \, , \, (z- {\mc S}^{\alpha +1})^{-1} ( \lambda - {\mc L}^{\alpha})  g\rangle_{\rho}\nonumber \\
&& =\ \cfrac{1}{2} \left[ \langle (\lambda -{\mc S}^{\alpha} ) g \, , \, (z-{\mc S}^{\alpha +1})^{-1}  (\lambda -{\mc S}^{\alpha} ) g \rangle_{\rho} \right. \nonumber \\
&& \quad \quad \left. -  \langle {\bar{\mc A}} g \, , \, (z-{\mc S}^{\alpha +1})^{-1} {\bar{\mc A}}  g\rangle_{\rho} \right]. \nonumber
\end{eqnarray}

The right-side of the previous equality is the difference of two non-negative terms. If the process is symmetric and generated by ${\mc L}^\alpha= {\mc S}^\alpha$, we have $\bar {\mc A} = 0$ and the second term disappears.  Trivially, the right-side is non-negative.
\medskip

Consider now the case when the process is mean-zero asymmetric and $|\bar a(y)|\leq \ve$ for all $y\in \ZZ^d$.  In addition, fix $\beta_0>\alpha_0 >1$ and take $\alpha \in [ \alpha_0, \beta_0]$.  We have, as $s_{\alpha + 1}\leq s_{\alpha}$, the Dirichlet forms (cf. \eqref{dirichletform}),
\begin{equation}
\label{eq:salpha}
\langle \varphi, -{\mc S}^{\alpha +1}\varphi\rangle_\rho \  \le \ \langle \varphi, -{\mc S}^{\alpha}\varphi\rangle_\rho.
\end{equation}
 A lower bound for (\ref{eq:sdfghj}), using Lemma \ref{sunder equivalence}, is then given by
\begin{eqnarray}
\label{lower_bound_monot}
&&\cfrac{1}{2} \left[ \langle (\lambda -{\mc S}^{\alpha} ) g \, , \, (z-{\mc S}^{\alpha})^{-1}  (\lambda -{\mc S}^{\alpha} ) g\rangle_{\rho} \right.\nonumber\\
& &\quad \quad \left. - \langle {\bar{\mc A}} g\, , \, (z-{\mc S}^{\alpha +1})^{-1} {\bar{\mc A}}  g \rangle_{\rho} \right].
\end{eqnarray}
Writing $(\lambda -{\mc S}^{\alpha} ) g=(\lambda -z) g+ (z-{\mc S}^{\alpha} ) g$ we see that if $0<z \le \lambda$ we have
\begin{eqnarray*}
&&\langle (\lambda -{\mc S}^{\alpha} ) g\, , \, (z-{\mc S}^{\alpha})^{-1}  (\lambda -{\mc S}^{\alpha} ) g \rangle_{\rho}\\
&& = \ (\lambda-z)^2 \langle g\, , \, (z-{\mc S}^\alpha)^{-1} g\rangle_{\rho}+ 2(\lambda -z) \langle g\, , \,  g \rangle_{\rho}\\
&&\quad \quad+ \langle g, (z-{\mc S}^\alpha) g\rangle_{\rho}\\
& &\ge \ \langle g, (z-{\mc S}^\alpha) g \rangle_{\rho}.
\end{eqnarray*}

Writing $\bar {\mc A}/\ve =:  {\mc A}_0$, our aim is now to bound the term
\begin{equation}
\label{help_monot}
\langle {\mc A}_0 g \, , \, (z-{\mc S}^{\alpha +1})^{-1} {\mc A}_0  g\rangle_{\rho} \ \leq \ C \,  \langle g, (z-{\mc S}^\alpha) g\rangle_{\rho}
\end{equation}
 with respect to a constant $C>0$ independent of $z, \lambda, \ve$ and $g$. Then, by  choosing $\ve \le 1/{\sqrt C} $, inserting into \eqref{lower_bound_monot}, we will show \eqref{monot_ineq}.

\medskip

Let now ${\mc S}^{(FR-NN)}$ be the generator of the one-dimensional nearest-neighbor symmetric simple exclusion process. By Proposition \ref{norms_comparison}, as $\alpha_0+1 >2$ the Dirichlet forms of ${\mc S}^{\alpha+1}$ and of ${\mc S}^{(FR-NN)}$ are uniformly equivalent for $\alpha \in [\alpha_0,\beta_0]$:  There exists a constant $C:=C(\alpha_0, \beta_0)$ such that for $\alpha \in [\alpha_0, \beta_0]$ and local functions $\varphi$,
\begin{equation}
\label{eq:comp}
C^{-1}\langle \varphi, - {\mc S}^{(FR-NN)}\varphi\rangle_\rho \ \le\ \langle \varphi,  -{\mc S}^{\alpha +1}\varphi\rangle_\rho\  \le\ C\langle\varphi, -{\mc S}^{(FR-NN)}\varphi\rangle_\rho.
\end{equation}

Then, for local functions $g$, we write, by Lemma \ref{sunder equivalence},
\begin{eqnarray*}
\label{eq:vf345}
\langle {\mc A}_0 g \, , \, (z-{\mc S}^{\alpha +1})^{-1} {\mc A}_0  g \rangle_{\rho} &=& \sup_{\varphi} \left\{ 2 \langle {\mc A}_0 g , \varphi \rangle_{\rho} - \langle \varphi \,,  \, (z-{\mc S}^{\alpha+1}) \varphi \rangle_{\rho} \right\}\\
& \le &\sup_{\varphi} \left\{ 2 \langle {\mc A}_0 g , \varphi \rangle_{\rho} - C^{-1}\langle \varphi \,,  \, (z-{\mc S}^{(FR-NN)}) \varphi \rangle_{\rho} \right\} \\
&= &\ C\langle {\mc A}_0 g \, , \, (z-{\mc S}^{(FR-NN)})^{-1} {\mc A}_0  g \rangle_{\rho}
\end{eqnarray*}

By finite-range sector inequality \eqref{finiterangesector}, which applies to the finite-range `mean-zero' operator ${\mc A}_0$, approximating $h:=(z-{\mc S}^{(FR-NN)})^{-1} {\mc A}_0  g$ by local functions, for a constant $K$ independent of $g$ and $\ve>0$,
\begin{eqnarray}
\label{eq:scafr}
&&\langle {\mc A}_0g, (z-{\mc S}^{(FR-NN)})^{-1} {\mc A}_0  g\rangle_\rho\\
 &&\quad \quad \quad \leq \ K \langle g , -{\mc S}^{(FR-NN)} g\rangle_\rho^{1/2}  \cdot \langle h, -{\mc S}^{(FR-NN)} h \rangle_\rho^{1/2}.\nonumber
\end{eqnarray}
Since $-{\mc S}^{(FR-NN)} \leq (z-{\mc S}^{(FR-NN)})$ in the sense of Dirichlet forms by \eqref{eq:comp}, substituting into \eqref{eq:scafr}, we obtain
$$\langle {\mc A}_0 g, (z-{\mc S}^\alpha){\mc A}_0 g\rangle_\rho \ \leq \ K^2 \langle g, (z-{\mc S}^{(FR-NN)} g\rangle_\rho.$$
Finally, as $-{\mc S}^{(FR-NN)} \leq -M {\mc S}^\alpha$, with $M = s^{(FR-NN)}(1)/\bar s(1)$, in the sense of Dirichlet forms (cf. proof of the left inequality in Lemma \ref{norms_comparison}), we obtain \eqref{help_monot}. \qed

\medskip
To simplify the notation we drop dependencies on $\lambda$ and fix $\alpha>0$ for the remainder of the section.
We observe if $h>0$ is sufficiently small, then the operator ${\mc G}^h = -\tfrac{{\mc S}^{\alpha +h} -{\mc S}^\alpha}{h}$, well defined on the set of local functions, is the pregenerator of a Markov process. Indeed, we have that the action of ${\mc G}^h$ on a local function $g$ is given by
\begin{eqnarray*}
({\mc G}^h g)(\eta) &=& \frac{1}{h} \sum_{x,y \in \ZZ}  \cfrac{{\bf 1}_{ y \ne 0}}{|y|^{1+\alpha}} \big[ 1- \tfrac{1}{|y|^h} \big] \, \big[ g(\eta^{x,x+y}) -g(\eta) \big] \\
&:=& \sum_{x,y} p^h (y) \left[ g(\eta^{x,x+y}) -g(\eta) \right],
\end{eqnarray*}
with
\begin{equation*}
p^h (y)\ :=\ \frac{1}{h} \cfrac{{\bf 1}_{ y \ne 0}}{|y|^{1+\alpha}} \big[ 1-  \tfrac{1}{|y|^h} \big].
\end{equation*}
If $h>0$, then $\sum_{y} p^h (y) < \infty$ and $0\leq p^h(y) \leq 1$ for all $y\neq 0$.
 It follows that ${\mc G}^h$ is closable and its closure in $L^2(\nu_\rho)$ is the generator of a symmetric exclusion process with  jump rate $p^h$.
For $h_0$ small and $y\neq 0$, by the mean-value theorem, for all $y\neq 0$,
\begin{equation}
\label{eq:plog}
\sup_{0<h<h_0} p^h(y)\ \leq \  \frac{\log|y|}{|y|^{\alpha+1} }\ =:\ p_{{\rm log}}(y).
\end{equation}

 We now define some notions, following \cite{Andjel_zr}.  Define
$\kappa(x) :=  \sum_{n\geq 0} 2^{-n}s^{(n)}(0,x)\ \leq \ 2$ for $x\in\ZZ^d$, where $s^{(n)}(0,y)$ is the $n$-step rate
of reaching $y$ starting from $0$ according to sub-probability kernel $s_\beta$.  Then,
$$\sum_{y\in\ZZ^d} s_\beta(y-x)\kappa(y) \leq 2 \kappa(x) \ \ {\rm and  \ \ }\sum_{x\in\ZZ^d} \kappa(x) \leq  \sum_{n\geq 0}2^{-1}\sum_{x\in\ZZ^d} s^{(n)}(0,x) \leq 2.$$

Let $\|\eta\| = \|\eta\|_\beta = \sum_x \kappa(x) \eta(x)$.  Define also ${\rm Lip}$ as the class of functions $f$ such that $|f(\eta) - f(\zeta)| \leq C\|\eta - \zeta\|$ for all $\eta, \xi \in \Omega$, and let $c(f)=c_\beta(f)$ be the smallest such constant $C$.  Note that all local functions belong to ${\rm Lip}$; for example, the basis element $f(\eta) = \prod_{i=1}^k (\eta(x_i)-\rho)$ satisfies
\begin{eqnarray*}
|f(\eta) - f(\zeta)| &=&  \sum_{j=1}^k|(\eta(x_j) - \zeta(x_j)|\prod_{i\neq j}|\eta(x_i)-\rho)|\\
&\leq& |1-\rho|^{k-1}\sum_{i=1}^k |\eta(x_i) - \zeta(x_i)| \\
&\leq& |1-\rho|^{k-1}\max\{\kappa^{-1}(x_i): 1\leq i\leq k\}\|\eta - \zeta\|.
\end{eqnarray*}

Analogously, let $\kappa'(y)$ be the $n$-step rate of $y$ with respect to sub-probability $p_{{\rm log}}(\cdot)$ defined in (\ref{eq:plog}).  As before, $\sum_y p_{\rm log}(y-x)\kappa'(y) \leq 2 \kappa'(x)$ and $\sum_y \kappa'(y) \leq 2$.  Define the distance $\|\cdot\|'$ with respect $\kappa'$, and let ${\rm Lip'}$ be the space of  functions with `Lipschitz' norm $c'(\cdot)$.

\begin{lemma}
\label{L^2_bound}
For $0<\beta \leq \gamma,\chi$, ${\mc S}^{\chi} u^{\gamma}$ is well defined and belongs to $\LL^2(\nu_\rho)$.  Moreover, we have $\sup_{0<h<h_0}\|{\mc G}^h u^{\alpha +h}\|_{\LL^2(\nu_\rho)}<\infty$.
In addition, $g^\alpha, g^\alpha_z\in \LL^2(\nu_\rho)$ and $\lim_{z\downarrow 0} g^\alpha_z = g^\alpha$ in $\LL^2(\nu_\rho)$.
\end{lemma}

\begin{proof} The argument is to develop pointwise, $\LL^\infty$ bounds.
We observe that exclusion generators are well-defined on ${\rm Lip}$.  In fact, for $f\in {\rm Lip}$, we have the pointwise bound $|{\mc S}^\chi f(\eta)| \leq 6 c(f)$.  Indeed, as $|\eta(\cdot)|\leq 1$, note
\begin{eqnarray*}
|f(\eta^{x,y}) - f(\eta)| &\leq& c(f) \|\eta^{x,y} - \eta\| \\
&= &c(f) (\kappa(x) +\kappa(y)) |\eta(x) - \eta(y)| \ \leq \ c(f)(\kappa(x) +\kappa(y)),\end{eqnarray*}
and write, noting $s_\chi(\cdot) \leq s_\beta(\cdot)$,
\begin{eqnarray*}
|{\mc S}^\chi f(\eta)| & \leq & c(f)\sum_x \sum_y s_\beta(y-x) (\kappa(x) + \kappa(y)) \ \leq \ 6c(f).\end{eqnarray*}

Let now $T^\gamma_t$ be the semigroup corresponding to ${\mc L}^{\gamma}$.    In Lemma \ref{semigroup_bound}, we show, for $f\in {\rm Lip}$, that $T_t^\gamma f \in {\rm Lip}$ and $c(T_t^\gamma f) \leq Mc(f)t$ with respect to a universal constant $M$.

Write $u^{\gamma} = \int_0^\infty e^{-\lambda t}T^{\gamma}_t f dt$.
Then,
$$u^\gamma \in {\rm Lip} \ \ {\rm  and  \ \ }c(u^\gamma) \leq Mc(f)\int_0^\infty te^{-\lambda t}dt.$$
Therefore,
$$|{\mc S}^\chi u^\gamma(\eta)| \ \leq \ 6c(u^\gamma) \ \leq \ 6Mc(f)\int_0^\infty te^{-\lambda t}dt.$$
In particular, for local functions $f$, ${\mc S}^\chi u^\gamma \in \LL^\infty$ and so also belongs to $\LL^2(\nu_\rho)$.  This proves the first statement.
Note that the above bounds depend only on $\beta$.

To prove the second statement, for $f\in {\rm Lip'}$, we have
\begin{eqnarray*}
|{\mc G}^h f(\eta)| & \leq & c'(f)\sum_{x,y\in\ZZ^d} p_{{\rm log}}(y-x)\big[\kappa'(y) + \kappa'(x)\big]\ \leq \  6 c'(f).
\end{eqnarray*}
Again, by Lemma \ref{semigroup_bound},
for all $h>0$, $T^{\alpha+h}_t f\in {\rm Lip'}$
and $c'(T^{\alpha+h}_tf) \leq Mc'(f)t$ where $M$ is a universal constant.  Hence, as before, we obtain $\sup_{0<h<h_0}\|{\mc G}^hu^{\alpha +h}(\eta)\|_{L^\infty}<\infty$ and the desired $\LL^2(\nu_\rho)$ bound.

To prove the third statement, as now $\|{\mc S}^{\alpha+1}u^\alpha\|^2_{\LL^2(\nu_\rho)}<\infty$, we have $\|g^\alpha\|^2_{\LL^2(\nu_\rho)} \leq \lambda^{-1}\|{\mc S}^{\alpha + 1}u^\alpha\|^2_{\LL^2(\nu_\rho)} <\infty$ by the resolvent bound.  Similarly,
$$\|g^\alpha_z\|^2_{\LL^2(\nu_\rho)} \ \leq\  2z\lambda^{-1}\|u^\alpha\|^2_{\LL^2(\nu_\rho)} + 2\|{\mc S}^{\alpha +1}u^\alpha\|^2_{\LL^2(\nu_\rho)}\ <\ \infty,$$
 and as $z\downarrow 0$, the limit $g^\alpha_z \rightarrow g^\alpha$ follows.
\end{proof}

\begin{lemma}
\label{semigroup_bound}
For $0<\beta\leq \gamma$, the semigroup $T^\gamma_t$ is invariant on ${\rm Lip}$ and $c(T_t^\gamma f) \leq 48c(f)t$.  Similarly, for $h>0$, $T^{\alpha+h}_t$ is invariant on ${\rm Lip}'$ and $c'(T^{\alpha+h}_tf) \leq 48 c'(f)t$.
\end{lemma}

\begin{proof}
 The argument is a modification of the  proof of Lemmas 2.2 and 2.6 in \cite{Andjel_zr} to the exclusion context.  One can couple two processes starting from $\eta$ and $\zeta$ with coupled generator
\begin{eqnarray*}
\bar Lf(\eta,\zeta) &=& \sum_{x,y\in ZZ^d}p_\gamma(y-x)\eta(x)(1-\zeta(x))\big[f(\eta^{x,y}, \zeta) - f(\eta,\zeta)\big]\\
&& \ + \sum_{x,y\in\ZZ^d}p_\gamma(y-x)\zeta(x)(1-\eta(x))\big[f(\eta,\zeta^{x,y}) - f(\eta,\zeta)\big]\\
&&\ + \sum_{x,y\in\ZZ^d}p_\gamma (y-x)\eta(x)\zeta(x)\big[f(\eta^{x,y},\zeta^{x,y}) - f(\eta,\zeta)\big].
\end{eqnarray*}
Let $\bar T_t$ be the corresponding semigroup.  Marginally, both processes are generated by ${\mc L}^\gamma$, noting $\eta(x)[\varphi(\eta^{x,y}) - \varphi(\eta)] = \eta(x)(1-\eta(y)[\varphi(\eta^{x,y}) - \varphi(\eta)]$.

Write
$$|T^\gamma_tf(\eta) - T^\gamma_tf(\zeta)| \ = \ |\bar T_t g(\eta,\zeta) | \ \leq \ c(f)\bar T_t h(\eta,\zeta)$$
where $g(\eta,\zeta) = f(\eta) - f(\zeta)$ and $h(\eta,\zeta) = \|\eta - \zeta\|$.
Note $p_\gamma (y-x) \leq 2s_\gamma (y-x)\leq 2s_\beta (y-x)$, and $|\eta(\cdot) - \zeta(\cdot)|\leq 1$.
Then, by the triangle inequality, we calculate
\begin{eqnarray*}
|\bar Lh(\eta,\zeta)| & \leq &  4\sum_{x,y} p_\gamma(y-x)\big[\kappa(y) + \kappa(x)\big] \\
& \leq & 8\sum_{x,y}s_\beta(y-x)\big[\kappa(y) + \kappa(x)\big]\ \leq \ 48.
\end{eqnarray*}
Hence, $\bar T_tf(\eta,\zeta) \leq 48c(f)t$.

The second statement follows the same argument noting that $p_{\alpha +h}(\cdot)\leq 2s_{\alpha +h}(\cdot)\leq 2p_{\rm log}(\cdot)$ when $h>0$.
\end{proof}

\begin{lemma}
\label{lem:der-res}
Let $f$ be a local function and $u_{\lambda}^{\alpha}$ be the $\LL^2(\nu_\rho)$ solution of the resolvent equation
$(\lambda -{\mc L}^{\alpha}) u_{\lambda}^{\alpha} =f$.
Then, for $\alpha \in (1,\infty)$, we have the $\LL^2(\nu_\rho)$-weak convergence,
$$h^{-1} [ u_{\lambda}^{\alpha+h} - u_{\lambda}^{\alpha}] \  \xrightarrow[h\to 0]\ v_{\lambda}^{\alpha},$$
where $v_\lambda^{\alpha}$ is the solution in ${\bb L}^2 (\nu_{\rho})$ of the resolvent equation
$$(\lambda - {\mc L}^{\alpha}) v_{\lambda}^{\alpha} \ =\  - (\alpha +1) {\mc S}^{\alpha+1} u_{\lambda}^{\alpha}.$$
\end{lemma}

\begin{proof}
 Let $v^h=h^{-1} [ u^{\alpha+h} - u^{\alpha}]$ and observe that
\begin{equation}
\label{eq:vah}
(\lambda- {\mc L}^\alpha) v^h + ({\mc S}^{\alpha} -{\mc S}^{\alpha +h}) v^h \ = \ {\mc G}^h u^\alpha.
\end{equation}

We now claim that
\begin{equation}\label{help_bound_monot}
(v^h)_{h>0} \ \  {\rm  is \ uniformly \ bounded \ in \ }{\LL}^2 (\nu_{\rho}).
\end{equation}
Assuming this bound, let $v_0$ be a limiting point for $(v^h)_{h>0}$ and denote by $v^{h'}$ a subsequence converging weakly to $v_0$.  Taking the scalar product of the two sides of (\ref{eq:vah}) with an arbitrary local function $\varphi$, we see that
\begin{equation}
\label{eq:vah2}
\langle (\lambda- {\mc L}^\alpha)^* \varphi \, , \,  v^{h'} \rangle_{\rho} + \langle ({\mc S}^{\alpha} -{\mc S}^{\alpha +h'}) \varphi \, , \,  v^{h'} \rangle_{\rho} \ =\  \langle {\mc G}^h
 \varphi \, ,\,  u^{\alpha}\rangle_{\rho}.
\end{equation}
Since $\varphi$ is local, we have that ${\mc G}^h$
converges strongly to $-(\alpha +1) {\mc S}^{\alpha +1} \varphi$ and  $({\mc S}^{\alpha} -{\mc S}^{\alpha +h'}) \varphi$ converges strongly to $0$. It follows that $v_0$ satisfies
\begin{equation*}
\langle (\lambda -{\mc L}^\alpha)^* \varphi \, ,\, v_0\rangle_{\rho} \ = \ \langle \varphi \, ,\,  (\lambda -{\mc L}^\alpha) v_0 \rangle_{\rho} \ =\  -(\alpha+1) \langle {\mc S}^{\alpha +1} \varphi \, , \, u^\alpha \rangle_{\rho}
\end{equation*}
for all local functions $\varphi$. Since the set of local functions is dense in ${\bb L}^2 (\nu_{\rho})$, we conclude that $v_0$ is a solution of the resolvent equation
$$(\lambda - {\mc L}^{\alpha}) v_{0} \ =\  - (\alpha +1) {\mc S}^{\alpha+1} u^{\alpha}.$$
By the uniqueness in $\LL^2(\nu_\rho)$ of the solution of this resolvent equation, we get the uniqueness of the limiting point $v_0$ and the desired statement in the lemma.
\medskip

Therefore, it remains only to show \eqref{help_bound_monot}.  By definition of $v^{h}$ and $u^{\alpha +h}$, we have
\begin{equation*}
(\lambda -{\mc L}^{\alpha} ) v^{h}\  = \ - {\mc G}^h u^{\alpha +h}.
\end{equation*}
Take the scalar product of each side of this equation with $v^h$. Then,
\begin{equation*}
\lambda \langle v^h ,v^h \rangle_{\rho} + \langle v^h , -{\mc S}^{\alpha}  v^h \rangle_{\rho} \ = \ \langle  v^h, -{\mc G}^h u^{\alpha+h} \rangle_{\rho} \ \leq \ K\|v^h\|_{\LL^2(\nu_\rho)},
\end{equation*}
where, by Lemma \ref{L^2_bound},
$K := \sup_{0<h<h_0} \|{\mc G^h} u^{\alpha + h}\|_{\LL^2(\nu_\rho)}  < \infty$. Since the Dirichlet form
$\langle v^h , -{\mc S}^{\alpha}  v^h \rangle_{\rho}\geq 0$, we have $\sup_{0<h<h_0}\|v^h\|_{\LL^2(\nu_\rho)}\leq K\lambda^{-1}$,
finishing the proof.
\end{proof}

\appendix

\section{Useful computations}
\label{Appendix_A}
In this section, $\theta_d = \theta_d(\cdot;s_0(\cdot))$ (cf. Subsection \ref{theta_section}).

\begin{lemma} \label{lemma-a1}
Let
$I_{d,\alpha} (t)\ =\ \int_{\TT^d} \cfrac{ \theta_d (u) \, t -1 + e^{-t \theta_d (u)} }{\theta^2_d (u)}du$ be the integral in \eqref{variance_theta}.
\begin{itemize}
\item If $d=1$,
$$
I_{d,\alpha} (t) \ \sim\
\left\{\begin{array}{rl}
t & \ {\rm when \ }\alpha<1,\\
t \log(t) & \ {\rm when  \ }\alpha=1,\\
t^{2-1/\alpha} & \ {\rm when \ } 1< \alpha <2, \\
t^{3/2} (\log(t))^{-1/2} & \ {\rm when \ }  \alpha=2,\\
t^{3/2} & \ {\rm when \ } \alpha>2.
\end{array}\right.
$$
\item If $d=2$,
\begin{equation*}
I_{d,\alpha} (t) \ \sim\
\left\{\begin{array}{rl}
t& \ {\rm when \ }  \alpha <2, \\
t\log (\log(t)) & \ {\rm when \ }   \alpha=2,\\
t\log(t) & \ {\rm when \ } \alpha>2.
\end{array}\right.
\end{equation*}
\item If $d\geq3$, for all $\alpha>0$, $I_{d,\alpha}(t) \sim t$.
\end{itemize}
\end{lemma}

\begin{proof}
We argue only in the one dimensional case, as the other statements are similar.

If $\alpha<1$ then the integrand, divided by $t$, converges pointwise as $t\uparrow\infty$,
$$\cfrac{ \theta_d (u) \, t -1 + e^{-t \theta_d (u)}}{t\theta^2_d (u)}\ \to \ \frac{1}{\theta_1 (u)},$$
and is dominated by $1/ \theta_1 (u)$. By Lemma \ref{lem:thetad}, the function $1/ \theta_1$ is integrable on $\TT^1$ and so the result follows by dominated convergence.

Let now $\alpha \ge 1$. Fix $\delta>0$ small and write $I_{1,\alpha}$ as the sum of the three integrals over $[0, \delta]$, $[\delta, 1-\delta]$ and $[1-\delta,1]$. The integral over $[\delta,1-\delta]$ is $O(t)$ as $\theta_1$ does not vanish on the domain.
By changing variables $v=1-u$ and periodicity of $\theta_1$, the integral over $[1-\delta, 1]$ is equal to the integral over $[0,\delta]$.

When $\alpha>2$, by changing variables $v={\sqrt t} u$,  we need to estimate
$$t^{3/2} \int_0^{\infty}  {\bf 1}_{0 \le v \le \delta {\sqrt t}} \frac{t \theta_1 (v t^{-1/2}) - 1 + e^{- t \theta_1 (v t^{-1/2})}}{[ t \theta_1 (v t^{-1/2}) ]^2}  dv.$$
By Lemma \ref{lem:thetad}, $\theta_1 (w) = J(1,\alpha) |w|^2 + o(|w|^2)$, as $w \to 0$, and therefore as $t\uparrow\infty$ the integrand converges pointwise to
$$h(v)\ = \ \frac{J(1,\alpha) v^2 - 1 + e^{- J(1,\alpha) v^2}}{[J(1,\alpha) v^2 ]^2}.$$
Since the function
\begin{equation*}
g(x) \ = \ \left\{\begin{array}{rl}
\frac{x- 1+e^{-x}}{x^2} & \ {\rm if \ } x>0 \\
 1/2 & \ {\rm if \ }x=0.
\end{array}\right.
\end{equation*}
is bounded near $0$ and is of order $O(x^{-1})$ for large $x$, noting again the asymptotics of $\theta_1(w)$, we have
 $\int_0^{\delta \sqrt{t}} \frac{t \theta_1 (v t^{-1/2}) - 1 + e^{- t \theta_1 (v t^{-1/2})}}{[ t \theta_1 (v t^{-1/2}) ]^2}  dv$ converges to $\int_0^{\infty} h (v) dv<\infty$ by bounded convergence, and the statement holds for $\alpha>2$.

When $1<\alpha<2$, by changing variables $v=t^{1/\alpha} u$, the result follows by similar calculations.

When $\alpha=1$, the calculation is more involved.  Consider the change of variables $v=tu$ in the integral over $u \in [0,\delta]$. We are the reduced to study
$
t \int_0^{\delta t} g(t \theta_1 (v/t)) dv$.
Observe
\begin{equation*}
\int_0^{\delta t} g(J(1,1) v ) dv \ =\  \int_0^\delta g(J(1,1)v) dv \,+ \, \int_{\delta}^{\delta t} \frac{e^{-J(1,1) v} -1}{[J (1,1) v ]^2}dv \, + \, \int_{\delta}^{\delta t} \frac{1}{J(1,1) v} dv.
\end{equation*}
As $t\uparrow \infty$, for fixed $\delta$, the second integral converges to $\int_\delta^{\infty} \frac{e^{-J(1,1) v} -1}{[J (1,1) v ]^2} dv$ and the third one equals $\log (t) / J(1,1)$.  Hence,
\begin{equation*}
\int_0^{\delta t} g(J(1,1) v ) dv \ =\  \frac{\log t}{J(1,1)} + o (\log t).
\end{equation*}
Therefore, to show the desired statement, it is enough to prove
\begin{equation}
\label{eq:surprise}
\limsup_{\delta \to 0}\,  \limsup_{t \to + \infty}\,  (\log t)^{-1} \, \int_0^{\delta t} \left[ g(t \theta_1 (v/t)) -g(J(1,1)v) \right] dv \ =\ 0.
\end{equation}
By Lemma \ref{lem:thetad}, for $v \in [0,\delta t]$, we have
$$
| t \theta_1 (v/t) -J(1,1) v| \ \le \ r (\delta) J(1,1) v
$$
where $\lim_{\delta\downarrow 0}r (\delta)-0$ uniformly in $t$. On the other hand, there exists a constant $C_0>0$ such that $|g' (x)| \le C_0/ (x^2+1)$ for $x\geq 0$. Consequently, for $\delta$ small so that $r (\delta) <1$, we have
\begin{equation*}
\left| \int_0^{\delta t} \left[ g(t \theta_1 (v/t)) -g(J(1,1)v) \right] dv \right| \ \le\  C_0 J(1,1)  r (\delta) \int_0^{\delta t} \frac{v}{1+ [(1-r (\delta)) J(1,1)]^2 v^2} dv.
\end{equation*}
Finally, sending $\delta \to 0$, the right-side vanishes and we get (\ref{eq:surprise}).

When $\alpha = 2$, using the substitution $u = \beta_t v$ with $t\beta^2_t|\log\beta_t| = 1$ and $\beta_t = O((t\log t)^{1/2})$, a similar method yields the result.
\end{proof}

\begin{lemma} \label{lemma on hat a}

In $d=1$, we have
\begin{equation}\label{a-hat}
{\hat a} (u) \ = \ i c (b_1^+ -b_1^-) \sum_{y=1}^{\infty}\cfrac{\sin(2\pi u y)}{y^{1+\alpha}}.
\end{equation}

\noindent When $\alpha>1$, let $\xi(\alpha) - \sum_{y=1}^\infty \cfrac{1}{y^\alpha}$.  As $u\downarrow 0$,
\begin{eqnarray*}
{\hat a} (u) &\sim & 2 \pi i c (b_1^+ -b_1^-) \xi(\alpha)\, u \\
\sup_{s\in{\bb T}}\Big\{|\hat a(s)+\hat a(u-s)|^2\Big\}&\preccurlyeq &{\sin^2(\pi u)}.
\end{eqnarray*}

\noindent When $\alpha=1$, as $u\downarrow 0$,
\begin{eqnarray*}
{\hat a} (u) &\sim & - 2 \pi i c (b_1^+ -b_1^-) \, u \log (u)\\
\sup_{s\in{\bb T}}\Big\{|\hat a(s)+\hat a(u-s)|^2\Big\}&\preccurlyeq & - {\sin^2(\pi u)} \log^2(u).
\end{eqnarray*}

In $d=2$, for $\alpha>1$ and $w \in{\mc C}_2$, we have, as $u\to w$,
\begin{eqnarray*}
{\hat a} (u) &\sim & 2\pi i (u-w) \cdot m.
\end{eqnarray*}
Also, for $\delta>0$ small, there exists $c(\delta)>0$ such that when $|u-w| \le \delta$, we have
\begin{equation*}
\sup_{s \in \TT^2} \left\{ |{\hat a} (u) +{\hat a} (u-s)|^2 \right\} \ \le \ c(\delta) |u-w|^2.
\end{equation*}
\end{lemma}

\begin{proof}
We prove the statements in $d=1$, the two dimensional case being similar. To show the first claim \eqref{a-hat}, we notice, for $y\in{\mathbb{Z}}$, that
\begin{equation*}
a(y)\ =\ c(b^+_1-b_1^-)\frac{1}{2|y|^{1+\alpha}}(1-2\textbf{1}_{y<0}),
\end{equation*}
so that
\begin{equation*}
{\hat a} (u) \ =\  \sum_{y\in{\mathbb{Z}}}e^{2\pi i uy}a(y)\ =\ i c (b_1^+ -b_1^-) \sum_{y=1}^{\infty}\cfrac{\sin(2\pi u y)}{y^{1+\alpha}}.
\end{equation*}

When $\alpha>1$, since the function $u\mapsto \sin(2\pi y u)/(2\pi y u) \to 1$ as $u\downarrow 0$ pointwise and is uniformly bounded in $y\geq 1$, we have
\begin{equation*}
\frac{{\hat a} (u)}{u} \ =\ 2 \pi  i c (b_1^+ - b_1^-)   \,  \sum_{y \ge 1} \frac{1} {y^{\alpha}} \cfrac{\sin(2 \pi y u)}{2\pi y u } \ \to \ 2\pi i c(b^+_1 - b^-_1)\xi(\alpha),
\end{equation*}
by bounded convergence, proving the second claim.

For the third claim, write
\begin{eqnarray*}
\frac{{\hat a} (s) +{\hat a} (u-s)}{\sin (\pi u)}  &=& ic (b_1^+ -b_1^-) \sum_{y=1}^{\infty} \cfrac{\sin(2\pi s y)+ \sin(2 \pi y (u-s))}{y^{1+\alpha}\sin(\pi u)}\\
& = &2i c (b_1^+ -b_1^-) \sum_{y=1}^{\infty} \cfrac{1}{y^{\alpha}} \, \cfrac{\sin(\pi u y)} {y\sin(\pi u)}\, \cos (\pi (u-2s) y),
\end{eqnarray*}
as
$\sin(2\pi s y)+ \sin(2 \pi y (u-s))=2\sin(\pi yu)\cos(\pi y(u-2s))$.
Note $\cos (\pi (u-2s) y)\leq 1$ and $|\sin(\pi y u)/ (y\sin (\pi u))|\leq 1$ uniformly in $y \ge 1$ and $u\in (0,1)$.  Hence, as $u\downarrow 0$,
\begin{equation}
\label{eq:bounda}
\sup_{s \in \TT}\Big\{ |{\hat a} (s) + {\hat a}(s-u)|^2\Big\} \ \preccurlyeq \ \sin^2 (\pi u).
\end{equation}

When $\alpha=1$, for fixed $\varepsilon >0$ small, we have
\begin{eqnarray*}
{\hat a} (u) & =& i c (b_1^+ -b_1^-) \sum_{y=1}^{\infty}\cfrac{\sin(2\pi u y)}{y^{2}}\\
&= &2\pi i c (b_1^+ -b_1^-) u \sum_{y=1}^{\lfloor\varepsilon/u\rfloor}\cfrac{1}{y} \\
&&\ + i c (b_1^+ -b_1^-) \sum_{y=1}^{\lfloor\varepsilon/u\rfloor}\cfrac{[\sin(2\pi u y) -2\pi u y]}{y^{2}}+ i c (b_1^+ -b_1^-) \sum_{y=\lfloor\varepsilon/u\rfloor+1}^{\infty}\cfrac{\sin(2\pi u y)}{y^{2}}.
\end{eqnarray*}
Since there exists $C_{\ve} >0$ such that $|\sin (2\pi u y) -2 \pi u y| \le C_{\varepsilon} |u|^3 y^3$ for $1\leq y \leq \lfloor \varepsilon/u\rfloor$ and $|\sin(2\pi uy)| \le 1$, the second and third sums on the right-side are of order $O(u)$. The first sum is equivalent in order to $- 2\pi i c (b_1^+ -b_1^-) u \log (u)$, proving the fourth claim.

The fifth claim is proved similarly by decomposing in the equation,
\begin{equation*}
\frac{{\hat a} (s) +{\hat a} (u-s)}{\sin (\pi u)}   \ = \ 2i c (b_1^+ -b_1^-) \sum_{y=1}^{\infty} \cfrac{1}{y^{2}} \, \cfrac{\sin(\pi u y)} {\sin(\pi u)}\, \cos (\pi (u-2s) y),
\end{equation*}
the sum according to $y \le [\ve /u]$ and $y \ge [\ve/u] +1$ for a fixed $\ve$ small.
\end{proof}

\begin{lemma}
\label{lem:holderineq}
Let $\alpha \in (1,2]$ and
\begin{equation}
\varphi_{\alpha} (s) =\left\{\begin{array}{cl}
|s|^{\alpha}, & \mbox{if}\,\,\, 1<\alpha<2\,,\\
|s|^2 |\log( |s|) |, & \mbox{if} \,\,\,\alpha=2\,.
\end{array}
\right.
\end{equation}
For $0<\delta <1$ sufficiently small, there exists $C = C(\alpha, \delta)>0$ such that for $u,s \in [0,\delta]^2$,
$$\varphi_{\alpha} (u-s) + \varphi_{\alpha} (s) \ \ge \ C \left[ \varphi_{\alpha} (u) +\varphi_{\alpha} (s) \right].$$
\end{lemma}

\begin{proof}
We only prove the statement for $\alpha=2$, as the proof for $\alpha \in (1,2)$ is similar. Observe first that the restriction of $\varphi_2$ to $[-\delta,\delta]$, for $\delta$ small, is an even convex function. For $0<x<1$, we write
$$(1-x)u \ =\  x \left(\frac{1-x}{x} s \right) +(1-x) (u-s)$$ and invoke convexity of $\varphi_2$ to get
\begin{equation*}
\varphi_2 ((1-x) u) \ \le\  x \varphi_2 \left( \frac{1-x}{x} s \right)   + (1-x) \varphi_2 (u-s).
\end{equation*}
Then,
\begin{eqnarray*}
\varphi_2 (u-s) + \varphi_2 (s) & \ge &\varphi_2 (u) \left[ \frac{1}{1-x} \frac{\varphi_2 ((1-x)u)}{\varphi_2 (u)} \right] + \varphi_2 (s) \left[ 1- \frac{x}{1-x} \frac{\varphi_2 \left( \frac{1-x}{x} s\right)}{\varphi_2 (s)}\right].
\end{eqnarray*}

Since,
\begin{equation*}
 \frac{1}{1-x} \frac{\varphi_2 ((1-x)u)}{\varphi_2 (u)} \ \ge\  (1-x) \left| 1+ \frac{\log (1-x)}{\log \delta} \right|
\end{equation*}
and
\begin{equation*}
\left| \frac{x}{1-x} \frac{\varphi_2 \left( \frac{1-x}{x} s\right)}{\varphi_2 (s)} \right| \ \le\  \frac{1-x}{x} \left| 1 + \frac{\log(\frac{1-x}{x})}{\log \delta}\right|,
\end{equation*}
taking $x$ sufficiently close to $1$, the claim follows.
\end{proof}

\begin{lemma}\label{lemma-a2}
Let
\begin{equation}
\label{eq:Jalphadeltau}
J_{\alpha} (\lambda, \delta, u)\ := \ \int_0^{\delta}  \frac{ds}{\lambda+ \theta_1 (s) + \theta_1 (s-u)}.
\end{equation}
Then, for $\lambda>0$ and $0 < u < \delta$ small, there exist constants $C_0, C_1>0$ such that
\begin{equation*}
J_{\alpha} (\lambda, \delta, u) \ \le\ \left\{
\begin{array}{rl}
C_0 \log \left( 1+ \frac{C_1}{\lambda + u/C_1}\right)  & \text{ if } \alpha=1\\
C_0 (\lambda +u^{\alpha} /C_1)^{1/\alpha -1} & \text{ if } \alpha\in (1,2)\\
C_0 \Big\{ \left[\lambda+ C_1 |u^2 \log (u)| \right] \left| \log\left( \lambda+ C_1 |u^2 \log (u)| \right) \right|\Big\}^{-1/2}
 &\text{ if } \alpha=2.
\end{array}\right.
\end{equation*}

\end{lemma}

\begin{proof}
Suppose $\alpha=1$.  Since $u \in (0,\delta)$ with $\delta \ll 1$, with respect to a suitable positive constant $\kappa_0$, by Lemma \ref{lem:thetad}, we have
\begin{eqnarray*}
J_{\alpha} (\lambda, \delta, u)  &\le &\int_0^\delta  \frac{ds}{\lambda+ \kappa_0 |s| + \kappa_0  |s-u|}\\
&= &\int_0^u \frac{ds}{\lambda+ \kappa_0 u} + \int_{u}^{\delta} \cfrac{ds}{\lambda + \kappa_0 s + \kappa_0 (s-u)}\\
& =& \frac{u}{\lambda + \kappa_0 u} + \frac{1}{2 \kappa_0} \log \left( 1 + \frac{2 \kappa_0 (\delta -u)}{\lambda + \kappa_0 u }\right)\\
&\le& \kappa_0^{-1} + (2 \kappa_0 )^{-1} \log \left( 1 + \frac{2 \kappa_0 \delta}{ \lambda + \kappa_0 u}\right),
\end{eqnarray*}
finishing the claim in this case.

Suppose $1 < \alpha <2$.   By Lemma \ref{lem:thetad}, as $s, u \in (0,\delta)$ with $\delta \ll 1$, and Lemma \ref{lem:holderineq}, we have
\begin{eqnarray*}
J_{\alpha} (\lambda, \delta, u) &\le& \int_0^\delta  \frac{ds}{\lambda+ \kappa_0 |s|^{\alpha} + \kappa_0  |s-u|^{\alpha}}\\
&\leq&\int_0^\delta  \frac{ds}{\lambda+\kappa_1  |s|^{\alpha} + \kappa_1  |u|^{\alpha}},
\end{eqnarray*}
for a suitable constants $\kappa_0$ and $\kappa_1$.
  By the change of variables  $t= s/ (\lambda +\kappa_1 u^{\alpha})^{1/\alpha}$, the last integral is  equal to
\begin{equation*}
(\lambda +\kappa_1 u^\alpha)^{1/\alpha -1} \int_0^{\delta (\lambda +\kappa_1 u^{\alpha})^{-\alpha^{-1}} }\frac{dt}{1+\kappa_1 t^\alpha} \ = \ O((\lambda +\kappa_1 u^\alpha)^{1/\alpha -1}),
\end{equation*}
which shows the desired statement.

Suppose $\alpha =2$.   Similarly, by Lemma \ref{lem:thetad} and Lemma \ref{lem:holderineq}, we have
\begin{eqnarray*}
\int_0^{\delta}  \frac{ds}{\lambda+ \theta_1 (s) + \theta_1 (s-u)}  &\leq&\int_0^\delta  \frac{ds}{\lambda+\kappa_1  |s^2 \log (s) | + \kappa_1  |u^2 \log u|}\\
&=&
C_{\lambda}^{-1}  (u) \int_0 ^{\delta / C_{\lambda} (u)} \frac{ds}{1 +
s^2 |\log( s) + \log (C_{\lambda} (u))|},
\end{eqnarray*}
for a positive constant $\kappa_1$ and
$C_{\lambda} (u) := \sqrt{\lambda + \kappa_1 |u^2 \log( u)|}$.

For $\lambda$ and $\delta$ small, $C_{\lambda} (u) <1$. Fix $0<\ve<1$.
 We split the last integral as follows:
\begin{eqnarray}
\label{eq:rtyuiop}
&&\int_0 ^{\delta / C_{\lambda} (u)} \frac{ds}{1 + s^2 |\log (s) + \log( C_{\lambda}
(u))|} \\
&&= \int_0 ^{\delta / C_{\lambda} (u)^\ve} \frac{ds}{1 + s^2 |\log (s) + \log( C_{\lambda} (u))|}
 +  \int_{\delta/C_{\lambda} (u)^\ve}^{\delta / C_{\lambda} (u)} \frac{ds}{1 + s^2 |\log( s) + \log( C_{\lambda} (u))|}. \nonumber
\end{eqnarray}

We claim
the first integral on the right-side of \eqref{eq:rtyuiop} is of order $O(|\log( C_{\lambda} (u))|^{-1/2})$:  Indeed,
for $s \in (0, \delta/C_{\lambda} (u)^\ve)$,
$$|\log (s) + \log( C_{\lambda} (u)) | \ \ge\  | \log (\delta) + (1-\ve) \log (C_{\lambda} (u))|$$
so that
\begin{equation*}
 \int_0 ^{\delta / C_{\lambda} (u)^\ve} \frac{ds}{1 + s^2 |\log (s) + \log (C_{\lambda} (u))|}
\ \le \ \frac{1}{|\log (\delta) +(1-\ve) \log (C_{\lambda} (u))|^{1/2}} \int_0^{\infty} \frac{dv}{1+v^2}.
\end{equation*}

On the other hand, the second integral on the right-side of (\ref{eq:rtyuiop}) is order $O(1)$:  Indeed,
this integral is bounded above by
$$\int_{\delta/C_{\lambda}(u)^\ve}^{\delta / C_{\lambda} (u)} \frac{ds}{1 + s^2 |\log \delta|} \ =\  O(C_\lambda(u)^\ve) \ = \ O(1),$$
finishing the proof. \end{proof}

\section*{Acknowledgements}

The research of CB was supported in part by the French Ministry of Education through the grant ANR-10-BLAN 0108 (SHEPI) .

PG thanks FCT (Portugal) for support through the research
project  PTDC/MAT/109844/2009 and CNPq (Brazil) for support through the research project  480431/2013-2. PG thanks  CMAT for support by ``FEDER" through the
``Programa Operacional Factores de Competitividade  COMPETE" and by
FCT through the project PEst-C/MAT/UI0013/2011.

SS was supported in part by ARO grant W911NF-14-1-0179.

\end{document}